\documentclass[lineno]{biometrika}

\usepackage{amsmath,amssymb,graphicx}

\usepackage{xcolor}

\usepackage{soul}

\usepackage{chngcntr}

\usepackage{lineno}
\nolinenumbers

\usepackage{xr}

\usepackage[font=small,justification=centering]{caption} 
\usepackage{times}
\usepackage{bm}
\usepackage{natbib}
\usepackage{graphicx}

\usepackage{chngcntr}
\usepackage{apptools}
\AtAppendix{\counterwithin{lemma}{section}}
\AtAppendix{\counterwithin{theorem}{section}}
\AtAppendix{\counterwithin{proposition}{section}}
\AtAppendix{\counterwithin{condition}{section}}
\AtAppendix{\counterwithin{equation}{section}}
\AtAppendix{\counterwithin{figure}{section}}

\usepackage[plain,noend]{algorithm2e}

\makeatletter
\renewcommand{\algocf@captiontext}[2]{#1\algocf@typo. \AlCapFnt{}#2} 
\def\@algocf@capt@plain{top}
\renewcommand{\algocf@makecaption}[2]{%
  \addtolength{\hsize}{\algomargin}%
  \sbox\@tempboxa{\algocf@captiontext{#1}{#2}}%
  \ifdim\wd\@tempboxa >\hsize
    \hskip .5\algomargin%
    \parbox[t]{\hsize}{\algocf@captiontext{#1}{#2}}
  \else%
    \global\@minipagefalse%
    \hbox to\hsize{\box\@tempboxa}
  \fi%
  \addtolength{\hsize}{-\algomargin}%
}
\makeatother

\def\T{{ \mathrm{\scriptscriptstyle T} }}

\providecommand{\cov}{\mathrm{ cov}}

\providecommand{\MDI}{\mathrm{\sc MDI}}

\providecommand{\E}{\mathrm{ E}}
\providecommand{\diag}{\mathrm{diag}}

\newcommand{\argmax}{\operatornamewithlimits{argmax}}

\begin{document}



\markboth{F. Bachoc, M.~G. Genton, K. Nordhausen, A. Ruiz-Gazen \and J. Virta}{Spatial blind source separation}

\title{Spatial blind source separation}

\author{Fran\c{c}ois BACHOC}
\affil{{Institut de Math\'ematiques de Toulouse, Universit\'e Paul Sabatier}, \\
118 route de Narbonne, 31062 Toulouse, France
\email{francois.bachoc@math.univ-toulouse.fr}}

\author{Marc G. GENTON}
\affil{Statistics Program, King Abdullah University of Science and Technology,\\
Thuwal 23955-6900, Saudi Arabia \email{marc.genton@kaust.edu.sa}}

\author{Klaus NORDHAUSEN}
\affil{CSTAT, Vienna University of Technology, \\Wiedner Hauptstr. 7, A-1040 Vienna, Austria
 \email{klaus.nordhausen@tuwien.ac.at}}

\author{Anne RUIZ-GAZEN}
\affil{Toulouse School of Economics, University {of} Toulouse Capitole, \\
21 all\'ee de Brienne, 31000 Toulouse, France  \email{anne.ruiz-gazen@tse-fr.eu}}

\author{\and Joni VIRTA}
\affil{{Department of Mathematics and Statistics, University of Turku, \\
		20014 Turun yliopisto, Finland, \\
		Department of Mathematics and Systems Analysis, Aalto University, \\
		PL 11000, 00076 AALTO, Finland.} \email{joni.virta@utu.fi}}

\maketitle

\begin{abstract}
Recently a blind source separation model was suggested for 
spatial data together with an estimator based on the simultaneous 
diagonalisation of two scatter matrices. The asymptotic properties of this estimator are derived here and a new estimator, based on the joint diagonalisation of more than two scatter matrices, is proposed. 
The {asymptotic} properties and merits of the novel estimator are verified in simulation studies. A real data example illustrates the method.
\end{abstract}

\begin{keywords}
{J}oint diagonalisation; {L}imiting distribution; 
{M}ultivariate random field; {S}patial scatter matrix.
\end{keywords}

\section{Introduction}

\thispagestyle{empty}

There is an abundance of multivariate data measured at spatial locations $ s_1,\ldots, s_n$ in a domain $\mathcal{S}^d \subseteq \mathbb{R}^d$. Such data {exhibit} two kinds of dependence{:} measurements taken closer to each other tend to be more similar than measurements taken further apart, and the variable values within a single location are likely to be correlated.

This complexity makes modelling multivariate spatial data computationally and theoretically difficult due to the large number of parameters required to represent the dependencies. In this work we address this problem through blind source separation, a framework established as independent component analysis for independent and identically distributed data and for stationary and non-stationary time series; see \citet{comon2010handbook} and  \citet{NordhausenOja2018}. Denoting a $p$-variate  random field as $ X( s)=\{X_{1}( s),\ldots,X_{p}( s)\}^\T$, where $^\T$ is the transpose operator, we assume that $ X(s) $ obeys the spatial blind source separation model introduced in \cite{NordhausenOjaFilzmoserReimann2015}. That is, $ X( s)$ at a location $ s$ is a linear mixture of an underlying $ p $-variate latent field $ Z( s) = \{Z_{1}( s),\ldots, Z_{p}(s)\}^\T$ with independent components,
\begin{equation} \label{eq:form:X:in:spatial:BSS}
X( s) =  \Omega  Z( s),
\end{equation}
where $ \Omega$ is an unknown $p \times p$  full rank matrix. {In this introduction section, we consider that the random fields $X$ and $Z$ have mean functions zero, for the sake of simplicity.}

{
When the observed random field $X$ takes the form \eqref{eq:form:X:in:spatial:BSS}, modeling and computational simplifications can be obtained.
Indeed, if no assumption at all is made on $X$, then the distribution of $X$ is characterized by $p$ covariance functions and {by} $p(p-1)/2$ cross-covariance functions. In contrast, when it is assumed that $X$ takes the form \eqref{eq:form:X:in:spatial:BSS}, then the distribution of $X$ is characterized by {$p$ covariance functions and by a $p \times p$ matrix.} As a function is an infinite-dimensional object, it is  more difficult to model and estimate than a fixed-dimensional matrix. Thus, when the observed random field $X$ takes the form \eqref{eq:form:X:in:spatial:BSS}, modeling simplifications are available.
}

{
When no assumption is made on $X$, a common practice in geostatistics is to let each of the $p$ covariance functions and each of the $p(p-1)/2$ cross-covariance functions of $X$ be characterized by $q$ parameters. For instance the case $q=2$ can correspond to a variance and a length scale parameter for an isotropic function. Then, the resulting $qp(p+1)/2$ parameters are usually estimated jointly by optimizing a fit criterion, typically the likelihood \citep{GentonKleiber2015}. This requires to perform an optimization in dimension $qp(p+1)/2$, where the computational cost of an evaluation of the likelihood is $O(p^3n^3)$. Once the $qp(p+1)/2$ parameters are estimated, the prediction of $X(s)$ for new values of $s$ can be performed at the computational cost  $O(p^3n^3)$.}

{
 In contrast, consider that {model} \eqref{eq:form:X:in:spatial:BSS} holds for $X$. We will show in this paper that an estimate of $\Omega^{-1}$ can be obtained. This is carried out by, first, computing scatter matrices with computational cost $O(p^2 n^2)$ and, second, performing an optimization in dimension $p^2$ where the computational cost of the function to be evaluated is $O(p^2)$, see {\S}~\ref{BSSest2} for details.  If each covariance function of $Z$ is characterized by $q$ parameters, each of them can be estimated separately, by optimizing the likelihood in dimension $q$. The evaluation cost of the likelihood is $O(n^3)$. Once the $qp$ covariance parameters are estimated, the prediction of $X(s)$ for new values of $s$ can be performed at cost $O( p n^3 )$. Indeed, the predictions of $Z_1(s),\ldots,Z_p(s)$ can be performed separately at cost $O(n^3)$ and  aggregated with negligible cost. 
 }

{Not all random fields $X$ obey a spatial blind source separation model of the form \eqref{eq:form:X:in:spatial:BSS}. For instance, \eqref{eq:form:X:in:spatial:BSS} forces the cross-covariance functions of $X$ to be symmetric. Nevertheless, {it} is a reasonable assumption in a fair number of practical situations \citep{NordhausenOjaFilzmoserReimann2015} and brings the computational benefits discussed above. Furthermore, an additional benefit of the form \eqref{eq:form:X:in:spatial:BSS} is dimension reduction. In blind source separation, often significantly fewer than the full $ p $ latent components are needed to capture the essential structure of the original observations and the remaining components can be discarded as noise.}

{We thus consider the spatial blind source separation model \eqref{eq:form:X:in:spatial:BSS} in this paper and focus on the estimation of $ \Omega^{-1} $. As discussed above, this estimation {} enables to estimate the cross-covariance functions of $X$ and to perform prediction.} Our approach for estimating $ \Omega^{-1} $  is based on the use of local covariance, or scatter, matrices, 
\begin{equation} \label{eq:hat:M:f}
\widehat{M}(f)=  n^{-1} \sum_{i=1}^n \sum_{j=1}^n f( s_i -  s_j)  X( s_i)  X( s_j)^\T,
\end{equation}
where $f: \mathbb{R}^d \to \mathbb{R}$ is called the kernel function. \citet{NordhausenOjaFilzmoserReimann2015} obtained estimators $\widehat{ \Gamma}(f)$ of $ \Omega^{-1} $ through a generalized eigendecomposition of pairs of local covariance matrices with kernels of the form $ (f_0, f_h) $, with $f_h( s_i -  s_j)=I(\| s_i -  s_j \| {\leq} h)$, for a positive constant $h$,  where $ I(\cdot)$ denotes the indicator function and $f_0( s)=I( s =0)$. Their estimators were based on the following definition, with $ f = f_h $ for some $ h > 0 $.

\begin{definition} \label{sBSSunmixest}
	An unmixing matrix estimator
	$\widehat{ \Gamma}(f)$ jointly diagonalizes $\widehat{M}(f_0)$ and $\widehat{M}(f)$
	in the following way
	\[
	\widehat{ \Gamma}(f) \widehat{M}(f_0) \widehat{ \Gamma}(f)^\T =  I_p \quad \mbox{and} \quad \widehat{ \Gamma}(f)\widehat{M}(f)
	\widehat{ \Gamma}(f)^\T =\widehat{ \Lambda}(f),
	\]
	where $\widehat{  \Lambda}(f)$ is a diagonal matrix with diagonal elements in decreasing order.
\end{definition}

This method is conceptually close to principal component analysis where latent variables that have maximal variance are found through the diagonalisation of the covariance matrix. However, since the covariance matrix does not capture spatial information, it was  extended to the concept of a local covariance matrix in \citet{NordhausenOjaFilzmoserReimann2015}. Analogously, diagonalising local covariance matrices then aims to find latent fields that maximize spatial correlation. 

Here, we expand on their work by not restricting the kernel $ f $ in Definition~\ref{sBSSunmixest} to be of the ``ball'' form $ f_h $. Furthermore, we derive the asymptotic behavior for the method proposed in \citet{NordhausenOjaFilzmoserReimann2015} for a large class of kernel functions $ f $.

{The idea when constructing these kernel functions is that {the mean values of $\widehat{M}(f)$ and $\widehat{M}(f_0)$ would be diagonal matrices if, in their definition, the mixed components $X$ were replaced by the latent components $Z$}. {Hence, a general blind source separation strategy is to undo the mixing in $X$ by finding a matrix $\widehat{\Gamma}(f)$ which simultaneously diagonalizes $\widehat{M}(f)$ and $\widehat{M}(f_0)$.} This is computationaly simple and can always be done exactly using generalized eigenvalue-eigenvector theory. From temporal blind source separation, it is however well known that when diagonalising only two matrices, the choice of the matrices can have a large impact on the separation efficiency. Therefore, it is a popular strategy to approximately diagonalize more than two matrices with the hope of including more information; see for example \cite{BelouchraniAbedMeraimCardosoMoulines:1997}, \cite{nordhausen2014robustifying}, \cite{MiettinenNordhausenOjaTaskinen:2014b}, \cite{matilainen2015new} and \cite{MiettinenIllnerNordhausenOjaTaskinenTheis2016}.
	Approximate diagonalization becomes then necessary as the matrices commute only at the population level but not when estimated using finite data. There are  many algorithms available for this purpose.  We use this idea to extend the method of \citet{NordhausenOjaFilzmoserReimann2015} to jointly diagonalize more than two local covariance matrices. 
	We also derive the asymptotic behaviour of these novel estimators.}

\section{Spatial blind source separation model} \label{BSSmodel}

\subsection{General assumptions}

In the spatial blind source separation model, the following assumptions are made:
\begin{assumption} \label{assumption:mean:zero:Z}
 $\E\{Z( s)\} =  0$ for $ s \in \mathcal{S}^d$;
 \end{assumption}
 \begin{assumption} \label{assumption:cov:IP:_Z}
 $\cov\{Z( s)\} =\E\{ Z( s) Z( s)^\T\}= I_p$;
 \end{assumption}
 \begin{assumption} \label{assumption:cov:D:Z}
  $\cov\{ Z( s_1),  Z( s_2)\}=\E\{ Z( s_1) Z( s_2)^\T\}= D( s_1, s_2)$, where $ D$ is a diagonal matrix whose diagonal elements depend only on $ s_1 -  s_2$. 
  \end{assumption}
Let  $\cov\{Z_k( s_i),Z_k( s_j)\}=
K_k( s_i- s_j)= D( s_i, s_j)_{k,k}$,
where $K_k$ denotes the stationary covariance function of $Z_k$, for  $k=1,\ldots,p$.

Assumption \ref{assumption:mean:zero:Z} is made for convenience and can easily be replaced by assuming a constant unknown mean (see Lemma \ref{lem:mean} in the supplementary material). Assumption \ref{assumption:cov:IP:_Z} says that the components of $ Z ( s)$ are uncorrelated and implies that the variances of the components are one, which reduces identifiability issues and comes without loss of generality.
Assumption~\ref{assumption:cov:D:Z} says that there is also no spatial cross-dependence between the components.
However, even after these assumptions are made, the model is not uniquely defined. The order of the latent fields and also their signs can be changed. This is common for all blind source separation approaches and is not considered a problem in practice.

\subsection{Identifiability}
\label{section:identifiability}

The expectations of $\widehat{M}(f)$ and $\widehat{M}(f_0)$ are {respectively}
\begin{equation}\label{eq:local_cov_population}
M(f)=  n^{-1}\sum_{i=1}^n \sum_{j=1}^n f( s_i - s_j) \E\{ X( s_i) X( s_j)^\T\}
\; \mbox{ and } \;
M(f_0)= n^{-1} \sum_{i=1}^n  \E\{ X( s_i) X( s_i)^\T\}.
\end{equation}
{Thus the empirical procedure of Definition \ref{sBSSunmixest}, operating on $\widehat{M}(f)$ and $\widehat{M}(f_0)$, can be associated to the following theoretical procedure, operating on $M(f)$ and $M(f_0)$.
}

\begin{definition} \label{sBSSunmix}
	For any function $f: \mathbb{R}^d \to \mathbb{R}$, an unmixing matrix functional
	$ \Gamma(f)$ is defined as a functional which jointly diagonalizes $M(f)$ and $M(f_0)$
	in the following way
	\[
	\Gamma(f)M(f_0) \Gamma(f)^\T =  I_p \quad \mbox{and} \quad  \Gamma(f) M(f) \Gamma(f)^\T =  \Lambda(f),
	\]
	where $ \Lambda(f)$ is a diagonal matrix with diagonal elements in decreasing order.
\end{definition}

We remark that an unmixing matrix $ \Gamma(f)$ can be found using the generalized eigenvalue-eigenvector theory. In addition, an unmixing matrix is never unique, since if $\Gamma(f)$ and $\Lambda(f)$ satisfy Definition \ref{sBSSunmix}, then $ S \Gamma(f)$ and $\Lambda(f)$ also satisfy Definition \ref{sBSSunmix} for any diagonal matrix $S$ with diagonal elements equal to $-1$ or $1$.
{We also remark that $\Lambda(f)$ is not the expectation of $\widehat{ \Lambda}(f)$, in general. Indeed, Definitions \ref{sBSSunmixest} and \ref{sBSSunmix} are {based on} non-linear functions of $\{\widehat{M}(f),\widehat{M}(f_0)\}$ and of $\{M(f),M(f_0)\}$.}

The usual notion of identifiability in blind source separation is that any unmixing functional $\Gamma(f)$ should recover the components of $Z$ up to signs and order of the components. Thus, any unmixing functional $\Gamma(f)$ should coincide with $\Omega^{-1}$, up to the order and signs of the rows.

\begin{definition} \label{def:identifiability:two:matrices}
We say that the unmixing problem given by $f$ is identifiable if any unmixing functional $\Gamma(f)$ satisfying Definition \ref{sBSSunmix} can be written as
$P S \Omega^{-1}$, where $P$ is a permutation matrix and $S$ is a diagonal matrix with diagonal elements equal to $-1$ or $1$.
\end{definition}

The motivation behind identifiability is that, if identifiability holds, then estimating $M(f_0)$ and $M(f)$ consistently by  $\widehat{M}(f_0)$ and $\widehat{M}(f)$ enables to obtain $\widehat{\Gamma}(f)$, which will be approximately equal to a matrix of the form $P S \Omega^{-1}$, with $P$ and $S$ as in Definition \ref{def:identifiability:two:matrices}.
 The following proposition provides a necessary and sufficient condition for identifiability. This proposition is proved in {{\S}~\ref{supplement:proof:identifiability:one}} of the supplementary material. All the other theoretical results in this paper are also proved in the supplementary material. Let $M^{-\T}$ denote the inverse of the transpose of $M$.

\begin{proposition} \label{prop:identifiability}
The unmixing problem given by $f$ is identifiable if and only if the diagonal elements of $\Omega^{-1} M(f) \Omega^{-\T}$ are {distinct}.
\end{proposition}

We remark that identifiability is a {joint} property of the kernel $f$ and the covariance functions $K_1,\ldots ,K_p$. For instance, consider the situation where $K_1,\ldots,K_p$ are compactly supported and equal to zero at distances larger than $0< r < \infty$, and where one uses the function $f( s)= I(r_1 <\| s\|\leq r_2)$, with $r \leq r_1 < r_2 < \infty$ as kernel. Then identifiability does not hold because $\Omega^{-1} M(f) \Omega^{-\T}$ is equal to the zero matrix. On the other hand, if $f$ is a ball kernel of the form $f( s)=I(\| s\|\leq r_0)$ with $r_0 >0$, then identifiability may hold, for the same covariance functions  $K_1,\ldots,K_p$.

Finally, for any kernel $f$, a necessary condition for identifiability is that there does not exist $k,l \in \{1,\ldots,p\}${, $ k \neq l $,} such that $K_k(s_i-s_j) = K_l(s_i-s_j)$ for all $i,j = 1, \ldots , n$. Indeed, if this was the case, then the diagonal elements $k$ and $l$ of $\Omega^{-1} M(f) \Omega^{-\T}$ would be equal, for any kernel $f$. 
An extreme example of this issue is $K_1 = \cdots = K_p$ with only Gaussian components. If this is the case, then, for any orthogonal matrix $Q$, the distribution of the random field $Q Z$ is the same as that of the random field $Z$. Hence, no statistical procedure can be expected to recover the components of $Z$, even up to signs and permutations, when only observing the transformed random field $X$.


\subsection{Relationships with other models of multivariate random fields}
\label{section:relationships}

The spatial blind source separation is notably different from the usual multivariate models for spatial data, which are often defined starting with their covariance functions contained in a cross-covariance matrix,
	\[
	C( s_1,  s_2) = \cov\{X( s_1), X( s_2)\} := \{C_{k,l}( s_1, s_2)\}_{k,l=1}^p,
	\]
	whereas our  approach for estimating $ \Omega^{-1} $ does not need to model or estimate the covariance functions of the latent fields $ Z_{1}( s),\ldots, Z_{p}(s) $.

	In a recent extensive review, \citet{GentonKleiber2015} discussed different approaches to define cross-covariance matrix functionals and gave a list of properties and conventions that they should satisfy, for instance stationarity and invariance under rotation. 
As \citet{GentonKleiber2015} pointed out, to create general classes of models with well-defined cross-covariance functionals is a major challenge.
Multivariate spatial models are particularly challenging as  many parameters need to be fitted. In textbooks such as \citet{Wackernagel2010} usually the following two popular models are described.

In the intrinsic correlation model it is assumed that the stationary covariance matrix $C( h)$ can be written as the product of the variable covariances and the spatial correlations, $C( h) = \rho( h)  T$, for all lags $ h$, where $ T$ is a non-negative definite $p \times p$ matrix and $\rho( h)$ a univariate spatial correlation function.

The more popular linear model of coregionalization is a generalization of the intrinsic correlation model,
and the covariance matrix then has the form
\[
C( h) = \sum_{k=1}^r \rho_{k}( h)  T_{k},
\]
for some positive integer $r\leq p$ with all the $\rho_{k}$'s being univariate spatial correlation functions and
$ T_{k}$'s being non-negative definite $p \times p$ matrices, often called coregionalization matrices.
Hence, with $r=1$ this reduces to the intrinsic correlation model.
The linear model of coregionalization implies a symmetric cross-covariance matrix.

Estimation in the linear model of coregionalization is discussed in several papers. \citet{GoulardVoltz1992} focused on the coregionalization matrices
using an iterative algorithm where the spatial correlation functions are assumed to be known. The algorithm was extended in \citet{Emery2010}. Assuming Gaussian random fields, an expectation-maximisation algorithm was suggested in \citet{Zhang2007} and a Bayesian approach was considered in \citet{Gelfand2004}.

There is a simple connection between the spatial blind source separation model and the linear model of coregionalization. 
The covariance matrix $C_{X}(h)$ resulting from a spatial blind source separation model is always symmetric and can be written as
\[
C_{X}(h) = \sum_{{k}=1}^p K_{{k}}(h) T_{k},
\]
with $T_{k} = \omega_{k} \omega_{k}^\T$, $\omega_{k}$ being the ${k}$th column of $\Omega$. Thus the spatial blind source separation model is a special case of the linear model of
coregionalization with $r = p$ and where all coregionalization matrices $T_{k}$,
${k}=1,\ldots,p$, are rank one matrices. \\

\section{Asymptotic properties for simultaneous diagonalisation of two matrices} \label{BSSest1}

Recall the definition \eqref{eq:hat:M:f} of a local covariance matrix and that 
\begin{equation} \label{eq:hat:M:fzero}
\widehat{M}(f_0)=  n^{-1} \sum_{i=1}^n   X( s_i)  X( s_i)^\T
\end{equation}
{is the covariance estimator. 
Asymptotic results can be derived for the previous estimators assuming that Assumptions 1 to 3 hold together with the following assumptions:}
\begin{assumption}
The coordinates $Z_{1},\ldots,Z_{p}$ of $ Z$ are stationary Gaussian processes on $\mathbb{R}^d$;
\end{assumption}
\begin{assumption} \label{assumption:expansion}
  A fixed  $\Delta>0$ exists so that, for all $n \in \mathbb{N}$ and, for all $i \neq j$, $i,j=1,\ldots,n$, $\| s_i -  s_j\| \geq \Delta$;
  \end{assumption}
\begin{assumption} \label{assumption:Kk}
  Fixed $A >0$ and $\alpha >0$ exist such that, for all $x \in \mathbb{R}^d$ and, for all $k=1,\ldots,p$,
\[
| K_k( x) | \leq \frac{ A }{ 1 + \| x\|^{d + \alpha} };
\]
\end{assumption}
\begin{assumption} \label{assumption:f}
Assuming Assumption \ref{assumption:Kk} holds, then for the same $A >0$ and $\alpha >0$ we have
\[
| f( x) | \leq \frac{ A }{ 1 + \| x\|^{d + \alpha} };
\]
\end{assumption}


\begin{assumption} \label{assumption:idendifiability:two:matrices:asymptotic}
We have
\[
\liminf_{n \to \infty}
\min_{i=2,\ldots,p} \left[
  \left\{  \Omega^{-1} M(f) \Omega^{-\T} \right\}_{i,i} 
  -
   \left\{  \Omega^{-1} M(f) \Omega^{-\T} \right\}_{i-1,i-1} 
    \right]
> 0.
\]
\end{assumption}

Assumption \ref{assumption:expansion} implies that $\mathcal{S}^d$ is unbounded as $n \to \infty$, which means that we address the increasing domain asymptotic framework \citep{cressie91statistics}.

Assumption \ref{assumption:f} holds in particular for the function $I( s=0)$ and for the ``ball'' and ``ring'' kernels $B(h)(s)= I(\| s\|\leq h)$ with fixed $h\geq 0$ and $R(h_1,h_2)(s)= I(h_1\leq \| s\|\leq h_2)$ with fixed $h_2 \geq h_1 \geq 0$. 

Up to reordering the components of $Z$, which comes without loss of generality,
Assumption~\ref{assumption:idendifiability:two:matrices:asymptotic}
is an asymptotic version of the identifiability condition in Proposition \ref{prop:identifiability}. Under Assumption \ref{assumption:idendifiability:two:matrices:asymptotic}, identifiability in the sense of Definition \ref{def:identifiability:two:matrices} holds for sufficiently large $n$, from Proposition \ref{prop:identifiability}.

{Proposition \ref{prop_cov_consistency} below gives  the consistency of the estimator $\widehat{M}(f)$, where $f$ satisfies Assumption \ref{assumption:f}. The proof of this proposition is provided in {\S}~\ref{appendix:two:matrices} of the supplementary material.}

\begin{proposition}\label{prop_cov_consistency}
{Suppose $n \to \infty$ and Assumptions \ref{assumption:mean:zero:Z} to \ref{assumption:Kk} hold and let $f: \mathbb{R}^d \to \mathbb{R}$ satisfy Assumption \ref{assumption:f}. Then $\widehat{M}(f)-M(f)\to 0$ in probability when $n \to \infty$.}
\end{proposition}

{We remark that $M(f)$ depends on $n$ and that we do not assume that the sequence of matrices $M(f)$ converges to a fixed matrix as $n \to \infty$. Hence, Proposition \ref{prop_cov_consistency} shows that $\widehat{M}(f)-M(f)$ converges to zero, and not that $\widehat{M}(f)$ converges to $M(f)$.}

{
Next, we show the joint asymptotic normality of $n^{1/2} \{\widehat{M}(f_0)-M(f_0)\}$ and $n^{1/2}\{\widehat{M}(f)-M(f)\}$, seen as sequences of $p^2 \times 1$ random vectors. 
Similarly as in Proposition~\ref{prop_cov_consistency}, we do not need to assume that the sequence of $2p^2 \times 2p^2$ covariance matrices of these two sequences of vectors converges to a fixed matrix. Hence, we will not show that these sequences of random vectors converge jointly to a fixed Gaussian distribution. Instead, we will show that the distances between the distributions of these random vectors and Gaussian distributions converge to zero as $n \to \infty$. As a distance between distributions, we consider a metric $d_w$  generating the topology of weak convergence on the set of Borel probability measures on Euclidean spaces (see, e.g., \cite{dudleyreal}, p. 393). The benefit of such a distance is that a sequence of distributions $(\mathcal{L}_n)_{n \in \mathbb{N}}$ converges to a fixed distribution $\mathcal{L}$ if and only if $d_w( \mathcal{L}_n , \mathcal{L} )$ converges to zero.
The next proposition provides the asymptotic normality result. It is proved in {\S}~\ref{appendix:two:matrices} of the supplementary material.} 

\begin{proposition}\label{prop_cov_asymptotic_normality}
{Assume the same assumptions as in Proposition \ref{prop_cov_consistency}. 
Let  $W(f)$ be the vector of size $p^2 \times 1$, defined  for $i = (a-1)p + b$, $a,b \in \{1,\ldots,p\}$, by 
\[
W(f)_i = {n}^{1/2}\{\widehat{M}(f)_{a,b} - M(f)_{a,b}\}.
\]
Let $Q_n$ be the distribution of $\{W(f)^\T , W(f_0)^\T\}^\T$.
Then, as $n \to \infty$,
\[
d_w[Q_n, \mathcal{N}\{0,V(f,f_0) \}]\to 0,
\]
where $\mathcal{N}$ denotes the normal distribution and details concerning the matrix $V(f,f_0)$ are given in Appendix \ref{appendix:expression:V}.
Furthermore, the largest eigenvalue of $V(f,f_0)$  is bounded as $n \to \infty$.}
\end{proposition}

{
In Proposition \ref{prop_cov_asymptotic_normality},  $V(f,f_0)$ is a $2p^2 \times 2p^2$ matrix that depends on $n$ and is interpreted as an asymptotic covariance matrix. Also, in Proposition \ref{prop_cov_asymptotic_normality}, the vectors $W(f)$ and $W(f_0)$, that are asymptotically Gaussian, are obtained by row vectorization of $n^{1/2} \{\widehat{M}(f_0)-M(f_0)\}$ and $n^{1/2}\{\widehat{M}(f)-M(f)\}$.}
{Taking $f( s)= I(\| s\|\leq h)$ with $h>0$ in Propositions \ref{prop_cov_consistency} and \ref{prop_cov_asymptotic_normality}  gives the asymptotic properties of the method proposed in \citet{NordhausenOjaFilzmoserReimann2015}.}

\begin{remark}{
Propositions \ref{prop_cov_consistency} and \ref{prop_cov_asymptotic_normality} remain valid when centering the process $X$ by $\bar{X}=n^{-1} \sum_{i=1}^n X( s_i)$.} Indeed, we prove in Lemma~\ref{lem:mean} 
of the supplementary material that the difference between the centered  estimator and $\widehat{M}(f)$ is of order $O_p(n^{-1})$.
\end{remark}

{
For a matrix $A$  with rows $l_1^\T,\ldots,l_k^\T$, let $\mathrm{vect}( A) = (l_1^\T,\ldots,l_k^\T)^\T$ be the row vectorization of $ A $ and for a matrix $ A$ of size $k \times k$, let ${ \mathrm{diag}}( A) = ( A_{1,1},\ldots, A_{k,k})^\T $.}
{Next, Proposition \ref{prop_eigen} shows the joint asymptotic normality of the estimators $\widehat{ \Gamma}(f)$ and $\widehat{  \Lambda}(f)$.
This proposition is proved in {\S}~\ref{appendix:two:matrices} of the supplementary material.}

\begin{proposition} \label{prop_eigen}
{Assume the same assumptions as in Proposition \ref{prop_cov_consistency}.} 
Assume also that Assumption \ref{assumption:idendifiability:two:matrices:asymptotic} holds.
For $\widehat{ \Gamma}(f)$ and $\widehat{  \Lambda}(f)$ in Definition \ref{sBSSunmixest}, let $Q_n$ be the distribution of
\[
{n}^{1/2} \left(
\begin{array}{c}
\mathrm{vect} \left\{
\widehat{ \Gamma}(f)
-
\Omega^{-1}
\right\}
\\
\diag \left\{
\widehat{  \Lambda}(f)
-
{  \Lambda}(f)
\right\}
\end{array} \right).
\]
Then, we can choose $\widehat{ \Gamma}(f)$ and $\widehat{  \Lambda}(f)$ in Definition \ref{sBSSunmixest} so that when $n\to \infty${,}
\[
d_w\{Q_n , \mathcal{N}(  0, {F_1}) \} \to 0, 
\]
where details concerning the matrix $ {F_1}$ are given in Appendix \ref{appendix:F:two:matrices}.
\end{proposition}

{In Proposition \ref{prop_eigen}, similarly as before, we consider the sequences of vectors obtained by vectorizing  $n^{1/2} \{\widehat{ \Gamma}(f) - \Omega^{-1}\}$ and taking the diagonal of $n^{1/2} \{ \widehat{  \Lambda}(f) - \Lambda(f) \}$. Again, we do not show that the sequence of joint distributions of these vectors converges to a fixed distribution. Instead, we show that these joint distributions are asymptotically close to Gaussian distributions, with covariance matrices given by ${F_1}$. We remark that ${F_1}$ denotes a sequence of $(p^2+p) \times (p^2 +p)$ matrices.
We also remark that, in Definition \ref{sBSSunmixest}, $\widehat{ \Gamma}(f)$ is not uniquely defined. It is defined up to the signs of its rows. Hence, Proposition \ref{prop_eigen} shows that there exists a choice of  the sequence $\widehat{ \Gamma}(f)$ in Definition \ref{sBSSunmixest} such that asymptotic normality holds as $n \to \infty$. 
}

{The performance of the estimators $\widehat{ \Gamma}(f)$ and $\widehat{  \Lambda}(f) $ depends on the choice of {$\widehat{M}(f)$} that should be chosen so that $\widehat{  \Lambda}(f)$ has diagonal elements as distinct as possible.} This is similar to the time series context as described in \citet{MiettinenNordhausenOjaTaskinen2012}. 
To avoid this dependency in the time series context, the joint diagonalisation of more than two matrices has been suggested 
and we will apply this concept  to the spatial context in the following section.

\section{Improving the estimation of the spatial blind source separation model by jointly diagonalising more than two matrices}\label{BSSest2}

Spatial blind source separation with more than two kernel functions of the form $f_0,f_1,\ldots,f_k$, with $k \geq 2$, can be formulated as

\begin{equation} \label{hat:gamma:k:matrices}
\widehat{\Gamma}
\in 
\argmax_{
\substack{
 \Gamma: \Gamma \widehat{M}(f_0) \Gamma^\T = I_p  \\
\Gamma \mbox{\small \; has rows } \gamma_1^\T,\ldots,\gamma_p^\T
} 
 }
 \sum_{l=1}^k \sum_{j=1}^p \{ \gamma_j^\T \widehat{M}(f_l) \gamma_j \}^2.
\end{equation}
We can show that, if $k=1$, the set of $\widehat{\Gamma}$ satisfying 
\eqref{hat:gamma:k:matrices} coincides with the set of $\widehat{\Gamma}(f_1)$ satisfying Definition \ref{sBSSunmixest}.
From experience in time series blind source separation \citep[see for example][]{MiettinenIllnerNordhausenOjaTaskinenTheis2016}, usually the diagonalisation of several matrices gives a better separation than those based on two matrices only.
In this paper, we indeed show that using $k \geq 2$ is beneficial from a theoretical point of view and in practice.

The identifiability notion of Definition \ref{def:identifiability:two:matrices} and Proposition \ref{prop:identifiability} can be extended to the case of more than two local covariance matrices. We first remark that the theoretical version of \eqref{hat:gamma:k:matrices} is
\begin{equation} \label{gamma:k:matrices}
\Gamma
\in 
\argmax_{
\substack{
 \Gamma: \Gamma M(f_0) \Gamma^\T = I_p  \\
\Gamma \mbox{\small \; has rows } \gamma_1^\T,\ldots,\gamma_p^\T
} 
 }
 \sum_{l=1}^k \sum_{j=1}^p \{ \gamma_j^\T M(f_l) \gamma_j \}^2.
\end{equation}

We then extend Definition~\ref{def:identifiability:two:matrices} and Proposition~\ref{prop:identifiability} to the case of more than two local covariance matrices.

\begin{definition} \label{def:identifiability:k:matrices}
We say that the unmixing problem given by $f_1,\ldots,f_k$ is identifiable if any unmixing functional $\Gamma$ satisfying \eqref{gamma:k:matrices} can be written as
$P S \Omega^{-1}$, where $P$ is a permutation matrix and $S$ is a diagonal matrix with diagonal elements equal to $-1$ or $1$.
\end{definition}

\begin{proposition} \label{prop:identifiability:k:matrices}
The unmixing problem given by $f_1,\ldots,f_k$ is identifiable if and only if for every pair $ i \neq j $, $i, j = 1, \ldots, p$, there exists $ l = 1, \ldots , k$ such that $   \{\Omega^{-1} M(f_l) \Omega^{- \T}\}_{i,i} \neq  \{\Omega^{-1} M(f_l) \Omega^{- \T}\}_{j,j}  $.
\end{proposition}

{ Proposition \ref{prop:identifiability:k:matrices} is proved in {\S}~\ref{supplement:proof:identifiability:two}} of the supplementary material.
We remark that the identifiability condition in Proposition \ref{prop:identifiability:k:matrices} is weaker than that in Proposition \ref{prop:identifiability}, because if the condition in Proposition \ref{prop:identifiability} holds with $f$ being one of the $f_1,\ldots,f_k$, then the condition in Proposition \ref{prop:identifiability:k:matrices} holds. This is one of the benefits of jointly diagonalising more than two matrices.

One of the main theoretical contributions of this paper is to provide an asymptotic analysis of the joint diagonalisation of several matrices in the spatial context.  Assumption \ref{assumption:idendifiability:two:matrices:asymptotic}, on asymptotic identifiability, can be replaced by the following weaker assumption.

\begin{assumption} \label{assumption:identifiability:k:matrices}
A fixed $\delta >0$ and $n_0 \in \mathbb{N}$ exist so that for all $n \in \mathbb{N}$, $n \geq n_0$, for every pair $ i \neq j $, $i, j = 1, \ldots, p$, there exists $ l = 1, \ldots , k$, such that $ |  \{
\Omega^{-1} M(f_l) \Omega^{- \T}\}_{i,i} - \{\Omega^{-1} M(f_l) \Omega^{- \T}\}_{j,j} | \geq \delta $.
\end{assumption}

{In the next proposition, we prove the consistency of $\widehat{\Gamma}$. This proposition is proved in {\S}~\ref{appendix:more:two:matrices} of the supplementary material.}

\begin{proposition} \label{prop:consistency:k:matricesb}
Suppose Assumptions \ref{assumption:mean:zero:Z} to \ref{assumption:Kk} hold.
Let $k \in \mathbb{N}$ be fixed. Let $f_1,\ldots,f_k: \mathbb{R}^d \to \mathbb{R}$ satisfy Assumption \ref{assumption:f}. 
Assume that Assumption \ref{assumption:identifiability:k:matrices} holds.
Let $\widehat{\Gamma} = \widehat{\Gamma}\{\widehat{M}(f_0),\widehat{M}(f_1),\ldots,\widehat{M}(f_k)\}$ satisfy \eqref{hat:gamma:k:matrices}.
Then we can choose $\widehat{\Gamma}$ so that $\widehat{\Gamma} \to  \Omega^{-1}$ in probability when $n$ goes to infinity.
\end{proposition}

{In Proposition \ref{prop:consistency:k:matricesb}, we remark that $\widehat{\Gamma}$ is defined only up to permutation of the rows and multiplications of them by $1$ or $-1$. Hence, we show that there exists a choice of a sequence $\widehat{\Gamma}$ that converges to $\Omega^{-1}$. The next proposition provides an asymptotic normality result. It is proved in {\S}~\ref{appendix:more:two:matrices} of the supplementary material.}

\begin{proposition}\label{asymp_bssmulti}
Assume the same assumptions as in Proposition \ref{prop:consistency:k:matricesb}.
Let
$(\widehat{\Gamma}_n)_{n \in \mathbb{N}}$ be any sequence of  $p \times p$ matrices so that for any $n \in \mathbb{N}$, $\widehat{\Gamma}_n = \widehat{\Gamma}_n\{\widehat{M}(f_0),\widehat{M}(f_1),\ldots,\widehat{M}(f_k)\}$  satisfies {\eqref{hat:gamma:k:matrices}}. Then, a sequence of permutation matrices $(P_n)$ and a sequence of diagonal matrices $(D_n)$ exist, with diagonal components in $\{-1,1\}$, so that  the distribution $Q_n$ of ${n}^{1/2} \mathrm{vect} ( \check{\Gamma}_n - \Omega^{-1} )$
with $\check{\Gamma}_n = D_n P_n \widehat{\Gamma}_n$ satisfies, as $n \to \infty$,
\[
d_w\{ Q_n , \mathcal{N}(0,{F_k}) \}
\to 0,
\]
where details concerning the matrix $F_k$ are given in Appendix \ref{appendix:F:k:matrices}.
\end{proposition}

{In Proposition \ref{asymp_bssmulti}, for any $n \in \mathbb{N}$, the choice of $\widehat{\Gamma}_n$
satisfying \eqref{hat:gamma:k:matrices} is not unique.
The proposition shows that, for any choice of the sequence of matrices $\widehat{\Gamma}_n$, one can exchange the rows and multiply them by $1$ or  $-1$, to obtain a sequence of matrices $\check{\Gamma}_n$ that converges to $\Omega^{-1}$ as $n \to \infty$. 
Furthermore, similarly as in Proposition \ref{prop_eigen}, we show that the sequence of distributions of ${n}^{1/2} \mathrm{vect} ( \check{\Gamma}_n - \Omega^{-1} )$ is asymptotically close to a sequence of Gaussian distributions. The sequence of $p^2 \times p^2$ covariance matrices of these Gaussian distributions  is ${F_k}$. }

The idea of joint diagonalisation is not new in spatial data analysis. For example in \citet{XieMyers1995}, \citet{XieMyersLong1995} and \citet{IacoMyersPalmaPosa2013}, in a model-free context, matrix variograms have been jointly diagonalized. However, the unmixing matrix was restricted to be orthogonal, which would therefore not solve the spatial blind source separation model.

While two symmetric matrices can always be simultaneously diagonalized, this is usually not the case for more than two matrices which are estimated based on finite data. Therefore, algorithms are needed for approximate joint diagonalisation. In this paper we use an algorithm which is based on {Givens} rotations \citep{Clarkson1988}. Other possible algorithms and their impact on the properties of the estimates are for example discussed in \citet{IllnerEtAl2015}.

\section{Simulations}\label{Simu}

\subsection{Preliminaries}

In this section we use simulated data to verify our asymptotic results and to compare the efficiencies of the different local covariance estimates under a varying set of spatial models. All simulations are performed in R \citep{R} with the help of the packages \textit{geoR} \citep{Rgeor}, \textit{JADE} \citep{MiettinenNordhausenTaskinen2017} and \textit{RcppArmadillo} \citep{Rrcpparmadillo}. To generate the simulation data, we have chosen some particular covariance functions for the latent fields. However, our proposed methods do not use this information in any way, but operate solely through the selection of local covariance matrices.

\subsection{{
Asymptotic approximation of the unmixing matrix estimator}}\label{subsec:simu_1}

We start with a simple simulation to establish the validity of the {asymptotic approximation} of the unmixing matrix estimator $\widehat{\Gamma}(f)$ for different kernels $f$ and to obtain some preliminary comparative results between the proposed estimators. We consider a centered, three-variate spatial blind source separation model $X({s}) = {\Omega} {Z}({s}) $ where each of the three independent latent fields has a Mat\'ern covariance function with shape and range parameters $(\kappa,\phi) \in \{(6,\text{1$\cdot$2}),(1,\text{1$\cdot$5}),(\text{0$\cdot$25},1)\}$, which correspond to the left panel in Fig.~\ref{fig:cov_functions}. We recall that the Mat\'ern correlation function is defined by
{
\[ 
\rho(h) = 2^{1-\kappa}\, \Gamma(\kappa)^{-1}\, (h/ \phi)^\kappa\, K_{\kappa}(h/ \phi),
\]
}where $ \kappa > 0$ is the shape parameter, $ \phi > 0 $ is the range parameter and $ K_{\kappa} $ is the modified Bessel function of the second kind {of order $ \kappa $}. Our location pattern is constructed in the following way: the first 200 locations are drawn uniformly random from an origin-centered square $ S_1 $ of side length $ 200^{1/2} $ units. For the next 200 locations, we scale the side length of the square $ S_1 $ by the factor $ 2^{1/2} $ to obtain the larger square $ S_2 $ and draw the points uniformly random on $ S_2 \setminus S_1 $. Next, we always scale the side length of the previous square $ S_j $ by $ 2^{1/2} $ to obtain $ S_{j + 1} $ and draw the same amount of locations we already have on $ S_{j + 1} \setminus S_j $, thus  doubling the number of points every time. This process is continued until we have obtained a total of $ 3200 $ locations. In the simulation we consider the sample sizes $ n = 100 \times 2^j $, for $ j = 1, \ldots, 5 $, each time using the first $ n $ of the $ 3200 $ points, that is, all points inside the $ j $th innermost square on the left-hand side of Fig.~\ref{fig:grids}. The six samples then correspond to nested samples of points and represent the increasing domain asymptotic scheme implied by Assumption \ref{assumption:expansion}.

\begin{figure}[t]
\centering
\includegraphics[width=\textwidth]{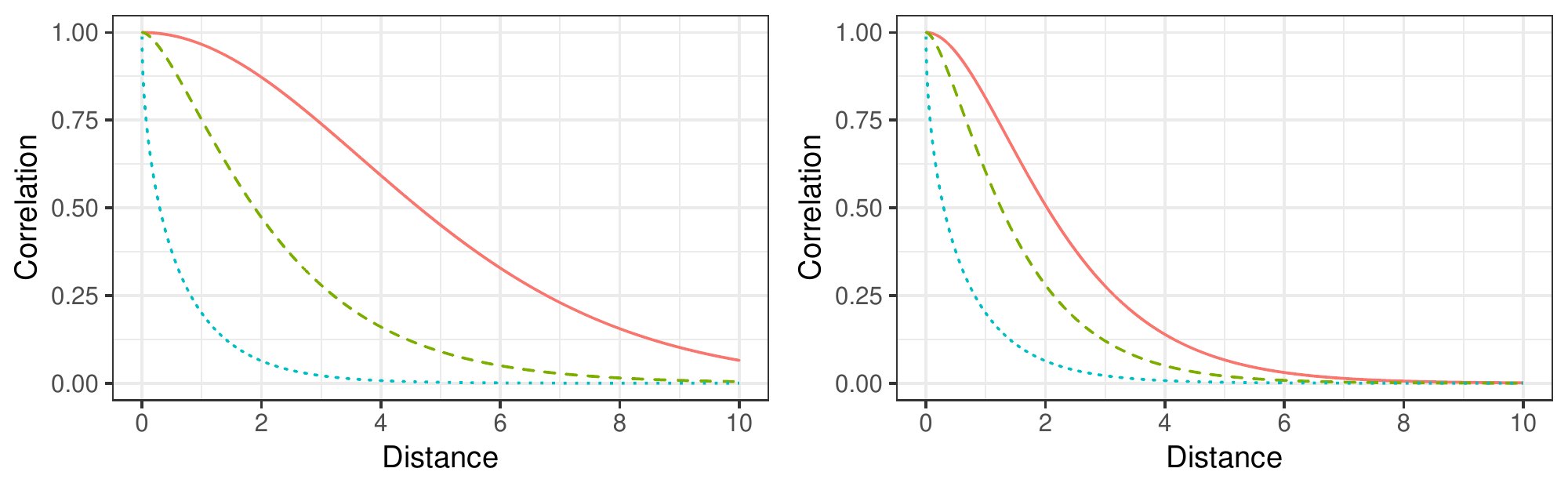}
\caption{Mat\'ern covariance functions of the {first (solid red line), second (dashed green line) and third (dotted blue line)} latent fields {used in \S~\ref{subsec:simu_1} (left panel) and \S~\ref{subsec:simu_2} (right panel). } The parameter vectors $ (\kappa, \phi) $ of the three fields equal, $ (6,\text{1$\cdot$2}),(1,\text{1$\cdot$5}),(\text{0$\cdot$25},1) $ and $ (2,1),(1,1),(\text{0$\cdot$25},1) $, respectively.}
\label{fig:cov_functions}
\end{figure}

\begin{figure}[t]
	\centering
	\includegraphics[width=0.8\textwidth]{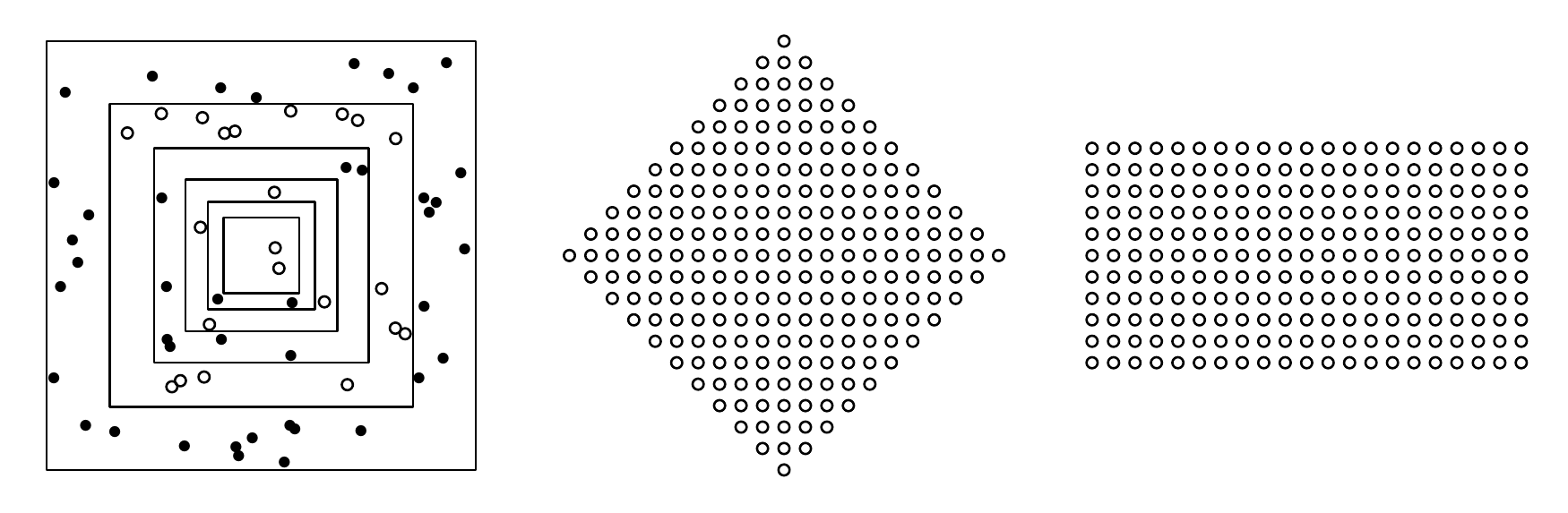}
	\caption{From left to right: the location pattern scheme used in  {\S}~\ref{subsec:simu_1} {(with the marker type alternating between the consecutive layers and only 1 percent of the locations shown for clarity), a diamond grid of radius 10 having $ n = 221 $ locations and a rectangle grid of radius 10 having $ n = 231 $ locations. The diamond and rectangle grids, with a one unit distance between two neighbouring locations, are used in {\S}~}\ref{subsec:simu_2}.}
	\label{fig:grids}
\end{figure}

We expect any successful unmixing estimator $\widehat{\Gamma}$ to satisfy $\widehat{\Gamma} {\Omega} \approx {I}_p$ up to sign changes and row permutations. The minimum distance index \citep{ilmonen2010new} is defined as,
\[
\MDI(\widehat{\Gamma}) = (p - 1)^{-1/2} \inf\{\| {C} \widehat{\Gamma} \Omega  - {I}_p \|, C \in \mathcal{C} \},
\]
where $\mathcal{C}$ is the set of all matrices with exactly one non-zero element in each row and column and $ \| \cdot \| $ is the Frobenius norm. The minimum distance index measures how close $\widehat{\Gamma} \Omega$ is to the identity matrix up to scaling, order and signs of its rows, and $0 \leq \MDI(\widehat{{\Gamma}}) \leq 1$ with lower values indicating more efficient estimation. Moreover, for any $\widehat{{\Gamma}}$ such that ${n}^{1/2} \mathrm{vect}(\widehat{{\Gamma}} - {I}_p) \rightarrow \mathcal{N}({0}, {\Sigma})$ for some limiting covariance matrix ${\Sigma}$, the transformed index $n (p - 1) \MDI(\widehat{{\Gamma}})^2$ converges to a limiting distribution $\sum_{i = 1}^k \delta_i \chi^2_i$ where $\chi^2_1, \ldots , \chi^2_k$ are independent chi-squared random variables with one degree of freedom and $\delta_1, \ldots , \delta_k$ are the $k$ non-zero eigenvalues of the matrix,
\[
\left({I}_{p^2} - {D}_{p, p} \right) {\Sigma} \left({I}_{p^2} - {D}_{p, p} \right),
\]
where ${D}_{p, p} = \sum_{j = 1}^p {E}^{jj} \otimes {E}^{jj}$ and $ E^{jj} $ is the $ p \times p$ matrix with one as its $ (j, j) $th element and the rest of the elements equal zero, and $\otimes$ is the usual tensor matrix product. In particular, the expected value of the limiting distribution is the sum of the limiting variances of the off-diagonal elements of $\widehat{{\Gamma}}$. This provides us with a useful single-number summary to measure the asymptotic efficiency of the method, i.e., the mean value of $n (p - 1) \MDI(\widehat{{\Gamma}})^2$ over several replications.

Following the argument of the proof of Proposition~\ref{prop:TCL:deux:matrices} 
in the supplementary material, our spatial blind source separation estimators are affine equivariant. 
More precisely, let $\widehat{\Gamma}(I_p)$ be computed from $\{Z(s_i)\}_{i=1,\ldots,n}$ according to \eqref{hat:gamma:k:matrices} and recall that $\widehat{\Gamma}$ is computed from $\{X(s_i)\}_{i=1,\ldots,n}$ according to \eqref{hat:gamma:k:matrices}. Then we have $\widehat{\Gamma} = \widehat{\Gamma}(I_p) \Omega^{-1}$, up to sign changes and row permutations.
In this  sense, $\widehat{\Gamma} {\Omega} $ is invariant to the value of $ {\Omega} $. 
 As the minimum distance index depends on $ \widehat{ \Gamma} $ only through $\widehat{\Gamma} {\Omega} $, it is thus without loss of generality that we may consider throughout {\S}~\ref{Simu} only the trivial mixing case ${\Omega} = {I}_3$. Taking different $\Omega$ into consideration would give exactly the same results as those provided below.

Recall that the ball and ring kernels are defined as $B(h)(s)= I(\| s\|\leq h)$ and $R(h_1,h_2)(s)= I(h_1\leq \| s\|\leq h_2)$ for fixed $ h \geq 0 $ and $h_2 \geq h_1 \geq 0$. We simulate 2000 replications  for each sample size $n$ and estimate the unmixing matrix in each case with three different choices for the local covariance matrix kernels: $B(1), R(1, 2)$ and $\{ B(1), R(1, 2) \}$, where the argument $ s $ is dropped and the brackets $ \{ \} $ denote the joint diagonalisation of the kernels inside. 
The latent covariance functions on the left panel of Fig.~\ref{fig:cov_functions} show that the dependencies of the last two fields die off rather quickly, and we would expect that already very local information is sufficient to separate the fields. Moreover, out of all one-unit intervals, the magnitudes of the three covariance functions differ the most from each other in the interval from 1 to 2 and we may reasonably assume that either $R(1, 2)$ or $\{ B(1), R(1, 2) \}$ will be the most efficient choice.

\begin{figure}[t]
\centering
\includegraphics[width=0.8\textwidth]{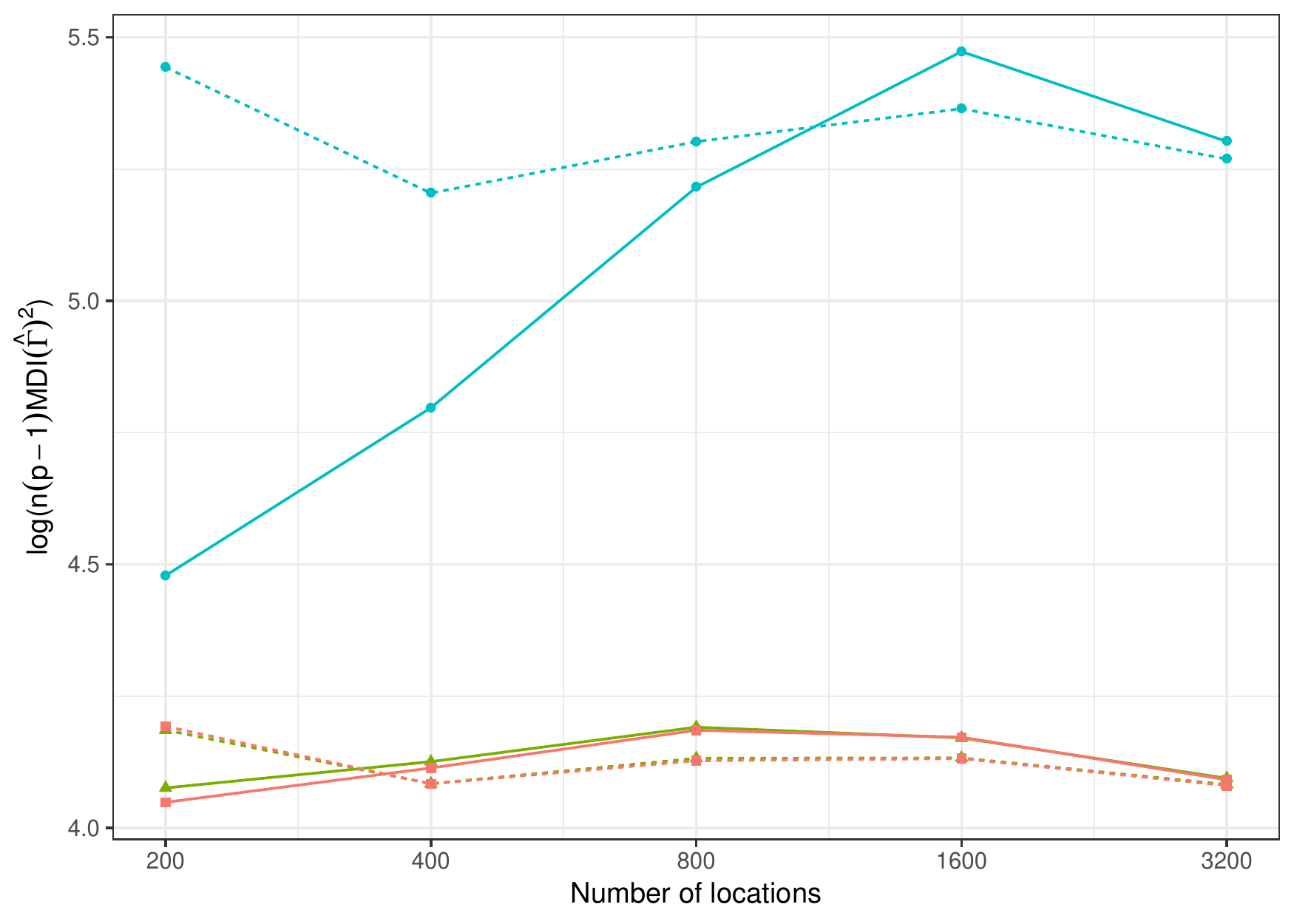}
\caption{{The solid lines give the mean values of $n (p - 1) \MDI(\widehat{{\Gamma}})^2$ in the first simulation and the dashed lines correspond to the asymptotic approximations of the same quantities. The three used local covariance matrices are $ B(1) $ (blue line), $ R(1, 2) $ (green line) and $ \{ B(1), R(1, 2) \} $ (red line).}}
\label{fig:simulation_1_1}
\end{figure}

The mean values of $n (p - 1) \MDI(\widehat{{\Gamma}})^2$ over the 2000 replications are shown as the solid lines in Fig.~\ref{fig:simulation_1_1}, with the dashed lines representing {the asymptotic approximated values of the means, towards which they are expected to converge} (see Propositions \ref{prop_eigen} and \ref{asymp_bssmulti}). As evidenced in Fig.~\ref{fig:simulation_1_1}, this is indeed what happens. For the reasons detailed in the previous paragraph, the kernel $R(1, 2)$ is notably a more efficient choice than $ B(1) $. However, the ball kernel still carries some additional information  to the ring as their joint diagonalisation, $\{ B(1), R(1, 2) \}$, gives the best results out of the three choices, albeit marginally. As the main purpose of the current simulation was to verify the limiting theorems and compare the different choices of kernels, the estimation accuracy of the sources was considered jointly, through the minimum distance index. However, as it is possible that some of the individual sources are more difficult to estimate than others, we have included in {\S}~\ref{SuppSimuComp} of the supplementary material a simulation study exploring individual component recovery.

The previous investigation and Fig.~\ref{fig:simulation_1_1} used only the expected value of the asymptotic distribution. 
{In Fig.~\ref{fig:simulation_1_2} of the supplementary material, 
we have also plotted the estimated densities of $n (p - 1) \MDI(\widehat{{\Gamma}})^2$ for all local covariance matrices and a few selected sample sizes and compared with the density of the asymptotic approximation estimated from a sample of 100,000 random variables drawn from the corresponding distributions. Overall, the two densities fit each other rather well, especially for the local covariance matrices involving the ring kernel. This shows that the asymptotic approximation {to the} distribution of $n (p - 1) \MDI(\widehat{{\Gamma}})^2$ is good already for small sample sizes.}
%

\subsection{The effect of range on the efficiency}\label{subsec:simu_2}

The second simulation explores the effect of the range of the latent fields on the asymptotically optimal choice of local covariance matrices. The comparisons between the estimators are made on the basis of the expected values of the asymptotic approximations to the distribution of $n (p - 1) \MDI(\widehat{{\Gamma}})^2$ (that is, using the equivalent of the dashed lines in Fig.~\ref{fig:simulation_1_1}), meaning that no randomness is involved in this simulation.

We consider three-variate random fields ${X}({s}) = {\Omega} {Z}({s}) $, where ${\Omega} = {I}_3$ and the latent fields have Mat\'ern covariance functions with respective shape parameters $\kappa = 2, 1, \text{0$\cdot$25}$ and a range parameter $ \phi \in\, \{\text{1$\cdot$0}, \text{1$\cdot$1}, \text{1$\cdot$2}, \ldots , \text{30$\cdot$0}\}$. The three covariance functions are shown for $ \phi = 1 $ in the right panel of Fig.~\ref{fig:cov_functions}.
The random field is observed at three different point patterns: diamond-shaped, rectangular and random, which was simulated once and held fixed throughout the study. The diamond-shaped point pattern has a radius of $ m = 30 $ and a total of $ n = 1861 $ locations, whereas the rectangular point pattern has a ``radius'' of $ m = 15 $ with a total of $ n = 1891 $ locations. In both patterns, the horizontal and vertical distance between two neighbouring locations is one unit and examples of the two pattern types are shown in the middle and right panels of Fig.~ \ref{fig:grids} with a radius $m=10$. A rectangular pattern with ``radius'' $ m $ is defined to have the width $ 2m + 1 $ and the height $ m + 1 $. The random point pattern is generated simply by simulating $ n = 1861 $ points uniformly in the rectangle $ (-30, 30) \times (-30 ,30) $. We consider a total of eight different local covariance matrices, $B(r), R(r-1, r)$ for $ r = 1, 3, 5$, and the joint diagonalisations of the previous sets: $ \{B(1), B(3), B(5)\} $ and $ \{R(0, 1), R(2, 3), R(4, 5)\} $. 

\begin{figure}[t]
	\centering
\includegraphics[width=0.8\textwidth]{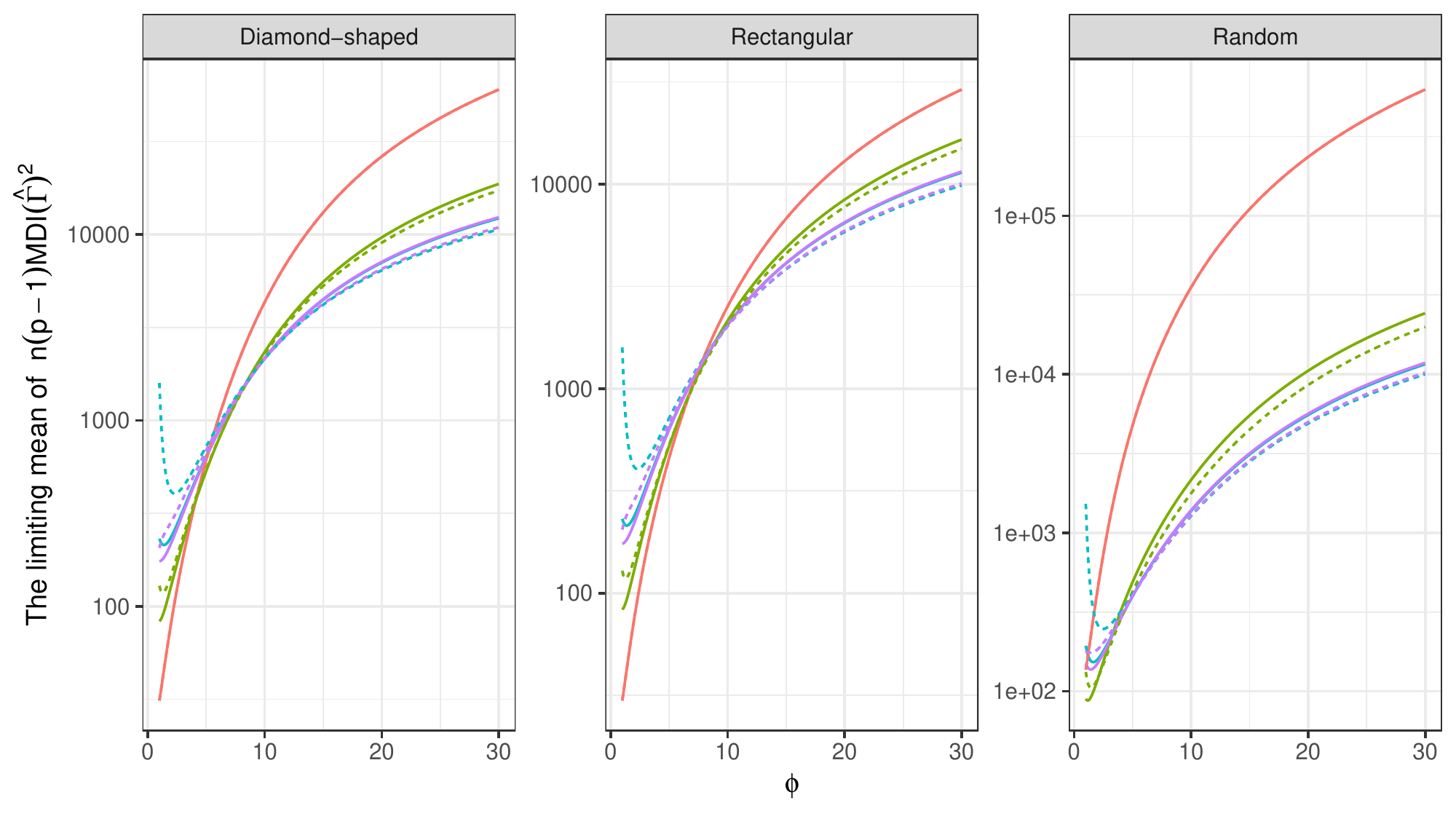}
	\caption{The {asymptotic approximate} mean values of $n (p - 1) \MDI(\widehat{{\Gamma}})^2$ as a function of the range of the latent Mat\'ern random fields for the different choices of local covariance matrices in the second simulation. {The solid and the dashed lines correspond, respectively, to the ball and ring kernels and the value of the parameter $ r $ is indicated by the color of the line as follows: 1 (red), 3 (green), 5 (blue), J (purple).} The $ y $-axis has a logarithmic scale.}
	\label{fig:simulation_2}
\end{figure}

The results of the simulation are displayed in Fig.~\ref{fig:simulation_2} where the two joint diagonalisations are denoted by having value ``J'' as the parameter $ r $. Recall that the lower the value on the $ y $-axis, the better that particular method is at estimating the three latent fields. The relative ordering of the different curves is very similar across all three plots, and it seems that the choice of the location pattern does not have a large effect on the results. In all the patterns, the local covariance matrices with either $ r = 1 $ or $ r = 3 $ are the best choices for small values of the range $ \phi $ but they quickly deteriorate as $ \phi $ increases. The opposite happens for the local covariance matrices with $ r = 5 $; they are among the worst for small $ \phi $ and relatively improve with increasing $ \phi $. The joint diagonalisation-based choices fall somewhere {in-between} and are never the best nor the worst choice. However, they yield performance very close to the best choice in the right end of the range-scale and are close to the optimal ones in the left end. Thus, their use could be justified in practice as the ``safe choice''. Comparing the two types of local covariance matrices, balls and rings, we observe that in the majority of cases the rings prove superior to the balls.

\subsection{Efficiency comparison}\label{subsec:simu_3}

To compare a larger number of local covariance matrices and their combinations, we simulate three-variate random fields ${X}({s}) = {\Omega} {Z}({s}) $, where ${\Omega} = {I}_3$ and the latent fields have Mat\'ern covariance functions with the shape parameters $\kappa = 6, 1, \text{0$\cdot$25}$ and the range parameter $\phi = 20$, in kilometers. We consider two different fixed-location patterns fitted inside the map of Finland; see Fig.~\ref{fig:finlands}. The first location pattern has the locations drawn uniformly from the map and the second location pattern is drawn from a west-skew distribution. Both patterns have a total of $n = 1000$ locations and to better distinguish the scale we have added three concentric circles with respective radii of 10, 20, and 30 kilometers in the empty area of the skew map.

We simulate a total of 2000 replications of the above scheme with the fixed maps. In each case we compute the minimum distance index values 
of the estimates obtained with the local covariance matrix kernels $B(r), R(r-10, r), G(r)$, where $r = 10, 20, 30, 100$, and the joint diagonalisation of each of the three quadruplets $\{ B(10), B(20), B(30), B(100 \}$, 
$\{ R(10), R(20), R(30), R(100 \}$ and $\{ G(10), G(20), G(30), G(100 \}$ adding up to a total of 15 estimators. The Gaussian kernel is parametrized as $G(r) \equiv \exp[ -\text{0$\cdot$5} \{ \Phi^{-1}(\text{0$\cdot$95}) s/r \}^2 ]$, where $ s $ is the distance and $\Phi^{-1}(x)$ is the quantile function of the standard normal distribution, making $G(r)$ have $90$\% of its total mass in the radius $r$ ball around its center. Thus, $G(r)$ can be considered  a smooth approximation of $B(r)$. 
The larger radius kernels $B(100)$, $R(90, 100)$, $G(100)$ are included in the simulation to investigate what happens when we overestimate the dependency radius. The mean minimum distance index values for the 15 estimators are plotted in Fig.~\ref{fig:simulation_4_1} 
and show that for both maps and all local covariance types, increasing the radius yields more accurate separation results all the way up to $r = 30$, whereas for $r = 100$ the results again worsen. This observation shows that when using a single local covariance matrix, the choice of the type and the radius are especially important, most likely requiring some expert knowledge on the study. However, this problem is completely averted when we use the joint diagonalisation of several matrices. For both maps and all local covariance types the joint diagonalisation produces results very comparable to the best individual matrices, even though the joint diagonalisations also include the ``bad choices'', $r = 10, 20, 100$. We also observe a similar behaviour in the first and second simulation studies where, in the absence of knowledge on the optimal choice, the joint diagonalisation either is the most efficient choice or provides a performance very close to the most efficient choice. Thus, we recommend the use of the joint diagonalisation of scatter matrices with a sufficiently large variation of radii for the kernels.

\begin{figure}[t]
	\centering
	\includegraphics[width=0.6\textwidth]{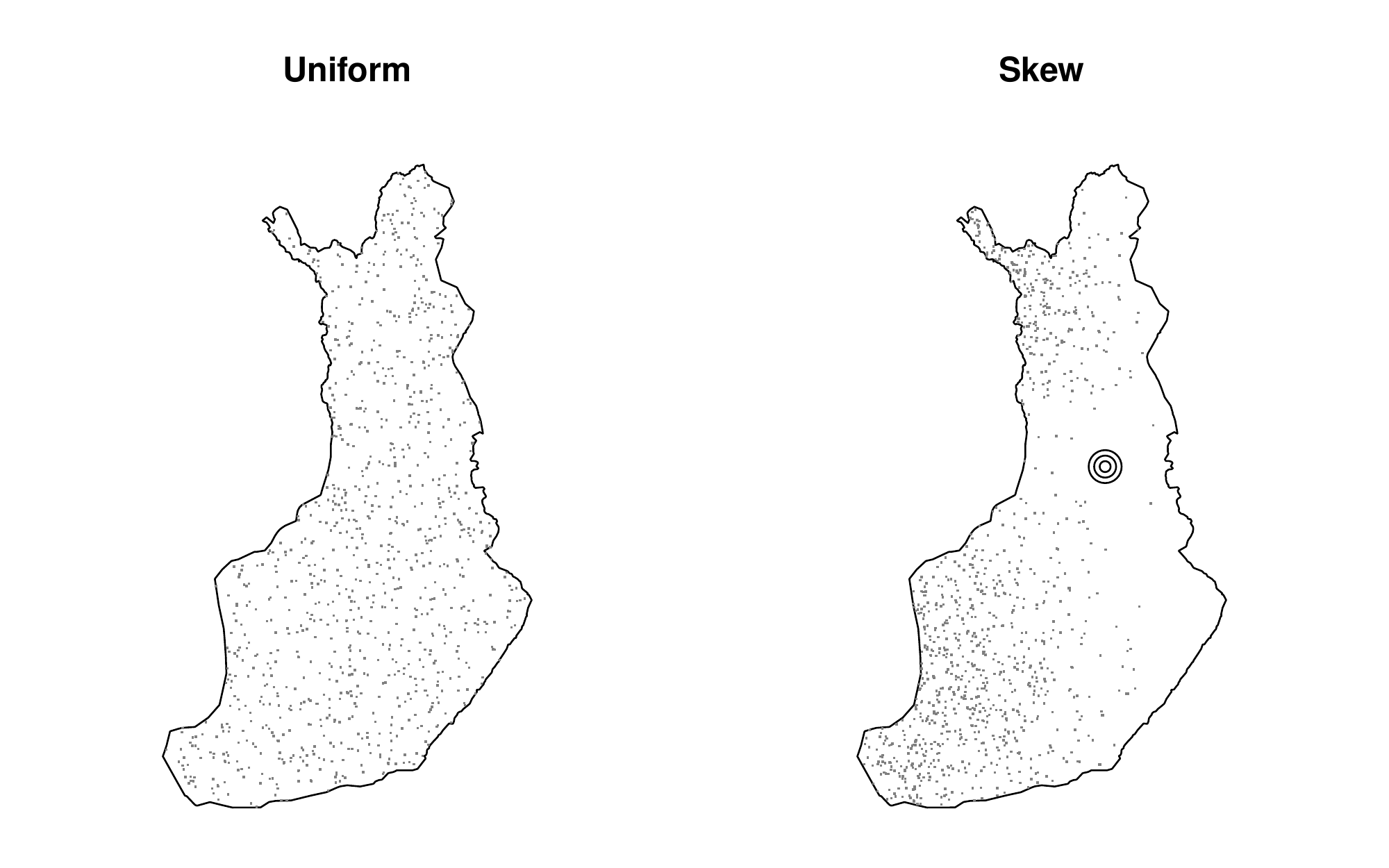}
	\caption{The two fixed location patterns in the map of Finland, the uniform on the left and the skew on the right.}
	\label{fig:finlands}
\end{figure}

Finally, a comparison between the two maps reveals that the relative behaviour of the estimators is roughly the same in both maps, but the estimation is generally more difficult in the skew map, revealed by the on average higher minimum distance index values. This is explained by the large number of isolated points which contribute no information to the estimation of the local covariance matrices, making the sample size essentially smaller than $n = 1000$. 


\begin{figure}[t]
	\centering
	\includegraphics[width=0.5\textwidth]{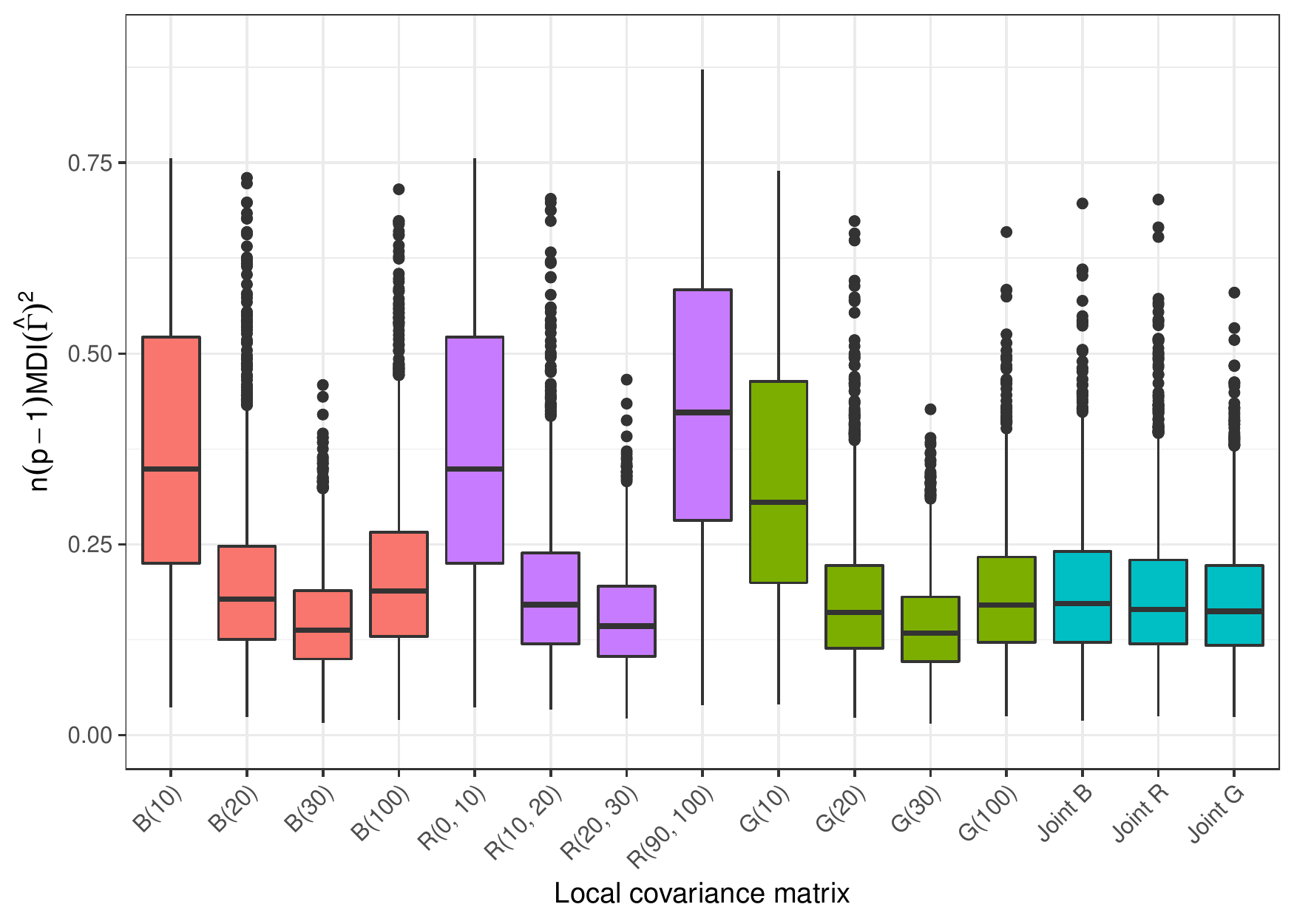}\includegraphics[width=0.5\textwidth]{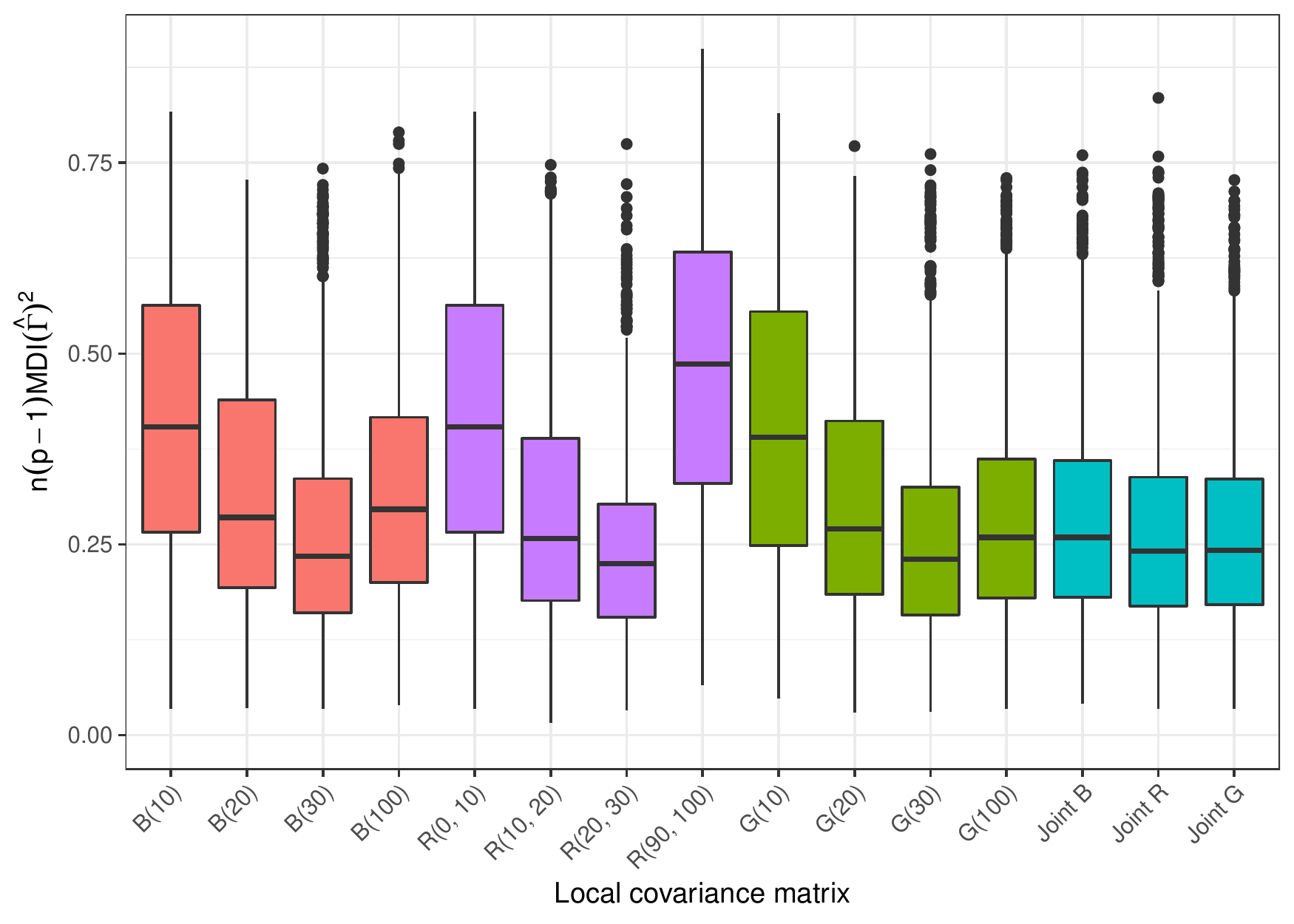}
	\caption{{The results of the efficiency study for the uniform sampling design (left) and for the skew design (right).}}
	\label{fig:simulation_4_1}
\end{figure}


\section{Data application}\label{Example}

To illustrate the benefit of jointly diagonalising more than two scatter matrices from a practical point of view, we reconsider the moss data from the Kola project which {are} available in the R package StatDa \citep{StatDA} and described in \citet{ReimannFilzmoserGarrettDutter:2008}, for example. The data {consist} of 594 samples of terrestrial moss collected at different sites in north Europe on the borders of Norway, Finland and Russia. The corresponding map with sampling locations is given in the online supplement in Fig.~\ref{fig:KolaInfo}. The amount of 31 chemical elements found in the moss samples was already used as a spatial blind source separation example in \citet{NordhausenOjaFilzmoserReimann2015} where the covariance matrix and $B(50)$ were simultaneously diagonalized.
The goal of that analysis was to reveal interpretable 
components exhibiting clear spatial patterns.
In \citet{NordhausenOjaFilzmoserReimann2015}, the radius of 50 kilometers was carefully chosen by an expert in that analysis and considered best compared to several other radii not mentioned there. The analysis found six meaningful components, which could be used to distinguish underlying natural geological patterns from environmental pollution patterns. These six components had the six largest eigenvalues and are visualized in Fig.~\ref{fig:B50} in the online supplement.

We show that the {gold standard} components can be stably estimated without subject knowledge on the optimal radius by simply jointly diagonalizing a large enough collection of local covariance matrices. To address the compositional nature of the data, we follow the same data preparation steps as in \citet{NordhausenOjaFilzmoserReimann2015} and then compute five competing spatial blind source separation estimates. The scatters we used in addition to the covariance matrix are detailed in Table~\ref{tab:CorKola}. Using these methods, we identify the six components with the highest correlations, in absolute values, to the six main components identified 
in \citet{NordhausenOjaFilzmoserReimann2015}. Table~\ref{tab:CorKola} gives the correlations of the six components. 
\begin{table}
	\def~{\hphantom{0}}
	\tbl{Maximal absolute correlations of different estimators with respect to the {gold} standard. All estimators used the empirical covariance matrix. The distances for the scatters are given in kilometers}{
		\begin{tabular}{c p{5.7cm} cccccc}
			Est & Scatters & IC1 & IC2 & IC3 & IC4 & IC5 & IC6 \\ 
			1 & $B(25)$  & 0$\cdot$96 & 0$\cdot$93 & 0$\cdot$91 & 0$\cdot$68 & 0$\cdot$64 & 0$\cdot$77 \\
			2 & $B(75)$  & 0$\cdot$98 & 0$\cdot$98 & 0$\cdot$92 & 0$\cdot$96 & 0$\cdot$91 & 0$\cdot$63 \\
			3 & $B(100)$ & 0$\cdot$76 & 0$\cdot$80 & 0$\cdot$77 & 0$\cdot$96 & 0$\cdot$60 & 0$\cdot$53 \\
			4 & $R(0,25)$, $R(25,50)$, $R(50,75)$, $R(75,100)$ & 0$\cdot$97 & 0$\cdot$98 & 0$\cdot$92 & 0$\cdot$97 & 0$\cdot$83 & 0$\cdot$80 \\
			5 & $R(0,10)$, $R(10,20)$, $R(20,30)$, $R(30,40)$,
			$R(40,50)$, $R(50,60)$, $R(60,70)$, $R(70,80)$
			& 0$\cdot$96 & 0$\cdot$97 & 0$\cdot$91 & 0$\cdot$97 & 0$\cdot$78 & 0$\cdot$77 \\
	\end{tabular}}
	\label{tab:CorKola}
	\begin{tabnote}
		Est, estimator; IC, independent component.
	\end{tabnote}
\end{table}
The table shows that when using only two scatters, estimators 1, 2 and 3, some components cannot be easily found. However, when jointly diagonalising more than two scatters, the results are more stable and less dependent on the chosen distances of the scatters as can be seen for estimators 4 and 5.

This is 
illustrated using the {gold standard} and estimators 3 and 4 in Fig.~\ref{fig:ICs_1_2} in the Appendix for the first two components. For completeness, {\S}~\ref{section:details:real:data} of the online supplement contains all six components for the three estimators.
The first two components  represent, according to \citet{NordhausenOjaFilzmoserReimann2015}, areas with different types of industrial contamination and Figure~\ref{fig:ICs_1_2} shows that the {gold standard} and estimator 4 agree quite well on these, but estimator 3 yields a different map. More precisely, the first component obtained by the {gold} standard and the estimator 4 highlights a cluster of negative scores around the Monchegorsk and Apatity region, which reveals the mining and processing of alkaline deposits. This cluster is not revealed by estimator~3. Similarly, the second components are similar between the {gold} standard and the estimator 4, but the component from the estimator 3 differs from these two, especially for the sampling locations in Finland. Thus, using several scatters gives a more stable impression whereas the maps can vary considerably when only two scatters are used, in which case subject expertise becomes more relevant.

\section{Discussion}

Our proposed methodology can be extended in multiple directions in future work. The assumptions of Gaussian or stationary fields could be relaxed. The spatial and temporal blind source separation methodologies could be combined to obtain spatio-temporal blind source separation. If used for dimension reduction, estimators for the number of latent non-noise fields could be devised using strategies similar to those in \cite{VirtaNordhausen2019}. Additionally, the combination of spatial blind source separation with univariate kriging and univariate modelling warrants investigation.

How to choose the local covariance matrices optimally is also of interest. This is still an open problem for temporal blind source separation methods, such as second-order blind identification \citep{BelouchraniAbedMeraimCardosoMoulines:1997}. Several strategies have been suggested, see for example \citet{tang2005recovery}, and many of them could be useful also in selecting the kernels in spatial blind source separation. The estimation accuracy of our proposed method is based on how well separated the eigenvalues of the matrices $ M(f_0)^{-1/2} M(f_l) M(f_0)^{-1/2} $, $ l = 1, \ldots , k$, are. Since the connection between the eigenvalues and the unknown covariance functions is complicated, our suggestion, backed up also by the simulations, is to stay on the safe side and jointly use a large number of ring kernels. However, including large numbers of unnecessary kernels can still have the drawback of inducing some noise to the estimates. One way to remove the unneeded kernels would be to first obtain preliminary estimates for the latent fields using a large number of kernels jointly. Then, our asymptotic results could be used to select from a large collection of sets of kernels, the one which achieves the smallest value of $\delta_1 + \cdots + \delta_k $; see {\S}~\ref{subsec:simu_1}. The final estimates could then be computed with this asymptotically optimal choice of kernels. A similar technique was used in the context of temporal blind source separation in \cite{taskinen2016more}.

\section*{Acknowledgement}
The work of {Nordhausen, Ruiz-Gazen and Virta} was partly supported by the CRoNoS Cost action. The work of {Nordhausen} was also partly supported by the Austrian Science Fund. {Ruiz-Gazen acknowledges funding from the French National Research Agency (ANR) under the Investments for the Future (Investissements d’Avenir) program.} The authors are very grateful for the comments by the referees which helped considerably to improve the manuscript.
 

\appendix
\appendixone

\counterwithin{lemma}{section}
\counterwithin{proposition}{section}
\counterwithin{theorem}{section}

\section{Appendix}

\subsection{Notation}
\label{appendix:notation}

Let $ y$ and $z$ be the $np \times 1$ vectors defined by $ y_{(i-1)p + j} = Y_j( s_i)$ and $ z_{(i-1)p + j} = Z_j( s_i)$, for $i=1,\ldots,n$, $j=1,\ldots,p$. Let $ R = \mathrm{cov}( y)$ and $ R_z = \mathrm{cov}( z)$.  Let $e_b(p)$ be the $b$th base column vector of $\mathbb{R}^p$ for $b=1,\ldots,p$.
For $f: \mathbb{R}^d \to \mathbb{R}$ and for $b,l=1,\ldots,p$,  let $ T_{b,l}(f)$ be the $np \times np$ matrix, that we see as a block matrix composed of $n^2$ blocks of sizes $p^2$, and with block $i,j$ equal to $f(  s_i -  s_j) (1/2)\{e_b(p)  e_l(p)^\T+ e_l(p)  e_b(p)^\T\}$.

For $b \in \mathbb{N}$, we let $\mathcal{D}(b) = \{ 1+ (i-1)(b+1) ; i=1,\ldots,b \}$. We remark that $\{ \mathrm{vect}( M)_i; i \in \mathcal{D}(b)\} = \{ M_{i,i} ; i=1,\ldots,b \}$ for a $b \times b$ matrix $M$.
Let $\bar{\mathcal{D}}_b = \{1,\ldots,b^2\} \backslash \mathcal{D}_b$. We remark that $\{ \mathrm{vect}( M)_i; i \in \bar{\mathcal{D}}(b)\} = \{ M_{i,j} ; i,j=1,\ldots,b , i \neq j \}$ for a $b \times b$ matrix $M$.
For $a \in \{1,\ldots,b^2\}$, let $I_b(a)$ and $J_b(a)$ be the unique $i,j \in \{1,\ldots,b \}$ so that $a = b(i-1)+j$. For $i  \in \{ 1,\ldots,b\}$, let $d_b(i) = 1+(i-1)(b+1)$ and note that $\{ \mathrm{vect}( M)_{d_b(i)} ; i=1,\ldots,b \} = \{ M_{i,i} ; i=1,\ldots,b \}$ for a $b \times b$ matrix $M$. For a matrix $ M$ of size $b \times b$, recall that $\diag( M) = (M_{1,1},\ldots,M_{b,b})^\T $ 
{and that $\mathrm{tr}(M)$ denotes its trace}.

\subsection{{Expression of the matrix $V(f,f_0)$ from Proposition \ref{prop_cov_asymptotic_normality}}}
\label{appendix:expression:V}

Let $f,g: \mathbb{R}^d \to \mathbb{R}$.
{Using the notation of Appendix} \ref{appendix:notation}, let $ \Sigma(f)$  and $ \Sigma(f,g)$ be the $p^2 \times p^2$ matrices defined by, for $i = (s-1)p + t$ and $j = (u-1)p + v$, with $s,t,u,v \in \{1,\ldots,p\}$,
\[
 \Sigma(f)_{i,j} = 2n^{-1}\mathrm{tr} \left\{ R  T(f)_{s,t} R  T(f)_{u,v} \right\} \ \mbox{and} \ \Sigma(f,g)_{i,j} = 2n^{-1}\mathrm{tr} \left\{  R  T(f)_{s,t} R  T(g)_{u,v} \right\}.
\]
Let
\vspace{-5mm}
\[
 V(f,g) =
\begin{pmatrix}
 \Sigma(f) &  \Sigma(f,g) \\
 \Sigma(g,f) &  \Sigma(g)
\end{pmatrix}.
\]
Then $V(f,f_0)$ is equal to $V(f,g)$ for $g = f_0$.

\subsection{Expression of the matrix ${F_1}$ from Proposition \ref{prop_eigen}}
\label{appendix:F:two:matrices}

 From Assumption \ref{assumption:idendifiability:two:matrices:asymptotic}, there exists $n_0 \in \mathbb{N}$ such that for $n \geq n_0$ the diagonal  elements of $\Omega^{-1} M(f)\Omega^{-\T}$ are strictly decreasing. Write these diagonal elements as $\lambda_1>  \cdots> \lambda_p$. Using the notation of Appendix~\ref{appendix:notation}, for $n \geq n_0$,
let $ A$, $B$, $C$ and $D$ be {respectively} the $p^2 \times p^2$, $p^2 \times p^2$, $p \times p^2$ and $p \times p^2$ matrices defined by
\[
A_{i,j} =
\begin{cases}
 - 1/2 & ~ ~\mbox{for} ~ ~ i=j  \in \mathcal{D}(p), \\
 - \lambda_{I_p(i)} \{ \lambda_{I_p(i)} - \lambda_{J_p(i)} \}^{-1} & ~ ~\mbox{for} ~ ~ i=j \not \in \mathcal{D}(p), \\
0  & ~ ~ \mbox{otherwise}, ~ ~
\end{cases}
\]
\[
B_{i,j} =
\begin{cases}
 \{\lambda_{I_p(i)} - \lambda_{J_p(i)} \}^{-1} & ~ ~\mbox{for} ~ ~ i=j   \not \in \mathcal{D}(p), \\
0 &~ ~ \mbox{otherwise}, ~ ~
\end{cases}
\]
\[
C_{i,j} =
\begin{cases}
 - \lambda_i & ~ ~\mbox{for} ~ ~ j = d_{p}(i), \\
0 &~ ~ \mbox{otherwise} ~ ~
\end{cases}
~ ~ ~
\mbox{and}
~ ~ ~
D_{i,j} =
\begin{cases}
1 & ~ ~\mbox{for} ~ ~ j = d_{p}(i), \\
0 &~ ~ \mbox{otherwise}. ~ ~
\end{cases}
\]
Let
\vspace{-5mm}
\[
G = \begin{pmatrix}
A & B \\
C & D
\end{pmatrix}.
\]
Let $M_{\Omega^{-1}}$ and $\bar{M}_{\Omega^{-1}}$ be {respectively} the $p^2 \times p^2$ and $(p^2+p) \times (p^2 + p)$ matrices defined by
\[
( M_{\Omega^{-1}} )_{a,b}
=
\begin{cases}
( \Omega^{-1} )_{ J_p(b) , J_p(a) } & ~ ~ \mbox{if} ~ ~ I_p(a) = I_p(b), \\
0 & ~ ~ \mbox{if} ~ ~ I_p(a) \neq I_p(b)
\end{cases}
~ ~ ~
\mbox{and}
~ ~ ~
\bar{M}_{\Omega^{-1}}
=
\begin{pmatrix}
M_{\Omega^{-1}} & 0 \\
0 & I_p
\end{pmatrix}.
\]

Let $\tilde{V}(f)$ be defined as $V(f_0,f)$ but with $R$ replaced by $R_z$.
Then, for $n \geq n_0$, ${F_1}$ is defined as 
\[
{F_1} = \bar{M}_{\Omega^{-1}} G \tilde{V}(f) G^\T \bar{M}_{\Omega^{-1}}^\T.
\]

\subsection{Expression of the matrix ${F_k}$ from Proposition \ref{asymp_bssmulti}}
\label{appendix:F:k:matrices}

Let $D(f) = \Omega^{-1} M(f) \Omega^{-\T}$.
For a diagonal matrix $\Lambda$, let $\Lambda_r = \Lambda_{r,r}$. Let $A_0,A_1,\ldots,A_k$ and $B$ be $p^2 \times p^2$ matrices defined by, for $n \geq n_0$ with the notation of Assumption \ref{assumption:identifiability:k:matrices},
\[
A_{0,i,j} =
\begin{cases}
 - 1/2 & ~ ~\mbox{for} ~ ~ i=j  \in \mathcal{D}(p), \\
 - 
\sum_{l=1}^k 
\{
D(f_l)_{I_p(i)} - D(f_l)_{J_p(i)}
\}
D(f_l)_{{I_p(i)}}
  & ~ ~\mbox{for} ~ ~ i=j \not \in \mathcal{D}(p), \\
0  & ~ ~ \mbox{otherwise}, ~ ~
\end{cases}
\]
\[
A_{l,i,j} =
\begin{cases} 
D(f_l)_{I_p(i)} - D(f_l)_{J_p(i)}
  & ~ ~\mbox{for} ~ ~ i=j \not \in \mathcal{D}(p), \\
0  & ~ ~ \mbox{otherwise}, ~ ~
\end{cases}
\quad \quad \mbox{ for } l=1,\ldots,k,
\]
and
\vspace{-5mm}
\[
B_{i,j} =
\begin{cases}
 1 & ~ ~\mbox{for} ~ ~ i=j  \in \mathcal{D}(p), \\
[\sum_{l=1}^k 
\{
D(f_l)_{I_p(i)} - D(f_l)_{J_p(i)}
\}^2 ]^{-1}
  & ~ ~\mbox{for} ~ ~ i=j \not \in \mathcal{D}(p), \\
0  & ~ ~ \mbox{otherwise}. ~ ~
\end{cases}
\]
Let $G$ be the $p^2 \times (k+1)p^2$ matrix defined by $G= B
(
A_0,  A_1,  \ldots  A_k 
)$, for $n \geq n_0$. Let $M_{\Omega^{-1}}$ be as in Appendix  \ref{appendix:F:two:matrices}.
Let $\tilde{V}(f_1,\ldots,f_k)$ be the $(k+1)p^2 \times (k+1)p^2$ matrix composed of $(k+1)^2$ blocks of size $p^2 \times p^2$ with block $(i+1),(j+1)$ defined similarly as $\Sigma(f_i,f_j)$ in Appendix \ref{appendix:expression:V}, but with $R$ replaced by $R_z$.
 Then, for $n \geq n_0$,
${F_k}$ is defined as
\[
{F_k}= M_{\Omega^{-1}} G\tilde{V}(f_1,\ldots,f_k) G^\T M_{\Omega^{-1}}^\T.
\]

\subsection{Map for data application}

\begin{figure}[h]
\centering
\includegraphics[width=0.99\textwidth]{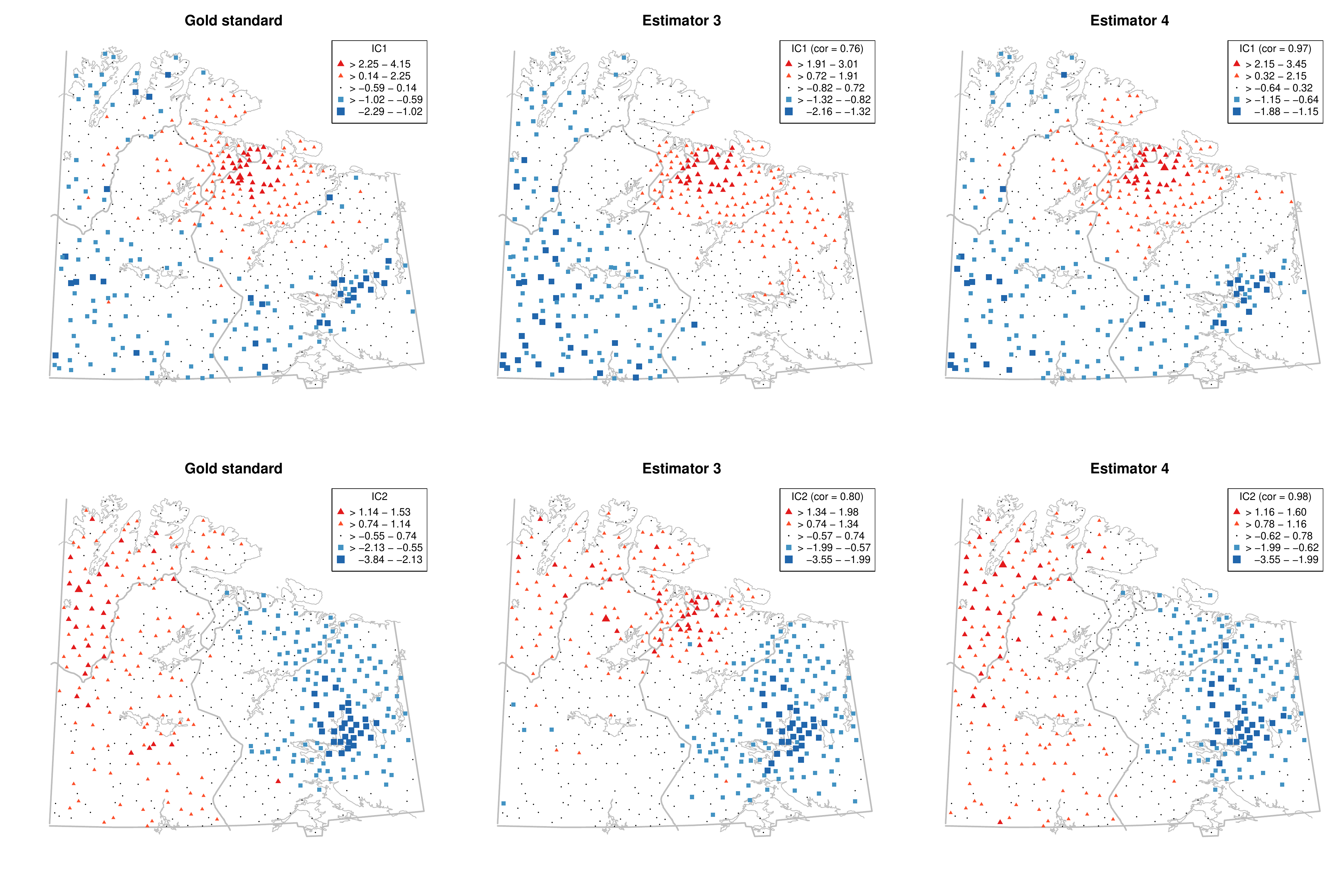}
\caption{The first two independent components from the {gold} standard and estimators 3 and 4.  
}
\label{fig:ICs_1_2}
\end{figure}


\section{Proofs}

\subsection{Introduction}

{
	We first prove Proposition \ref{prop:identifiability} in Section \ref{supplement:proof:identifiability:one}. Then Section \ref{supplement:general:results} provides general results for the proofs of Propositions \ref{prop_cov_consistency}, \ref{prop_cov_asymptotic_normality}, \ref{prop_eigen}, \ref{prop:consistency:k:matricesb} and \ref{asymp_bssmulti}.
	Section \ref{appendix:two:matrices} provides the proofs of Propositions \ref{prop_cov_consistency}, \ref{prop_cov_asymptotic_normality} and \ref{prop_eigen}.
	Section \ref{supplement:proof:identifiability:two} provides the proof of Proposition \ref{prop:identifiability:k:matrices}.
	Section \ref{appendix:more:two:matrices} provides the proofs of Propositions \ref{prop:consistency:k:matricesb} and \ref{asymp_bssmulti}.
}

{
	Propositions \ref{prop_cov_consistency}, \ref{prop_cov_asymptotic_normality}, \ref{prop_eigen}, \ref{prop:consistency:k:matricesb} and \ref{asymp_bssmulti} correspond to Proposition 
	\ref{prop:consistency}, \ref{prop:TCL:for:gamma:lambda}, \ref{prop:TCL:deux:matrices}, \ref{prop:permuting:lines:consistency} and \ref{prop:permuting:lines:asymptotic:normality}, respectively.
}

\subsection{{Proof of Proposition \ref{prop:identifiability}}}
\label{supplement:proof:identifiability:one}

We let $D(f) = \Omega^{-1} M(f) \Omega^{-\T}$ for $f: \mathbb{R}^d \to \mathbb{R}$.
{We restate Proposition \ref{prop:identifiability} and prove it.}

\begin{proposition} \label{prop:identifiability:in:supplement}
	{The unmixing problem given by $f$ is identifiable if and only if the diagonal elements of $\Omega^{-1} M(f) \Omega^{-\T}$ are {distinct}.}
\end{proposition}

\begin{proof}[{of Proposition \ref{prop:identifiability:in:supplement} (Proposition \ref{prop:identifiability})}]
	We have that $\Gamma(f)$ is an unmixing functional if and only if
	\[
	\Gamma(f) \Omega D(f_0) \Omega^\T \Gamma(f)^\T =  I_p \quad \mbox{and} \quad  \Gamma(f) \Omega D(f) \Omega^\T \Gamma(f)^\T =  \Lambda(f).
	\]
	Hence the matrix $ \Gamma(f) \Omega $ is orthogonal, and its rows provide the eigenvectors of the diagonal matrix $D(f)$. If the diagonal elements of $D(f)$ are {distinct}, then the set of one-dimensional eigenspaces of $D(f)$ is unique and thus $ \Gamma(f) \Omega = P S$ where $P$ is a permutation matrix and $S$ is a diagonal matrix with diagonal elements equal to $-1$ or $1$. If there are two diagonal elements of $D(f)$ that are equal, say the first and the second without loss of generality, then consider the $p \times p$ block diagonal matrix $Q$ with first $2 \times 2$ block equal to $\{(2^{-1/2},2^{-1/2})^\T , (2^{-1/2},-2^{-1/2})^\top\}$ and second  $(p-2) \times (p-2)$ block equal to $I_{p-2}$. Then there is an unmixing functional $\Gamma(f)$ such that $\Gamma(f) \Omega  = Q$. In this case, $\Gamma(f) = Q \Omega^{-1} $, which is not of the form $P S \Omega^{-1}$.
\end{proof}

\subsection{General results} \label{supplement:general:results}
Recall that $d \in \mathbb{N}$ and $p \in \mathbb{N}$ are fixed. $Z_1,\ldots,Z_p$ are $p$ independent stationary Gaussian processes on $\mathbb{R}^d$ with zero mean functions, unit variances and covariance functions $K_1,\ldots,K_p$. We have $ Z = (Z_1,\ldots,Z_p)^\T$ and $ X = (X_1,\ldots,X_p)^\T =   \Omega  Z$ with $ \Omega$ a fixed invertible $p \times p$ matrix.

Let $ s_1,\ldots, s_n$ be the $n$ observation points in $\mathcal{S}^d$ and let $f$ be a kernel function from $\mathbb{R}^d$ into $\mathbb{R}$. We recall
\[
\widehat{ M}(f) = n^{-1} \sum_{i=1}^n 
\sum_{j=1}^n
f(  s_i -  s_j  )
X( s_i)  X( s_j)^\T.
\]
and
\[
M(f) = 
n^{-1} \sum_{i=1}^n \sum_{j=1}^n
f(   s_i -  s_j  ) \Omega D( s_i, s_j) \Omega^\T
\]
where $ D( s_i, s_j)$ is the $p \times p$ diagonal matrix defined by
\[
D( s_i, s_j)_{k,k} = K_k( s_i -  s_j).
\]

Let $|x| = \max_{i=1,\dots,m}| x_i|$ be the sup norm for $x \in \mathbb{R}^m$.
Since this norm is equivalent to the Euclidean norm,  and since we work under Assumptions \ref{assumption:expansion} to \ref{assumption:f}, we can assume without loss of generality that the following conditions hold.

\begin{condition} \label{cond:minimal:distance}
	With $\Delta>0$ defined in Assumption \ref{assumption:expansion}, for all $n \in \mathbb{N}$ and for all $a \neq b$, $a,b \in \{1,\ldots,n\}$, we have $| s_a -  s_b| \geq \Delta$.
\end{condition}

\begin{condition} \label{cond:bound:covariance}
	With $A< + \infty$ and $\alpha >0$ defined in Assumptions \ref{assumption:Kk} and \ref{assumption:f}, for all $ s \in \mathbb{R}^d$ and for all $k=1,\ldots,p$, we have
	\[
	| K_k( s) | \leq \frac{ A }{ 1 + | s|^{d + \alpha} }.
	\]
\end{condition}

\begin{condition} \label{cond:bound:weight:function}
	With $A< + \infty$ and $\alpha >0$ defined in Assumptions \ref{assumption:Kk} and \ref{assumption:f}, for all $ s \in \mathbb{R}^d$, we have
	\[
	| f( s) | \leq \frac{ A }{ 1 + | s|^{d + \alpha} }.
	\]
\end{condition}

For a matrix $M$, denote by $M_{i,j}$ the element from the $i$th row and the $j$th column of $M$.
For a vector $V_n$ or a matrix $M_n$, denote by $(V_n)_i$ the $i$th element of $V_n$ and by $(M_n)_{i,j}$  the element from the $i$th row and the $j$th column of $M_n$.
The singular values of a $n\times n$ matrix $ M$ are
denoted by $\rho_1( M)\geq\dots \geq\rho_n( M)\geq 0 $ and, in the
case when $ M$ is symmetric, the eigenvalues are denoted by
$\lambda_1( M)\geq\dots \geq\lambda_n( M)$.  The spectral norm is
given by $\rho_1( M)$ and $\| M\|_F^2=\sum_{i,j}(M_{i,j})^2 $ denotes
the Frobenius norm. 
For a sequence of random variables $X_n$, we write $X_n = o_p(1)$ when
$X_n$ converges to $0$ in probability as $n \to \infty$ and we write
$X_n = O_p(1)$ when $X_n$ is bounded in probability as $n \to \infty$. Let $ e_i(k)$ be the $i$th base column vector of $\mathbb{R}^k$. Let $ y$ be the $np \times 1$ vector defined by $ y_{(i-1)p + j} = X_j( s_i)$, for $i=1,\ldots,n$, $j=1,\ldots,p$.

\begin{lemma} \label{lem:bounded:covy}
	Under Conditions \ref{cond:minimal:distance} and \ref{cond:bound:covariance}, there exists a finite constant $C < + \infty$ so that for all $n \in \mathbb{N}$,
	\[
	\lambda_1\{\mathrm{cov}(y)\} \leq C.
	\]
\end{lemma}

\begin{proof}
	Let $ \ell_a^\T$ denote the $a$th row of $ \Omega$. We have
	\begin{align} \label{eq:bound:cov:fun:X}
	\left| \mathrm{cov} \{ X_a( s_i) , X_b( s_j) \} \right|
	& = \left| \mathrm{cov} \{  \ell_a^\T  Z( s_i) , \ell_b^\T Z( s_j) \} \right| \nonumber
	\\
	& = \left|  \ell_a^\T  D( s_i, s_j) \ell_b \right| \nonumber  \\
	& \leq 
	\sum_{k=1}^p | \Omega_{a,k} | | \Omega_{b,k} |
	\frac{A}{1+| s_i -  s_j|^{d+\alpha}}  
	\nonumber \\
	& \leq p \max_{k,l=1,\ldots,p} ( \Omega_{k,l} )^2
	\frac{A}{1+| s_i -  s_j|^{d+\alpha}}
	\end{align}
	from Condition \ref{cond:bound:covariance}.
	Note also that $\lambda_1\{\mathrm{cov}( y)\} = \lambda_1\{\mathrm{cov}(\tilde{ y})\}$ where $\tilde{ y}$ is the $np \times 1$ vector defined by $\tilde{ y}_{(j-1)n + i} = X_j( s_i)$ for $i=1,\ldots,n$, $j=1,\ldots,p$.
	Hence, the lemma is a direct consequence of Lemma 6 in \cite{furrer16asymptotic}.
\end{proof}

\

The next theorem provides a general multivariate central limit theorem for quadratic forms of Gaussian vectors. It extends standard central limit theorems in spatial statistics, see, e.g., \cite{bachoc14asymptotic} or \cite{istas1997quadratic}, by allowing cases where the sequence of covariance matrices is non-converging or asymptotically singular. The full proof is given for self-consistency, although some of the arguments have appeared previously.

\begin{theorem} \label{prop:general:TCL}
	Let $( y_n)$ be a sequence of $n$-dimensional centered Gaussian vectors. Let $ R_n$ be the covariance matrix of $ y_n$. Assume that for all $n$, $\lambda_1( R_n) \leq A$ where $A$ is a fixed finite constant. Let $k \in \mathbb{N}$ be fixed and let $( T_{1,n}),\ldots,( T_{k,n})$ be $k$ sequences of deterministic $n \times n$ symmetric matrices. Assume that for $i=1,\ldots,k$, for $n \in \mathbb{N}$, $\rho_1( T_{i,n} ) \leq A$. Let $ \Sigma_n$ be the $k \times k$ matrix defined for $1 \leq i,j \leq k$, by
	\[
	( \Sigma_n )_{i,j} = 2n^{-1} \mathrm{tr} \left(
	R_n  T_{i,n}  R_n  T_{j,n}
	\right).
	\]
	Let $ r_n$ be the $k$-dimensional vector defined for $i=1,\ldots,k$, by
	\[
	( r_n)_i = \mathrm{tr} \left(
	n^{-1}R_n  T_{i,n} \right).
	\]
	Let $ V_n$ be the $k \times 1$ vector defined for $1 \leq i \leq k$, by
	\[
	( V_n)_i = n^{-1}y_n^\T T_{i,n}  y_n.
	\]
	Let $Q_n$ be the probability measure of ${n}^{1/2}(  V_n -  r_n )$ on $\mathbb{R}^k$. Let $\mathcal{N}( 0 ,  \Sigma_n)$ be the Gaussian distribution on $\mathbb{R}^k$ with mean vector $ 0$ and covariance matrix $ \Sigma_n$. Let $d_w$ denote a metric generating the topology of weak convergence on the set of Borel probability measures on $\mathbb{R}^k$; for specific examples see the discussion in \cite{dudleyreal} p. 393. Then we have, for $n \to \infty$,
	\[
	d_w\{Q_n, \mathcal{N} (  0,  \Sigma_n ) \} \to 0.
	\]
\end{theorem}

\begin{proof}
	Assume that $d_w\{ Q_n, \mathcal{N} (  0, \Sigma_n ) \} \not \to 0$ when  $n \to \infty$. Then there exists $\epsilon>0$ fixed and a subsequence $n_m$ so that $d_w\{ Q_{n_m}, \mathcal{N} (  0,  \Sigma_{n_m} ) \} \geq \epsilon$. Let $a_1,\ldots,a_k \in \mathbb{R}$ be fixed.
	Let $ S_{n_m} =  R_{n_m}^{1/2} (\sum_{i=1}^k a_i  T_{i,n_m}) R_{n_m}^{1/2}$. We have
	\[
	\sum_{i,j=1}^k a_i a_j ( \Sigma_{n_m})_{i,j}
	=
	2 \mathrm{tr} \left(
	S_{n_m}^{2}
	\right)/n_m.
	\]
	Hence, we see that $ \Sigma_{n_m}$ is a non-negative matrix, and, from the assumptions on $( R_{n_m})$ and $( T_{i,n_m})$, that $( \Sigma_{n_m})_{i,i} \leq 2A^4$. Also, $|( r_n)_i| = | n^{-1} \mathrm{tr} (  R_n^{1/2}  T_{i,n}  R_n^{1/2} ) |  \leq A^2$. Hence, by compacity, and up to extracting a further subsequence, we can assume that $ r_{n_m} \to  r$ and $ \Sigma_{n_m} \to  \Sigma$ when ${n_m \to \infty}$. One can show simply that $d_w\{\mathcal{N}(  0,  \Sigma_{n_m} ) , \mathcal{N}(  0,  \Sigma ) \} \to 0$ when ${n_m \to \infty}$. Hence, when $n_m \to \infty$,
	\begin{align} \label{eq:for:TCL:dw:eps}
	\limsup d_w\{Q_{n_m} , \mathcal{N}( 0, \Sigma)\} \geq \epsilon.
	\end{align}
	Let us prove \eqref{eq:for:TCL:dw:eps} . We have
	\begin{align*}
	&\left| \limsup d_w\{Q_{n_m} , \mathcal{N}( 0, \Sigma)\}
	-
	\limsup d_w\{ Q_{n_m}, \mathcal{N} (  0,  \Sigma_{n_m} ) \}
	\right|
	& \\
	& \leq 
	\limsup \left|  d_w\{Q_{n_m} , \mathcal{N}( 0, \Sigma)\}
	-
	d_w\{ Q_{n_m}, \mathcal{N} (  0,  \Sigma_{n_m} ) \}
	\right| & \\
	& \leq 
	\limsup d_w\{\mathcal{N} (  0,  \Sigma_{n_m} ) , \mathcal{N}( 0, \Sigma)\}
	\\
	& = 0,
	\end{align*}
	where we have applied the triangle inequality for the metric $d_w$ in the last inequality above. Hence $\limsup d_w\{Q_{n_m} , \mathcal{N}( 0, \Sigma)\} = \limsup d_w\{Q_{n_m} , \mathcal{N}( 0, \Sigma_{n_m})\} \geq \epsilon$. Thus \eqref{eq:for:TCL:dw:eps}  is proved.

	We remark that the matrix $S_{n_m} =  R_{n_m}^{1/2} (\sum_{i=1}^k a_i  T_{i,n_m}) R_{n_m}^{1/2}$ is symmetric, because $T_{1,n_m},\ldots,T_{k,n_m}$ are assumed to be symmetric in the theorem. Hence, $S_{n_m}$ can be diagonalized and there exist a matrix $P_{n_m}$ such that $ P_{n_m}  P_{n_m}^\T =  I_{n_m}$ and a diagonal matrix $ D_{n_m}$ such that $ S_{n_m}=  P_{n_m} D_{n_m} P_{n_m}^\T$. Let also $ z_{n_m} =  R_{n_m}^{-1/2}  y_{n_m}$.
	Observe that $ z_{n_m}$ follows the $\mathcal{N}(  0,  I_{n_m} )$ distribution. We have
	\begin{align*}
	\sum_{i=1}^k a_i (V_n)_i
	& = n_m^{-1}  y_{n_m}^\T \left( \sum_{i=1}^k a_i  T_{i,n_m} \right) y_{n_m} \\
	& = n_m^{-1}  z_{n_m}^\T   S_{n_m}  z_{n_m} \\
	& = n_m^{-1} \sum_{a=1}^{n_m}  (\xi_{n_m})_a^2 \lambda_a \left(  S_{n_m} \right),
	\end{align*}
	where $\xi_{n_m}$ follows the $\mathcal{N}(  0,  I_{n_m} )$ distribution. 
	Hence letting
	\begin{align*}
	W_{n_m} & =
	{n_m}^{1/2} \{
	\sum_{i=1}^k a_i ( V_{n_m})_i
	-
	\sum_{i=1}^k a_i ( r_{n_m})_i
	\} ,
	\end{align*}
	we have
	\begin{align*}
	W_{n_m} & =  { n_m^{-1/2}} \sum_{l=1}^{n_m} \{ (\xi_{n_m})_{l}^2 - 1\} \lambda_{l} \left(  S_{n_m}
	\right).
	\end{align*}
	If $\sum_{i,j=1}^k a_i a_j  \Sigma_{i,j} = 0$, then $\sum_{i,j=1}^k a_i a_j ( \Sigma_{n_m})_{i,j} \to 0$ when $n_m \to \infty$. Hence, $2n_m^{-1} \mathrm{tr} (  S_{n_m}^2 )\linebreak[1] \to 0$ and so $\mathrm{var}( W_{n_m} ) \to 0$. 
	Hence $ W_{n_m} \to \mathcal{N}(0,0) = \mathcal{N}(0, \sum_{i,j=1}^k a_i a_j \Sigma_{i,j} )$
	in distribution when $n_m \to \infty$.
	
	Now, if  $\sum_{i,j=1}^k a_i a_j \Sigma_{i,j} >0$, one can show from the Lindeberg-Feller central limit theorem that when $n_m \to \infty$, $W_{n_m} \to \mathcal{N}(0, \sum_{i,j=1}^k a_i a_j \Sigma_{i,j} )$ in distribution, see also Lemma 2 in \cite{istas1997quadratic}.
	
	Hence, since both of the above-considered convergences in distribution hold for any $a_1,\ldots,a_k$, we have, by Cram\'er-Wold theorem, that when $n_m \to \infty$,
	${n_m}^{1/2}(  V_{n_m} -  r_{n_m} ) \to \mathcal{N}(  0 , \Sigma )$ in distribution. This is in contradiction with \eqref{eq:for:TCL:dw:eps}. Hence when ${n \to \infty}$
	\begin{align*}
	d_w\{Q_n, \mathcal{N} ( 0,  \Sigma_n ) \}\to 0.
	\end{align*}\
\end{proof}

\subsection{Asymptotics when diagonalising two matrices}
\label{appendix:two:matrices}

The next proposition gives the consistency of $\widehat{ M}(f)$.

\begin{proposition} \label{prop:consistency}
	Let Conditions \ref{cond:minimal:distance} and \ref{cond:bound:covariance} hold and let $f: \mathbb{R}^d \to \mathbb{R}$ satisfy Condition \ref{cond:bound:weight:function}. Then as $n \to \infty$,
	$\widehat{ M}(f) -  M(f) \to 0$ in probability.
\end{proposition}

\begin{proof}
	Clearly ${E} \{ \widehat{ M}(f) \} =  M(f)$. 
	Let $k,l \in \{ 1,\ldots,p\}$ be fixed.  In order to prove the proposition, it is sufficient to show that when 
	$n \to \infty$, $\mathrm{var}\{\widehat{ M}(f)_{k,l} \} \to 0$.
	
	We have,
	\begin{align*}
	\widehat{ M}(f)_{k,l} & = n^{-1} \sum_{i=1}^n   \sum_{j=1}^n  f(  s_i -  s_j)   e_k(p)^\T  X( s_i)  X( s_j)^\T  e_l(p) \\
	& = n^{-1} \sum_{i=1}^n   \sum_{j=1}^n  f(  s_i -  s_j)    X( s_i)^\T  e_k(p)  e_l(p)^\T  X( s_j)  \\
	& = n^{-1} \sum_{i=1}^n   \sum_{j=1}^n  f(  s_i -  s_j)    X( s_i)^\T 
	[
	(1/2) \{e_k(p)  e_l(p)^\T
	+ e_l(p)  e_k(p)^\T \}]
	X( s_j).
	\end{align*}
	Let $ T_{k,l}(f)$ be the $np \times np$ matrix, that we see as a block matrix composed of $n^2$ blocks of sizes $p^2$, and with block $i,j$ equal to $f(  s_i -  s_j) (1/2)\{e_k(p)  e_l(p)^\T+ e_l(p)  e_k(p)^\T\}$. We remark that $ T_{k,l}(f)$ is symmetric.
	With this notation,
	\[
	\widehat{ M}(f)_{k,l} =  n^{-1}y^\T  T_{k,l}(f) y.
	\]
	The largest singular value of $ T_{k,l}(f)$ is bounded as $n \to \infty$. Indeed, from Gershgorin's circle theorem, $\rho_1\{T_{k,l}(f)\}$ is no larger than $\max_{i=1,\ldots,np} \sum_{j=1}^{np} |  T_{k,l}(f)_{i,j}  |$. This maximum is no larger than $ \max_{i=1,\ldots,n} \sum_{j=1}^n | f(  s_i -  s_j) | $. This last quantity is bounded as $n \to \infty$ from Condition \ref{cond:bound:weight:function} and from Lemma 4 in \cite{furrer16asymptotic}.
	
	Hence $\rho_1\{ T_{k,l}(f)\}$ is bounded by a constant $B < + \infty$. Thus, using Theorem 3.2d.3 in \cite{mathai1992quadratic} and the fact that $T_{k,l}(f)$ is symmetric,
	\begin{eqnarray*}
		\mathrm{var}\{ \widehat{ M}(f)_{k,l} \}
		& = & 
		n^{-2} \mathrm{var}\{
		y^\T  T_{k,l}(f) y
		\}
		\\
		& = & 2 n^{-2}
		\mathrm{tr}
		\{
		\mathrm{cov}( y)
		T_{k,l}(f)
		\mathrm{cov}( y)
		T_{k,l}(f)
		\}
		\\
		& \leq &
		2pn^{-1} B^2 C^2,
	\end{eqnarray*}
	with $\lambda_1 \{\mathrm{cov}( y)\} \leq C$ from Lemma \ref{lem:bounded:covy}.
\end{proof}

The next proposition is a corollary of Theorem \ref{prop:general:TCL} and gives the asymptotic normality of $\widehat{ M}(f)$.

\begin{proposition} \label{prop:TCL:for:gamma:lambda}
	Let, for $k,l=1,\ldots,p$ and $f: \mathbb{R}^d \to \mathbb{R}$, $ T_{k,l}(f)$ be defined as in the proof of Proposition \ref{prop:consistency}.
	Let $ R = \mathrm{cov}( y)$ and let $ \Sigma(f)$ be the $p^2 \times p^2$ matrix defined by, for $i = (s-1)p + t$ and $j = (u-1)p + v$, with $s,t,u,v \in \{1,\ldots,p\}$,
	\[
	\Sigma(f)_{i,j} = 2n^{-1}\mathrm{tr} \left\{ R  T(f)_{s,t} R  T(f)_{u,v} \right\}.
	\]
	Define, for $g: \mathbb{R}^d \to \mathbb{R}$, $ \Sigma(f,g)$ as the $p^2 \times p^2$ matrix defined for $i = (s-1)p + t$ and $j = (u-1)p + v$, with $s,t,u,v \in \{1,\ldots,p\}$ by
	\[
	\Sigma(f,g)_{i,j} = 2n^{-1}\mathrm{tr} \left\{  R  T(f)_{s,t} R  T(g)_{u,v} \right\}.
	\]
	Let
	\[
	V(f,g) =
	\begin{pmatrix}
	\Sigma(f) &  \Sigma(f,g) \\
	\Sigma(g,f) &  \Sigma(g)
	\end{pmatrix}.
	\]
	Then, $V(f,g)$ is symmetric non-negative definite.
	
	Assume that Conditions \ref{cond:minimal:distance} and \ref{cond:bound:covariance} hold. Let $f_1,f_2 : \mathbb{R}^d \to \mathbb{R}$ satisfy Condition \ref{cond:bound:weight:function}.
	Let for $r=1,2$, $ W(f_r)$ be the vector of size $p^2 \times 1$, defined for $i = (a-1)p + b$, $a,b \in \{1,\ldots,p\}$, by  $  W(f_r)_i = {n}^{1/2} \{ \widehat{ M}(f_r)_{a,b} -  M(f_r)_{a,b} \}$.
	
	Let $Q_n$ be the distribution of $\{W(f_1)^\T , W(f_2)^\T\}^\T$.
	Then as $n \to \infty$
	\begin{equation} \label{eq:CLT:Vfunfdeux}
	d_w[Q_n, \mathcal{N}\{ 0, V(f_1,f_2)\} ]\to 0.
	\end{equation}
	Furthermore $\lambda_1\{ V(f_1,f_2) \}$ is bounded as $n \to \infty$.
\end{proposition}

Proposition \ref{prop_cov_asymptotic_normality} is a direct corollary of Proposition \ref{prop:TCL:for:gamma:lambda} with $f_1=f$ and $f_2=f_0$. Moreover, Proposition \ref{prop:TCL:for:gamma:lambda} gives details concerning the matrix $ V(f_1,f_2)$.

\begin{proof}
	Let $a,b \in \{1,\ldots,p\}$ and $f : \mathbb{R}^d \to \mathbb{R}$.
	We have seen in the proof of Proposition \ref{prop:consistency} that
	\[
	\widehat{ M}(f)_{a,b} =  n^{-1} y^\T  T_{a,b}(f) y.
	\]
	Hence
	\[
	{E} ( \widehat{ M}(f)_{a,b} ) =  M(f)_{a,b} = n^{-1} \mathrm{tr} \{  R  T_{a,b}(f) \}.
	\]
	Let us first prove that $V(f,g)$ is symmetric non-negative definite. Let $W(f)$ be defined as $W(f_1)$, but with $f_1$ replaced by $f$. Let $W(g)$ be defined similarly with $f_1$ replaced by $g$.
	For $a_1,\ldots,a_{p^2},b_1,\ldots,b_{p^2} \in \mathbb{R}$, we have
	\begin{eqnarray*}
		\mathrm{var} \left\{
		\sum_{i=1}^{p^2}
		a_i W(f)_{i}
		+
		\sum_{i=1}^{p^2}
		b_i W(g)_{i}
		\right\}
		& = & 
		\sum_{i,j=1}^{p^2}
		a_i a_j \mathrm{cov}
		\{
		W(f)_i , W(f)_j
		\}
		+
		\sum_{i,j=1}^{p^2}
		b_i b_j \mathrm{cov}
		\{
		W(g)_i , W(g)_j
		\}
		\\
		& & +  
		2 \sum_{i,j=1}^{p^2}
		a_i b_j \mathrm{cov}
		\{
		W(f)_i , W(g)_j
		\}.
	\end{eqnarray*}
	Now for $i,j=1,\ldots,p^2$, let $a,b,u,v \in \{1,\ldots,p\}$ such that $i=(a-1)p + b$ and $j=(u-1)p + v$. We have, using Theorem 3.2d.3 in \cite{mathai1992quadratic} and the fact that $T_{a,b}(f)$ and $T_{u,v}(f)$ are symmetric,
	\begin{eqnarray*}
		\mathrm{cov} \{
		W(f)_i , W(f)_j
		\}
		& = &
		n \mathrm{cov} \{
		\widehat{M}(f)_{a,b} , \widehat{M}(f)_{u,v}
		\}
		\\
		& = &
		n^{-1}
		\mathrm{cov} \{
		y^\T T_{a,b}(f) y
		,
		y^\T T_{u,v}(f) y
		\}
		\\
		& = &
		2 n^{-1}
		\mathrm{tr} \{
		R T_{a,b}(f) R T_{u,v}(f)
		\}
		\\
		& = &
		\Sigma(f)_{i,j}.
	\end{eqnarray*}
	We show similarly that
	\[
	\mathrm{cov} \{
	W(g)_i , W(g)_j
	\}
	=
	\Sigma(g)_{i,j}
	\]
	and that
	\[
	\mathrm{cov} \{
	W(f)_i , W(g)_j
	\}
	=
	\mathrm{cov} \{
	W(g)_j , W(f)_i
	\}
	=
	\Sigma(f,g)_{i,j}
	=
	\Sigma(g,f)_{j,i}.
	\]
	Hence 
	\begin{eqnarray*}
		\mathrm{var} \left\{
		\sum_{i=1}^{p^2}
		a_i W(f)_{i}
		+
		\sum_{i=1}^{p^2}
		b_i W(g)_{i}
		\right\}
		& = &
		\sum_{i,j=1}^{p^2}
		a_i a_j \Sigma(f)_{i,j}
		+
		\sum_{i,j=1}^{p^2}
		b_i b_j \Sigma(g)_{i,j}
		\\
		& &
		+
		\sum_{i,j=1}^{p^2}
		a_i b_j \Sigma(f,g)_{i,j}
		+
		\sum_{i,j=1}^{p^2}
		a_i b_j \Sigma(g,f)_{j,i}
		\\
		& = &
		\begin{pmatrix}
			a_1,\ldots , a_{p^2},b_1,\ldots , b_{p^2}
		\end{pmatrix}
		\begin{pmatrix}
			\Sigma(f) &  \Sigma(f,g) \\
			\Sigma(g,f) &  \Sigma(g)
		\end{pmatrix}
		\begin{pmatrix}
			a_1 \\
			\vdots \\ 
			a_{p^2} \\
			b_1 \\
			\vdots \\ 
			b_{p^2} 
		\end{pmatrix}.
	\end{eqnarray*}
	Hence, since a square matrix is uniquely defined by its corresponding quadratic forms, it follows that $V(f,g)$ is the covariance matrix of the vector
	\[
	\begin{pmatrix}
	W(f)_1, 
	\ldots  ,
	W(f)_{p^2}, 
	W(g)_1 ,
	\ldots , 
	W(g)_{p^2} 
	\end{pmatrix}^\top.
	\]
	Hence, $V(f,g)$ is symmetric non-negative definite.
	
	Let us now prove \eqref{eq:CLT:Vfunfdeux}.
	We have, from Lemma \ref{lem:bounded:covy} and the proof of Proposition \ref{prop:consistency}, that $\lambda_1(  R )$ and $\rho_1\{  T_{a,b}(f_r) \}$ are bounded as $n \to \infty$ for $r=1,2$. Hence, \eqref{eq:CLT:Vfunfdeux} is a consequence of Theorem \ref{prop:general:TCL}. Finally, $\lambda_1\{ V(f_1,f_2) \}$ is bounded as $n \to \infty$ because each component of $ V(f_1,f_2)$ is bounded as $n \to \infty$.
\end{proof}

Our objective is now to prove Proposition \ref{prop_eigen}, which is a central limit theorem for 
$\widehat{ \Gamma}(f)
-
\Omega^{-1}$
and
$\widehat{  \Lambda}(f)
- \Lambda(f)
$.

There is an equivariance property in Definition \ref{sBSSunmixest} that we will exploit. More precisely, let
$ D_0 = D (0,0) = I_p$ and let
\[
\widehat{ D}_0 = n^{-1} \sum_{i=1}^n
Z( s_i)  Z( s_i)^\T.
\]
For $f: \mathbb{R}^d \to \mathbb{R}$, let
\[
D(f) = n^{-1} \sum_{i=1}^n \sum_{j=1}^n
f(  s_i -  s_j )
D(  s_i ,  s_j )
\]
and
\[
\widehat{ D}(f) = n^{-1} \sum_{i=1}^n \sum_{j=1}^n
f(  s_i -  s_j )
Z( s_i)  Z( s_j)^\T.
\]

Let
$ \Gamma\{\widehat{ D}_0 , \widehat{ D}(f)  \}$ and $\Lambda\{\widehat{ D}_0 , \widehat{ D}(f)  \}$ satisfy the following modification of Definition \ref{sBSSunmixest}:
\begin{equation} \label{eq:bss:on:Z}
\Gamma\{\widehat{ D}_0 , \widehat{ D}(f)  \} \widehat{ D}_0  \Gamma\{\widehat{ D}_0 , \widehat{ D}(f)  \} ^\T =  I_p \quad \mbox{and} \quad \Gamma\{\widehat{ D}_0 , \widehat{ D}(f)  \}
\widehat{ D}(f)
\Gamma\{\widehat{ D}_0 , \widehat{ D}(f)  \}^\T =\Lambda\{\widehat{ D}_0 , \widehat{ D}(f)  \},
\end{equation}
where $\Lambda\{\widehat{ D}_0 , \widehat{ D}(f)  \}$ is a diagonal matrix with diagonal elements in decreasing order.
Then, we can show that
\begin{equation} \label{eq:equivariance:property}
\Gamma\{\widehat{ M}_0 , \widehat{ M}(f)  \} 
=
\Gamma\{\widehat{ D}_0 , \widehat{ D}(f)  \}\Omega^{-1}
~ ~ 
\text{and}
~ ~
\Lambda\{\widehat{ M}_0 , \widehat{ M}(f)  \}
=
\Lambda\{\widehat{ D}_0 , \widehat{ D}(f)  \}
\end{equation}
satisfy Definition \ref{sBSSunmixest}.
The above display is the equivariance property that we will exploit. That is, we will first show a central limit theorem for $\Gamma\{\widehat{ D}_0 , \widehat{ D}(f)  \} - I_p $ and $\Lambda\{\widehat{ D}_0 , \widehat{ D}(f)  \} - \Lambda(f)$ in {Lemma \ref{lem:TCL:deux:matrices:D}}. Then, we will use the equivariance property \eqref{eq:equivariance:property} to obtain, directly,  a central limit theorem for $\widehat{\Gamma}(f) - \Omega^{-1} $ and $\widehat{\Lambda}(f) - \Lambda(f)$ in Proposition \ref{prop:TCL:deux:matrices}.

In the next {lemma}, we first show a central limit theorem for $\widehat{D}_0 - I_p$ and $\widehat{D}(f) - D(f)$. Recall that $\mathrm{vect}( M) = (l_1^\T,\ldots,l_k^\T)^\T$ where $l_1^\T,\ldots,l_k^\T$ are the $k$ rows of a matrix $ M$. Recall also the notation $f_0(x) = I(x=0)$.

\begin{lemma} \label{lem:TCL:hat:D}
	Let Conditions \ref{cond:minimal:distance} and \ref{cond:bound:covariance} hold.
	Let $f: \mathbb{R}^d \to \mathbb{R}$ satisfy Condition \ref{cond:bound:weight:function}.
	Let
	\[
	Y_n = 
	\begin{pmatrix}
	{n}^{1/2} \mathrm{vect} (
	\widehat{D}_0
	-
	D_0
	)
	\\
	{n}^{1/2} \mathrm{vect} \{
	\widehat{D}(f)
	-
	D(f) \}
	\end{pmatrix}.
	\]
	Let $\tilde{V}(f)$ be as $V(f_0,f)$ in Proposition \ref{prop:TCL:for:gamma:lambda} but where $R$ is replaced by $\mathrm{cov}(z)$  where $z$ is the $np \times 1$ vector defined  for $i=1,\ldots,n$, $j=1,\ldots,p$, by $z_{(i-1)p + j} = Z_j( s_i)$.
	Let $Q_{\mbox{\small st},n}$ be the distribution of $Y_n$.
	Then we have
	\[
	d_w[ Q_{\mbox{\small st},n} , \mathcal{N}\{ 0 , \tilde{V}(f) \}] 
	\to 0.
	\]
	Furthermore, $\lambda_1(\tilde{V}(f))$ is bounded as $n \to \infty$.
\end{lemma}
\begin{proof}
	The proof is identical to the proof of Proposition \ref{prop:TCL:for:gamma:lambda}. We remark that, with the notation of the proof of Proposition \ref{prop:TCL:for:gamma:lambda}, $  W(f_r)= {n}^{1/2} \mathrm{vect} \{  \widehat{ M}(f_r) -M(f_r) \}$, for $r=1,2$.
\end{proof}

Now, we show, in {Lemma \ref{lem:DL:proba:for:TCL}}, that the transformation given by \eqref{eq:bss:on:Z}, that defines $\Gamma\{\widehat{ D}_0 , \widehat{ D}(f)  \}$ and $\Lambda\{\widehat{ D}_0 , \widehat{ D}(f)  \}$ from $\widehat{ D}_0$ and $\widehat{ D}(f)$, is asymptotically linear, so to speak. 
This will allow us to transfer the central limit theorem for $\widehat{D}_0 - I_p$ and $\widehat{D}(f) - D(f)$ into a central limit theorem for $\Gamma\{\widehat{ D}_0 , \widehat{ D}(f)  \} - I_p $ and $\Lambda\{\widehat{ D}_0 , \widehat{ D}(f)  \} - \Lambda(f)$, for an appropriate  choice of $\Gamma\{\widehat{ D}_0 , \widehat{ D}(f)  \}$.
This argument is similar to the delta method in asymptotic statistics.

We will need the following notation.
We let $\mathcal{D}(k) = \{ 1+ (i-1)(k+1) ; i=1,\ldots,k \}$. We remark that $\{ \mathrm{vect}( M)_i; i \in \mathcal{D}(k)\} = \{ M_{i,i} ; i=1,\ldots,k \}$ for a $k \times k$ matrix $M$.
Let $\bar{\mathcal{D}}_k = \{1,\ldots,k^2\} \backslash \mathcal{D}_k$. We remark that $\{ \mathrm{vect}( M)_i; i \in \bar{\mathcal{D}}(k)\} = \{ M_{i,j} ; i,j=1,\ldots,k , i \neq j \}$ for a $k \times k$ matrix $M$.
For $a \in \{1,\ldots,k^2\}$, let $I_k(a)$ and $J_k(a)$ be the unique $i,j \in \{1,\ldots,k \}$ so that $a = k(i-1)+j$. For $i  \in \{ 1,\ldots,k\}$, let $d_k(i) = 1+(i-1)(k+1)$ and note that $\{ \mathrm{vect}( M)_{d_k(i)} ; i=1,\ldots,k \} = \{ M_{i,i} ; i=1,\ldots,k \}$ for a $k \times k$ matrix $M$. For a matrix $ M$ of size $k \times k$, recall that $\diag( M) = (M_{1,1},\ldots,M_{k,k})^\T $.

\begin{lemma} \label{lem:DL:proba:for:TCL}
	Let Conditions \ref{cond:minimal:distance} and \ref{cond:bound:covariance} hold.
	Let $f: \mathbb{R}^d \to \mathbb{R}$ satisfy Condition \ref{cond:bound:weight:function}. 
	Assume that Assumption \ref{assumption:idendifiability:two:matrices:asymptotic} holds.
	Remark then that there exists $n_0 \in \mathbb{N}$ such that for $n \geq n_0$ the diagonal  elements of $ D(f)$ are strictly decreasing. Write these diagonal elements as $\lambda_1>  \ldots> \lambda_p$. For $n \geq n_0$,
	let $ A$ be the $p^2 \times p^2$ matrix defined by
	\[
	A_{i,j} =
	\begin{cases}
	- 1/2 & ~ ~\mbox{for} ~ ~ i=j  \in \mathcal{D}(p), \\
	- \lambda_{I_p(i)} \{ \lambda_{I_p(i)} - \lambda_{J_p(i)} \}^{-1} & ~ ~\mbox{for} ~ ~ i=j \not \in \mathcal{D}(p), \\
	0  & ~ ~ \mbox{otherwise}. ~ ~
	\end{cases}
	\]
	Let $ B$ be the $p^2 \times p^2$ matrix defined by
	\[
	B_{i,j} =
	\begin{cases}
	\{\lambda_{I_p(i)} - \lambda_{J_p(i)} \}^{-1} & ~ ~\mbox{for} ~ ~ i=j   \not \in \mathcal{D}(p), \\
	0 &~ ~ \mbox{otherwise}. ~ ~
	\end{cases}
	\]
	Let $ C$ be the $p \times p^2$ matrix defined by
	\[
	C_{i,j} =
	\begin{cases}
	- \lambda_i & ~ ~\mbox{for} ~ ~ j = d_{p}(i), \\
	0 &~ ~ \mbox{otherwise}. ~ ~
	\end{cases}
	\]
	Let $ D$ be the $p \times p^2$ matrix defined by
	\[
	D_{i,j} =
	\begin{cases}
	1 & ~ ~\mbox{for} ~ ~ j = d_{p}(i), \\
	0 &~ ~ \mbox{otherwise}. ~ ~
	\end{cases}
	\]
	Then, with probability going to one as $n \to \infty$, 
	there exist $ \Gamma\{\widehat{ D}_0 , \widehat{ D}(f)  \}$ and $\Lambda\{\widehat{ D}_0 , \widehat{ D}(f)  \}$ satisfying \eqref{eq:bss:on:Z}. Furthermore,  $ \Gamma\{\widehat{ D}_0 , \widehat{ D}(f)  \}$ can be chosen so that as $ n \to \infty $,
	\[
	\begin{pmatrix}
	{n}^{1/2} ( \mathrm{vect} [    \Gamma\{ \widehat{ D}_0 , \widehat{ D}(f)  \} -  I_p ] )\\
	{n}^{1/2} (  \mathrm{diag} [  \Lambda\{\widehat{ D}_0 , \widehat{ D}(f)  \}  -   D(f) ] )
	\end{pmatrix}
	=
	\begin{pmatrix}
	A & B \\
	C &  D
	\end{pmatrix}
	\begin{pmatrix}
	{n}^{1/2} \{  \mathrm{vect} ( \widehat{ D}_0 -  D_0 )  \} \\
	{n}^{1/2} [ \mathrm{vect}  \{ \widehat{ D}(f) -  D(f)  \} ]
	\end{pmatrix}
	+ o_p(1).
	\]
\end{lemma}

\begin{proof}
	Let us assume that $n \geq n_0$ throughout the proof.
	From Proposition \ref{prop:consistency},
	with probability going to one, the eigenvalues of $D_0^{-1/2}\widehat{D}(f) D_0^{-1/2}$ are {distinct}. In the rest of the proof, we set ourselves on the event when this is the case.
	Then, choose $\widehat{ \Gamma} =  \Gamma\{ \widehat{ D}_0 , \widehat{ D}(f) \}$ and $\widehat{ \Lambda} =  \Lambda\{ \widehat{ D}_0 , \widehat{ D}(f) \}$ satisfying \eqref{eq:bss:on:Z} and such that
	\begin{equation} \label{eq:unicity:condition}
	\sum_{j=1}^p 
	\widehat{\Gamma}_{i,j} \geq 0 \;\;  (i=1,\ldots,p).
	\end{equation}
	We remark that $\widehat{ \Gamma}$ and $\widehat{ \Lambda} $ indeed exist, since when $\Gamma\{ \widehat{ D}_0 , \widehat{ D}(f) \}$ and $  \Lambda\{ \widehat{ D}_0 , \widehat{ D}(f) \}$  satisfy \eqref{eq:bss:on:Z}, one can multiply each row of $\Gamma\{ \widehat{ D}_0 , \widehat{ D}(f) \} $ by $1$ or $-1$ and still satisfy \eqref{eq:bss:on:Z}.

	Let
	\[
	T_1 =
	\begin{pmatrix}
	{n}^{1/2} \mathrm{vect} (   \widehat{ \Gamma}  -  I_p )\\
	{n}^{1/2} \mathrm{diag} \{   \widehat{ \Lambda} -  D(f) \} 
	\end{pmatrix}
	\]
	and
	\[
	T_2 =
	\begin{pmatrix}
	A &  B \\
	C &  D
	\end{pmatrix}
	\begin{pmatrix}
	{n}^{1/2} \{  \mathrm{vect} (   \widehat{ D}_0 - D_0 ) \} \\
	{n}^{1/2} [  \mathrm{vect} \{   \widehat{ D}(f) -  D(f)  \} ]
	\end{pmatrix}.
	\]
	Assume that $ T_1 -  T_2 \not \to 0$ in probability when $n \to \infty$.
	Then there exist $\epsilon >0$ and a subsequence $n_m \to \infty$ so that along $n_m$
	\begin{equation} \label{eq:for:gamma:lambda:ABCD:absurd}
	P( | T_1 -  T_2| \geq \epsilon ) \geq \epsilon.
	\end{equation}
	
	One can show, as for the proof of Proposition \ref{prop:consistency}, that $\limsup \lambda_1\{  D(f) \} < + \infty$ when $n \to \infty$. Hence, up to extracting a further subsequence, we can assume that when $n_m \to \infty$, $ D(f) \to D_{\infty}(f)$, where $ D_{\infty}(f)$ has {distinct}, decreasing, eigenvalues.
	
	From Lemma 4.3 in \cite{sun2002strong}, since $D_0^{-1/2} D_{\infty}(f) D_0^{-1/2} = D_{\infty}(f)$ is diagonal, there exists a sequence of random orthogonal matrices $U_n$ such that $U_{n_m} \widehat{D}_0^{-1/2} \widehat{D}(f) \widehat{D}_0^{-1/2} U_{n_m}^\T = \Lambda_{n_m}$ is diagonal and goes to $ D_{\infty}(f)$ in probability and so that $U_{n_m} \to I_p$ in probability when $n_m \to \infty$.  
	Hence, the pair $(U_{n_m} \widehat{D}_0^{-1/2} , \Lambda_{n_m})$ satisfies \eqref{eq:bss:on:Z} with probability going to one. Furthermore, all the matrices $\Gamma\{ \widehat{ D}_0 , \widehat{ D}(f) \}$ for which $\Gamma\{ \widehat{ D}_0 , \widehat{ D}(f) \}$ and $\widehat{\Lambda}$ satisfy \eqref{eq:bss:on:Z} satisfy $\Gamma\{ \widehat{ D}_0 , \widehat{ D}(f) \}  = S U_{n_m} \widehat{D}_0^{-1/2}$ where $S$ is a diagonal matrix with diagonal elements equal to $-1$ or $1$. Hence with probability going to one, we must have $S = I_p$ for \eqref{eq:unicity:condition} to be also satisfied. Hence, with probability going to one,  $\widehat{\Gamma} = U_{n_m} \widehat{D}_0^{-1/2}$. 
	Hence we have finally obtained $\widehat{\Gamma} \to I_p$  and $| \widehat{\Lambda} - D(f) | \to 0 $ in probability when $n_m \to \infty$.

	The rest of the proof is similar to those given in \cite{ilmonen2010characteristics} and \cite{MiettinenNordhausenOjaTaskinen2012}.
	By definition of $\widehat{\Gamma}$ and $\widehat{\Lambda}$, we have
	\begin{equation*}
	\widehat{\Gamma} \widehat{D}_0 \widehat{\Gamma}^\T = I_p
	~ ~ \mbox{and} ~ ~
	\widehat{\Gamma} \widehat{D}(f) \widehat{\Gamma}^\T = \widehat{\Lambda}.
	\end{equation*}
	
	Hence
	\begin{eqnarray*}
		& ( \widehat{\Gamma} - I_p ) \widehat{D}_0 \widehat{\Gamma}^\T
		+ ( \widehat{D}_0 - I_p ) \widehat{\Gamma}^\T + ( \widehat{\Gamma} - I_p)^\T = 0  ~ ~ \mbox{and} \\
		& ( \widehat{\Gamma} - I_p ) \widehat{D}(f) \widehat{\Gamma}^\T
		+ \{\widehat{D}(f) - D(f)\} \widehat{\Gamma}^\T + D(f) ( \widehat{\Gamma} - I_p)^\T = \widehat{\Lambda} - D(f).
	\end{eqnarray*}
	Also, from {Lemma~\ref{lem:TCL:hat:D}}, we have 
	${n_m}^{1/2} ( \widehat{D}_0 - I_p) = O_p(1)$ and ${n_m}^{1/2} \{\widehat{D}(f) - D(f)\} = O_p(1)$. Thus, we get
	\begin{eqnarray*}
		{n_m}^{1/2} ( \widehat{D}_0 - I_p )
		&=& - {n_m}^{1/2} ( \widehat{\Gamma} - I_p)
		- {n_m}^{1/2}  ( \widehat{\Gamma} - I_p)^\T + o_p(1)
		~ ~ \mbox{and} \\
		{n_m}^{1/2} \{ \widehat{D}(f) - D(f) \}
		&= &- {n_m}^{1/2} ( \widehat{\Gamma} - I_p)  D(f)
		- {n_m}^{1/2} D(f)  ( \widehat{\Gamma} - I_p)^\T
		+ {n_m}^{1/2}  \{ \widehat{\Lambda} - D(f) \}
		+ o_p(1).
	\end{eqnarray*}
	This then yields
	\begin{eqnarray*}
		& {n_m}^{1/2} ( \widehat{\Gamma}_{ii} - 1 )
		= - \frac{1}{2} {n_m}^{1/2} ( \widehat{D}_{0,i,i} - 1)
		+ o_p(1) \\
		& (\lambda_i - \lambda_j)  {n_m}^{1/2} \widehat{\Gamma}_{i,j}
		=  {n_m}^{1/2} \widehat{D}(f)_{i,j}
		- \lambda_i {n_m}^{1/2} \widehat{D}_{0,i,j} + o_p(1),
		~ ~ \mbox{$i \neq j$, and} \\
		&  {n_m}^{1/2} ( \widehat{\Lambda}_{i,i} - \lambda_i )
		=  {n_m}^{1/2} \{\widehat{D}(f)_{i,i} - \lambda_i  \}
		- \lambda_i {n_m}^{1/2} ( \widehat{D}_{0,i,i} - 1) + o_p(1).
	\end{eqnarray*}
	This is in contradiction with \eqref{eq:for:gamma:lambda:ABCD:absurd}, by definition of $A$, $B$, $C$ and $D$. Hence the proof is finished.
\end{proof}

From {Lemmas~\ref{lem:TCL:hat:D} and \ref{lem:DL:proba:for:TCL}}, we now obtain a central limit theorem for $\Gamma\{\widehat{ D}_0 , \widehat{ D}(f)  \} - I_p $ and $\Lambda\{\widehat{ D}_0 , \widehat{ D}(f)  \} - \Lambda(f)$. Note that $\Lambda(f) = D(f)$.

\begin{lemma} \label{lem:TCL:deux:matrices:D}
	Assume the same conditions as in {Lemma~\ref{lem:DL:proba:for:TCL}} and let $n_0$ be defined as in {Lemma~\ref{lem:DL:proba:for:TCL}}. Let, for $n \geq n_0$,
	\[
	G = \begin{pmatrix}
	A & B \\
	C & D
	\end{pmatrix},
	\]
	from {Lemma~\ref{lem:DL:proba:for:TCL}}. 
	For $\Gamma\{\widehat{D}_0 , \widehat{ D}(f) \}$ and $\Lambda\{\widehat{D}_0 , \widehat{ D}(f) \}$ satisfying \eqref{eq:bss:on:Z},
	let
	\[
	X_n = {n}^{1/2}
	\begin{pmatrix}
	\mathrm{vect} [
	\Gamma\{\widehat{D}_0 , \widehat{ D}(f) \}
	-
	I_p
	]
	\\
	\mathrm{diag} [
	\Lambda\{\widehat{M}_0 , \widehat{ M}(f) \}
	-
	D(f)
	]
	\end{pmatrix}.
	\]
	Let $Q_n$ be the distribution of $X_n$.
	Let $\tilde{V}(f)$ be 
	defined as in {Lemma~\ref{lem:TCL:hat:D}}.
	Then, we can choose $\Gamma\{\widehat{D}_0 , \widehat{ D}(f) \}$ and $\Lambda\{ \widehat{D}_0 , \widehat{ D}(f) \}$ satisfying \eqref{eq:bss:on:Z} such that when $n\to \infty$,
	\[
	d_w\{Q_n , \mathcal{N}( 0,G \tilde{V}(f) G^\T ) \} \to 0.
	\]
\end{lemma}

\begin{proof}
	The {lemma} is a direct consequence of {Lemmas~\ref{lem:TCL:hat:D} and \ref{lem:DL:proba:for:TCL}}. The proof is carried out by contradiction, by taking subsequences along which the bounded sequences of matrices $G$ and $\tilde{V}(f)$ converge, and by applying Slutsky's lemma.
\end{proof}

We now use the equivariance property \eqref{eq:equivariance:property} to conclude.

\begin{proposition} \label{prop:TCL:deux:matrices}
	Assume the same conditions as in {Lemma~\ref{lem:DL:proba:for:TCL}}. 
	Let $\tilde{V}(f)$ and  $G$ be defined as in {Lemmas~\ref{lem:TCL:hat:D} and \ref{lem:TCL:deux:matrices:D}}.
	Let $M_{\Omega^{-1}}$ be the $p^2 \times p^2$ matrix defined by
	\[
	( M_{\Omega^{-1}} )_{a,b}
	=
	\begin{cases}
	( \Omega^{-1} )_{ J_p(b) , J_p(a) } & ~ ~ \mbox{if} ~ ~ I_p(a) = I_p(b), \\
	0 & ~ ~ \mbox{if} ~ ~ I_p(a) \neq I_p(b).
	\end{cases}
	\]
	Let $\bar{M}_{\Omega^{-1}}$ be the matrix of size $(p^2+p) \times (p^2 + p)$ defined by
	\[
	\bar{M}_{\Omega^{-1}}
	=
	\begin{pmatrix}
	M_{\Omega^{-1}} & 0 \\
	0 & I_p
	\end{pmatrix}.
	\]
	For $\Gamma\{\widehat{M}_0,\widehat{M}(f)\}$ and $\Lambda\{\widehat{M}_0 , \widehat{ M}(f) \}$ satisfying Definition \ref{sBSSunmixest}, let
	\[
	X_n = {n}^{1/2}
	\begin{pmatrix}
	\mathrm{vect} [
	\Gamma\{\widehat{M}_0 , \widehat{ M}(f) \}
	-
	\Omega^{-1}
	]
	\\
	\mathrm{diag} [
	\Lambda\{\widehat{M}_0 , \widehat{ M}(f) \}
	-
	\Lambda(f)
	]
	\end{pmatrix}.
	\]
	Let $Q_n$ be the distribution of $X_n$. Let, for $n \geq n_0$ with $n_0$ as in {Lemma~\ref{lem:DL:proba:for:TCL}},
	\[
	F = \bar{M}_{\Omega^{-1}} G \tilde{V}(f) G^\T \bar{M}_{\Omega^{-1}}^\T.
	\]
	Then, we can choose $\Gamma\{\widehat{M}_0 , \widehat{ M}(f) \},\Lambda\{ \widehat{M}_0 , \widehat{ M}(f) \}$, satisfying Definition \ref{sBSSunmixest}, so that when $n\to \infty$,
	\[
	d_w\{Q_n , \mathcal{N}( 0,F ) \} \to 0.
	\]
\end{proposition}

\begin{proof}
	The proof directly follows from {Lemma~\ref{lem:TCL:deux:matrices:D}} and from \eqref{eq:equivariance:property}. Indeed, for $\Gamma\{\widehat{D}_0 , \widehat{ D}(f) \}$ and $\Lambda\{ \widehat{D}_0 , \widehat{ D}(f) \}$ satisfying the central limit theorem in {Lemma~\ref{lem:TCL:deux:matrices:D}}, we can choose $\Gamma\{\widehat{M}_0 , \widehat{ M}(f) \}$ and $\Lambda\{ \widehat{M}_0 , \widehat{ M}(f) \}$ satisfying Definition \ref{sBSSunmixest} and such that
	\[
	\begin{pmatrix}
	\mathrm{vect} [
	\Gamma\{\widehat{M}_0 , \widehat{ M}(f) \}
	-
	\Omega^{-1}
	]
	\\
	\mathrm{diag} [
	\Lambda\{\widehat{M}_0 , \widehat{ M}(f) \}
	-
	D(f)
	]
	\end{pmatrix}
	=
	\bar{M}_{\Omega^{-1}}
	\begin{pmatrix}
	\mathrm{vect} [
	\Gamma\{\widehat{D}_0 , \widehat{ D}(f) \}
	-
	I_p
	]
	\\
	\mathrm{diag} [
	\Lambda\{\widehat{D}_0 , \widehat{ D}(f) \}
	-
	D(f)
	]
	\end{pmatrix}.
	\]
	We also remark that $\Lambda(f) = D(f)$.
\end{proof}

\subsection{{Proof of Proposition  \ref{prop:identifiability:k:matrices}}}
\label{supplement:proof:identifiability:two}

{
	We let $D(f) = \Omega^{-1} M(f) \Omega^{-\T}$ for $f: \mathbb{R}^d \to \mathbb{R}$.
	We restate Proposition \ref{prop:identifiability:k:matrices} and prove it.}

\begin{proposition} \label{prop:identifiability:k:matrices:in:supplement}
	{
		The unmixing problem given by $f_1,\ldots,f_k$ is identifiable if and only if for every pair $ i \neq j $, $i, j = 1, \ldots, p$, there exists $ l = 1, \ldots , k$ such that $   \{\Omega^{-1} M(f_l) \Omega^{- \T}\}_{i,i} \neq  \{\Omega^{-1} M(f_l) \Omega^{- \T}\}_{j,j}  $.}
\end{proposition}

\begin{proof}[{of Proposition \ref{prop:identifiability:k:matrices:in:supplement} (Proposition \ref{prop:identifiability:k:matrices})}]
	{We have that $\Gamma$ satisfies \eqref{gamma:k:matrices} if and only if
		\begin{equation} \label{eq:Gamma:Omega:k}
		\Gamma \Omega
		\in 
		\argmax_{
			\substack{
				{L}: {L} {L} ^\T = I_p  \\
				{L} \mbox{\small \; has rows } l_1^\T,\ldots,l_p^\T
			} 
		}
		\sum_{l=1}^k \sum_{j=1}^p \{ l_j^\T D(f_l) l_j \}^2.
		\end{equation}
		If the condition of the proposition holds, then only orthogonal matrices of the form $PS$ satisfy \eqref{eq:Gamma:Omega:k}, with the double sum being equal to 
		\[
		\sum_{l=1}^k \sum_{j=1}^p  D(f_l)_{j,j}^2,
		\]
		where $P$ is a permutation matrix and $S$ is a diagonal matrix with diagonal elements equal to $-1$ or $1$, see the end of the proof of {Lemma \ref{lem:consistency:k:matrices}} below.
		Assume now that there exist $ i \neq j $, $i, j = 1, \ldots, p$, such that for $ l = 1, \ldots , k$, $   \{\Omega^{-1} M(f_l) \Omega^{- \T}\}_{i,i} =  \{\Omega^{-1} M(f_l) \Omega^{- \T}\}_{j,j}  $. Without loss of generality, assume that $i=1$ and $j=2$.
		Then consider the $p \times p$ block diagonal matrix $Q$ with first $2 \times 2$ block equal to $\{(2^{-1/2},2^{-1/2})^\T , (2^{-1/2},-2^{-1/2}){^\T}\}$ and second  $(p-2) \times (p-2)$ block equal to $I_{p-2}$. Then one can show that $Q$ satisfies \eqref{eq:Gamma:Omega:k}.
		In this case, $\Gamma = Q \Omega^{-1} $ satisfies \eqref{gamma:k:matrices}, and is not of the form $P S \Omega^{-1}$.}
\end{proof}

\subsection{Asymptotics when diagonalising more than two matrices}
\label{appendix:more:two:matrices}

\begin{lemma} \label{lem:consistency:k:matrices}
	Let Conditions \ref{cond:minimal:distance} and \ref{cond:bound:covariance} hold.
	Let $k \in \mathbb{N}$ be fixed. Let $f_1,\ldots,f_k: \mathbb{R}^d \to \mathbb{R}$ satisfy Condition \ref{cond:bound:weight:function}. Assume that Assumption \ref{assumption:identifiability:k:matrices} holds. Let $\widehat{\Gamma} = \widehat{\Gamma}\{\widehat{D}_0,\widehat{D}(f_1),\ldots,\widehat{D}(f_k)\}$ be such that
	\begin{equation} \label{eq:hat:gamma:k:matrices:in:appendix}
	\widehat{\Gamma}
	\in 
	\argmax_{
		\substack{
			\Gamma: \Gamma \widehat{D}_0 \Gamma^\T = I_p  \\
			\Gamma \mbox{ has rows } \gamma_1^\T,\ldots,\gamma_p^\T
		} 
	}
	\sum_{l=1}^k \sum_{j=1}^p \{\gamma_j^\T \widehat{D}(f_l) \gamma_j \}^2.
	\end{equation}
	Then we can choose $\widehat{\Gamma}$ so that $\widehat{\Gamma} \to  I_p$ in probability when $n \to \infty$.
\end{lemma}

\begin{proof}
	Let, for $U$ a $p \times p$ orthogonal matrix with rows $u_1^\T,\ldots,u_p^\T$,
	\[
	\widehat{g}(U) = \sum_{l=1}^k \sum_{j=1}^p
	\{ u_j^\T \widehat{D}_0^{-1/2} \widehat{D}(f_l) \widehat{D}_0^{-1/2} u_j \}^2.
	\]
	Let
	\begin{eqnarray*}
		E_0  & =&
		\{
		U \mbox{ orthogonal with rows $u_1^\T$,\ldots,$u_p^\T$};\\\
		& & 
		\sum_{j=1}^p j (u_1)_j^2 \leq \cdots \leq \sum_{j=1}^p j (u_p)_j^2
		\mbox{ and for } \; i=1,\ldots,p, \sum_{j=1}^p U_{i,j} \geq 0
		\}.
	\end{eqnarray*}
	
	We observe that any orthogonal matrix can be obtained from a matrix in $E_0$, by row permutation and row multiplication by $1$ or $-1$. Hence, for any $n$, there exists $\widehat{U}$ so that $\widehat{U} \in \argmax_{U \in E_0} \widehat{g}(U)$ and $\widehat{U} \widehat{D}_0^{-1/2}$ satisfies \eqref{eq:hat:gamma:k:matrices:in:appendix}.
	
	We now aim at showing that
	$\widehat{U} \to I_p$ in probability as $n \to \infty$, which will conclude the proof since $\widehat{D}_0 \to I_p$ in probability.
	Assume that this is not the case. Then, there exists $\epsilon >0$ and a subsequence $(n_m)_{m \in \mathbb{N}}$ so that for all $m \in \mathbb{N}$ and along $n_m$
	\begin{equation} \label{eq:consistency:morethantwo:hatU}
	\mathrm{pr}( \| \widehat{U} - I_p  \|_F \geq \epsilon ) \geq \epsilon.
	\end{equation}
	The matrices $D(f_1),\ldots,D(f_l)$ are bounded (this can be shown as in Proposition \ref{prop:consistency}). Hence, by compacity, up to extracting a further subsequence, we have that \eqref{eq:consistency:morethantwo:hatU} holds along $n_m$ and, as $m \to \infty$ and along $n_m$, $D(f_1) \to D_{\infty}(f_1),\ldots,D(f_k) \to D_{\infty}(f_k)$.
	
	We let 
	\[
	g_{\infty}(U) = \sum_{l=1}^k \sum_{j=1}^p \{u_j^\T D_{\infty}(f_l) u_j \}^2.
	\]
	We have, from Proposition \ref{prop:consistency} and as observed in \cite{MiettinenIllnerNordhausenOjaTaskinenTheis2016}, that, as $m \to \infty$ and along $n_m$, 
	\[
	\sup_{ U \in U_0 } \left| \widehat{g}(U) - g_{\infty}(U) \right|
	\to 0
	\]
	in probability as $m \to \infty$.
	Hence, using a standard M-estimator argument and because $E_0$ is compact, if the unique maximum of $g_{\infty}$ on $E_0$ is $I_p$, we obtain that, as $m \to \infty$ and along $n_m$, $\widehat{U} \to I_p$ in probability. This is contradictory to \eqref{eq:consistency:morethantwo:hatU}.
	
	Hence, to conclude the proof, it suffices to show that the unique maximum of $g_{\infty}$ on $E_0$ is $I_p$.
	We have
	\begin{align} \label{eq:gU:frobenius}
	g_{\infty}(U) & =
	\sum_{l=1}^k
	\| U^\T D_{\infty}(f_l) U \|_F^2
	-
	\sum_{i \neq j} \{U^\T D_{\infty}(f_l) U  \}_{i,j}^2
	\\
	& \leq
	\sum_{l=1}^k \| U^\T D_{\infty}(f_l) U \|_F^2 
	\nonumber \\
	& = \sum_{l=1}^k \| D_{\infty}(f_l) \|_F^2. \nonumber
	\end{align}
	Also, 
	\[
	g_{\infty}(I_p) = \sum_{l=1}^k \sum_{j=1}^p D_{\infty}(f_l)_{j,j}^2
	=
	\sum_{l=1}^k \| D_{\infty}(f_l) \|_F^2.
	\]
	We next show that the identity matrix $ I_p $ is the unique maximizer of $ g_{\infty} $ in $ E_0 $. To see this, consider an arbitrary orthogonal matrix $ U $ which maximizes $ g_{\infty} $. From \eqref{eq:gU:frobenius} we see that $ U^\T D_{\infty}(f_l) U $ is a diagonal matrix for all $l = 1, \ldots, k$. Then, by its non-singularity, the matrix $ U $ must have a column with a non-zero first element. Call the first (from the left) such column of $ U $ by $ u $. We show that all other elements of $ u $ must be zero. By the previous, $ u $ is an eigenvector of all $ D_{\infty}(f_l) $ and we have,
	\[ 
	D_{\infty}(f_l) u = \psi_l u, \quad \mbox{for all } l = 1, \ldots, k,
	\]
	for some eigenvalues $ \psi_l \in \mathbb{R}$, $ l = 1, \ldots, k $. Assume then that $ u $ has a second non-zero element at some arbitrary position $ q \neq 1 $, meaning that both $ u_1, u_q \neq 0 $. Then we write
	\[ 
	D_{\infty}(f_l)_{1,1} u_1 = \psi_l u_1 \quad \mbox{and} \quad D_{\infty}(f_l)_{q,q} u_q = \psi_l u_q, \quad \mbox{for all } l = 1, \ldots, k,
	\]
	which in turn implies that $ D_{\infty}(f_l)_{1,1} = D_{\infty}(f_l)_{q,q} $ for all $ l = 1, \ldots , k $. By a continuity argument, this is a contradiction with Assumption \ref{assumption:identifiability:k:matrices}. As the choice of $ q $ was arbitrary, the only non-zero element in $ u $ is the first. Repeating now the same reasoning for other elements besides the first, we observe that each column of the maximizer $ U $ must have a single non-zero element, and by its orthogonality we have $ U = P D$ for some permutation matrix $ P $ and some diagonal matrix $ D $ with diagonal components in $ \{-1, 1\} $. The only matrix of that form belonging to $ E_0 $ is $ I_p $ and thus, for all $U \in E_0$ with $U \neq I_p$, we have $g(U) < g(I_p)$. 
\end{proof}

\begin{lemma} \label{lem:DL:k:matrices}
	Assume the same setting and conditions as in {Lemma~\ref{lem:consistency:k:matrices}}. For a diagonal matrix $\Lambda$, let $\Lambda_r = \Lambda_{r,r}$. Let $A_0,A_1,\ldots,A_k$ and $B$ be $p^2 \times p^2$ matrices defined by, for $n \geq n_0$ with the notation of Assumption \ref{assumption:identifiability:k:matrices},
	\[
	A_{0,i,j} =
	\begin{cases}
	- 1/2 & ~ ~\mbox{for} ~ ~ i=j  \in \mathcal{D}(p), \\
	- 
	\sum_{l=1}^k 
	\{
	D(f_l)_{I_p(i)} - D(f_l)_{J_p(i)}
	\}
	D(f_l)_{{I_p(i)}}
	& ~ ~\mbox{for} ~ ~ i=j \not \in \mathcal{D}(p), \\
	0  & ~ ~ \mbox{otherwise}, ~ ~
	\end{cases}
	\]
	for $l=1,\ldots,k$,
	\[
	A_{l,i,j} =
	\begin{cases} 
	D(f_l)_{I_p(i)} - D(f_l)_{J_p(i)}
	& ~ ~\mbox{for} ~ ~ i=j \not \in \mathcal{D}(p), \\
	0  & ~ ~ \mbox{otherwise}, ~ ~
	\end{cases}
	\]
	and
	\[
	B_{i,j} =
	\begin{cases}
	1 & ~ ~\mbox{for} ~ ~ i=j  \in \mathcal{D}(p), \\
	[\sum_{l=1}^k 
	\{
	D(f_l)_{I_p(i)} - D(f_l)_{J_p(i)}
	\}^2 ]^{-1}
	& ~ ~\mbox{for} ~ ~ i=j \not \in \mathcal{D}(p), \\
	0  & ~ ~ \mbox{otherwise}. ~ ~
	\end{cases}
	\]
	Then, as $n \to \infty$, there exists $\widehat{\Gamma}$ satisfying \eqref{eq:hat:gamma:k:matrices:in:appendix} so that
	\[
	{n}^{1/2}
	\mathrm{vect}
	(
	\widehat{\Gamma} - I_p
	)
	=
	B
	\begin{pmatrix}
	A_0 & A_1 & \ldots & A_k 
	\end{pmatrix}
	\begin{pmatrix}
	{n}^{1/2} \mathrm{vect} (
	\widehat{D}_0
	-
	D_0
	) \\
	{n}^{1/2} \mathrm{vect} \{
	\widehat{D}(f_1)
	-
	D(f_1)
	\} \\
	\vdots \\
	{n}^{1/2} \mathrm{vect} \{
	\widehat{D}(f_k)
	-
	D(f_k)
	\}
	\end{pmatrix}.
	\]
\end{lemma}

\begin{proof}
	Assume that $n \geq n_0$ throughout the proof.
	Let $\widehat{\Gamma}$ satisfy \eqref{eq:hat:gamma:k:matrices:in:appendix} and $\widehat{\Gamma} \to I_p$ in probability when $n\to \infty$ (the existence follows from {Lemma~\ref{lem:consistency:k:matrices}}). The proof of the {lemma} follows the proofs of ii) in Theorem 2 of \cite{MiettinenIllnerNordhausenOjaTaskinenTheis2016} and Theorem 3 in \cite{virta2018asymptotically} and as such, we present below only some key steps.  
	
	From {Lemma~\ref{lem:TCL:hat:D}}, we have 
	$n^{1/2} ( \widehat{D}_0 - I_p) = O_p(1)$ and $n^{1/2} \{\widehat{D}(f_l) - D(f_l)\} = O_p(1)$, for all $ l = 1, \ldots , k $. By a continuity argument and our assumptions, we further have $D(f_1) \to D_{\infty}(f_1),\ldots,D(f_k) \to D_{\infty}(f_k)$ such that the limit matrices satisfy: there exists a fixed $\delta >0$ so that  for every pair $ i \neq j $, $i, j = 1, \ldots, p$, there exists $ l = 1, \ldots , k$ such that $ | D_\infty(f_l)_{i,i} - D_\infty(f_l)_{j,j} | \geq \delta $. The previous convergence holds up to extracting a subsequence. We omit this step in this proof for concision, but see the proof of  {Lemma~\ref{lem:DL:proba:for:TCL}}.
	Finally, the rotation $ \widehat{U} $ so that $\widehat{\Gamma} = \widehat{U} \widehat{D}_0^{-1/2}$ also satisfies $ \widehat{U} \to I_p$ in probability.

	Then, as in \cite{virta2018asymptotically}, the maximisation problem,
	\[
	\argmax_{
		\substack{
			U: U U^\T = I_p  \\
			U \mbox{ has rows } u_1^\T,\ldots,u_p^\T }}
	\sum_{l=1}^k \sum_{j=1}^p
	\{ u_j^\T \widehat{D}_0^{-1/2} \widehat{D}(f_l) \widehat{D}_0^{-1/2} u_j \}^2,
	\]
	yields the estimation equations $ n^{1/2}\widehat{Y} = n^{1/2}\widehat{Y}^\T $, where
	\[ 
	n^{1/2}\widehat{Y} = n^{1/2}\sum_{l=1}^k \widehat{U} \widehat{R}(f_l) \widehat{U}^\T \mathrm{Diag}\{\widehat{U} \widehat{R}(f_l) \widehat{U}^\T\},
	\]
	where we have used the shorthand $ \widehat{R}(f_l) = \widehat{D}_0^{-1/2} \widehat{D}(f_l) \widehat{D}_0^{-1/2} $ and $ \mathrm{Diag}(M) $ is equal to the square matrix $ M $ but with its off-diagonal elements set to zero. 
	Linearizing the estimating equations asymptotically and vectorizing, we arrive at the following form,
	\begin{align}\label{eq:estimating_eq_1}
	\begin{split}
	& (I_p - K) \sum_{l=1}^k \left( [ \mathrm{Diag}\{\widehat{U} D(f_l) \widehat{U}^\T\} \widehat{U} D(f_l) \otimes I_p ] +  [\mathrm{Diag}\{\widehat{U} D(f_l) \widehat{U}^\T\} \otimes  D(f_l)] K \right)\\
	& \cdot n^{1/2}\mathrm{vec} ( \widehat{U} - I_p ) = - (I_p - K) n^{1/2}\mathrm{ vec}( \widehat{F} ) + o_p(1),
	\end{split}
	\end{align}
	where $ K $ is the $ p^2 \times p^2 $ commutation matrix satisfying $ K^2 = I_p $, $ n^{1/2}\widehat{F} = \sum_{l=1}^k n^{1/2}\{ \widehat{R}(f_l) - D(f_l) \} D_\infty(f_l) = O_p(1)$ and $\mathrm{vec}( M) = (c_1^\T,\ldots,c_{q}^\T)^\T$ is the column vectorization where $c_1,\ldots,c_{q}$ are the ${q}$ columns of a matrix $ M$. The orthogonality constraint can be similarly linearized to yield,
	\begin{align}\label{eq:estimating_eq_2}
	\{(\widehat{U} \otimes I_p ) + K \} n^{1/2}\mathrm{vec} ( \widehat{U} - I_p ) = 0.
	\end{align}
	Summing \eqref{eq:estimating_eq_1} and \eqref{eq:estimating_eq_2}, we obtain,
	\begin{align}\label{eq:short_form} 
	\widehat{A} n^{1/2}\mathrm{vec} ( \widehat{U} - I_p ) = - (I_p - K) n^{1/2}\mathrm{ vec}( \widehat{F} ) + o_p(1),
	\end{align}
	where $\widehat{A} \rightarrow A = (I_p - K) \{ ( \sum_{l=1}^k D_l^2 \otimes I_p ) + ( \sum_{l=1}^k D_l \otimes  D_l ) K \} + I_p + K $ in probability, where we use the notation $ D_l = D_\infty(f_l) $, $ l = 1, \ldots , k $. Using the fact that $ K (A \otimes B) K = B \otimes A $ for any conformable matrices $ A, B $, we get the alternative form,
	\[ 
	A = \left(\sum_{l=1}^k D_l^2 \otimes I_p - \sum_{l=1}^k D_l \otimes  D_l + I_p \right) + \left( \sum_{l=1}^k D_l \otimes  D_l - \sum_{l=1}^k I_p  \otimes D_l^2 + I_p  \right) K.
	\]
	Continuing as in \cite{virta2018asymptotically}, each diagonal element of $ \widehat{U} $ has a corresponding $ 1 \times 1 $ diagonal block equal to 2 in $ A $. Similarly, each pair of $ (a, b) $th and $ (b, a) $th off-diagonal elements in $ \widehat{U} $ has a corresponding $ 2 \times 2 $ sub-matrix in $ A $ of the form,
	\[ 
	A_{ab} = \begin{pmatrix}
	1 + \sum_{l=1}^k d_{la}^2 - \sum_{l=1}^k d_{la} d_{lb} & 1 - \sum_{l=1}^k d_{lb}^2 + \sum_{l=1}^k d_{la} d_{lb}\\
	1 - \sum_{l=1}^k d_{la}^2 + \sum_{l=1}^k d_{la} d_{lb} & 1 + \sum_{l=1}^k d_{lb}^2 - \sum_{l=1}^k d_{la} d_{lb},
	\end{pmatrix},
	\]
	where $ d_{la} $ is the $ a $th diagonal element of $ D_l $.
	The inverse of the sub-matrix is
	\[ 
	A_{ab}^{-1} = \left\{ 2 \sum_{l=1}^k ( d_{la} - d_{lb} )^2 \right\}^{-1}
	\begin{pmatrix}
	1 + \sum_{l=1}^k d_{lb}^2 - \sum_{l=1}^k d_{la} d_{lb} & \sum_{l=1}^k d_{lb}^2 - \sum_{l=1}^k d_{la} d_{lb} - 1\\
	\sum_{l=1}^k d_{la}^2 - \sum_{l=1}^k d_{la} d_{lb} - 1 & 1 + \sum_{l=1}^k d_{la}^2 - \sum_{l=1}^k d_{la} d_{lb},
	\end{pmatrix},
	\]
	showing that $ A $ is invertible as by our assumptions $ \sum_{l=1}^k ( d_{la} - d_{lb} )^2 \neq 0 $ for all distinct pairs $ a, b = 1, \ldots , p $. Thus, by Slutsky's theorem, we obtain from \eqref{eq:short_form} that,
	\[ 
	n^{1/2}\mathrm{vec} ( \widehat{U} - I_p ) = - A^{-1} n^{1/2}\mathrm{ vec}( \widehat{F} - \widehat{F}^\T ) + o_p(1),
	\]
	showing that, $ n^{1/2}( \widehat{U} - I_p ) = O_p(1) $. Consequently, also $n^{1/2}( \widehat{\Gamma} - I_p ) = O_p(1)$.
	
	%

	Finally, we next proceed as in the proof of ii) in Theorem 2 of \cite{MiettinenIllnerNordhausenOjaTaskinenTheis2016} to obtain that, as $n \to \infty$, 
	\[ 
	{n}^{1/2} ( \widehat{\Gamma}_{i,i} - 1 )
	= - (1/2) {n}^{1/2} ( \widehat{D}_{0,i,i} - 1)
	+ o_p(1)
	\]
	and for $i \neq j$,
	\[ 
	{n}^{1/2}  \widehat{\Gamma}_{i,j}
	=  
	\frac{
		\sum_{l=1}^k
		\{ D(f_l)_{i} - D(f_l)_{j} \}
		{n}^{1/2}\{ \widehat{D}(f_l)_{i,j} 
		-
		D(f_l)_{i} \widehat{D}_{0,i,j}  \}
	}{
		\sum_{r=1}^k
		\{D(f_r)_{i} - D(f_r)_{j} \}^2
	}
	+o_p(1).
	\]
	Hence, the {lemma} follows from the definition of $A_0,A_1,\ldots,A_k,B$.
\end{proof}

\begin{lemma} \label{lem:TCL:k:matrix}
	Assume the same settings and conditions as in {Lemma~\ref{lem:consistency:k:matrices}}.
	Let $\Sigma(f,g)$ be as in Proposition \ref{prop:TCL:for:gamma:lambda} with $R$ replaced by $\mathrm{cov}(z)$  where $z$ is the $np \times 1$ vector defined by, for $i=1,\ldots,n$, $j=1,\ldots,p$, $z_{(i-1)p + j} = Z_j( s_i)$. Let $f_0(x) = I(x=0)$.
	Let $\tilde{V}(f_1,\ldots,f_k)$ be the $(k+1)p^2 \times (k+1) p^2$ matrix, composed of $(k+1)^2$ blocks of sizes $p^2 \times p^2$ with block $(i+1),(j+1)$ equal to $\Sigma(f_i,f_j)$ for $i,j=0,\ldots,p$.
	
	Let $G$ be the $p^2 \times (k+1)p^2$ matrix defined by $G = B
	(
	A_0,  A_1,  \ldots  A_k 
	)$, for $n \geq n_0$, with the notation of {Lemma~\ref{lem:DL:k:matrices}}. Let $M_{\Omega^{-1}}$ be as in Proposition \ref{prop:TCL:deux:matrices} and let, for $n \geq n_0$,
	\[
	F = M_{\Omega^{-1}} G \tilde{V}(f_1,\ldots,f_k) G^\T M_{\Omega^{-1}}^\T.
	\]
	Then, $\widehat{\Gamma} = \widehat{\Gamma}\{\widehat{M}_0,\widehat{M}(f_1),\ldots,\widehat{M}(f_k)\}$ satisfying \eqref{eq:hat:gamma:k:matrices:in:appendix}, with $\widehat{D}_0,\widehat{D}(f_1),\ldots,\widehat{D}(f_k)$ replaced by $\widehat{M}_0,\widehat{M}(f_1),\ldots,\widehat{M}(f_k)$, can be chosen so that, with $Q_n$ the distribution of ${n}^{1/2} \mathrm{vect} ( \widehat{\Gamma} - \Omega^{-1} )$, we have as $n \to \infty$
	\[
	d_w\{Q_n , \mathcal{N}(0,F) \}
	\to 0.
	\]
\end{lemma}

\begin{proof}
	The proof is the same as that of Proposition \ref{prop:TCL:deux:matrices}. In particular, for $\widehat{\Gamma}\{ \widehat{D}_0,\widehat{D}(f_1),\ldots,\widehat{D}(f_k) \}$ satisfying \eqref{eq:hat:gamma:k:matrices:in:appendix}, the matrix $\widehat{\Gamma}\{ \widehat{D}_0,\widehat{D}(f_1),\ldots,\widehat{D}(f_k) \}\Omega^{-1}$ satisfies \eqref{eq:hat:gamma:k:matrices:in:appendix}, with $\widehat{D}_0,\widehat{D}(f_1),\ldots,\widehat{D}(f_k)$ replaced by $\widehat{M}_0,\widehat{M}(f_1),\ldots,\widehat{M}(f_k)$.
\end{proof}

\begin{proposition} \label{prop:permuting:lines:consistency}
	{Assume the same settings and conditions as in {Lemma~\ref{lem:consistency:k:matrices}}.
		Let
		$(\widehat{\Gamma}_n)_{n \in \mathbb{N}}$ be any sequence of  $p \times p$ matrices so that for any $n \in \mathbb{N}$, $\widehat{\Gamma}_n = \widehat{\Gamma}_n\{\widehat{M}_0,\widehat{M}(f_1),\ldots,\widehat{M}(f_k)\}$  satisfies \eqref{eq:hat:gamma:k:matrices:in:appendix}, with $\widehat{D}_0,\widehat{D}(f_1),\ldots,\widehat{D}(f_k)$ replaced by $\widehat{M}_0,\widehat{M}(f_1),\ldots,\widehat{M}(f_k)$. Then, there exists a sequence of permutation matrices $(P_n)$ and a sequence of diagonal matrices $(D_n)$, with diagonal components in $\{-1,1\}$ so that, with $\check{\Gamma}_n = D_n P_n \widehat{\Gamma}_n$, the sequence $(\check{\Gamma}_n)$ satisfies the conclusions of {Lemma \ref{lem:consistency:k:matrices}}, with the limit $I_p$ replaced by $\Omega^{-1}$.}
\end{proposition}

\begin{proof}
	{With the notation of the proof of {Lemma \ref{lem:consistency:k:matrices}}, for $\widehat{\Gamma}_n$ satisfying \eqref{eq:hat:gamma:k:matrices:in:appendix}, there exist $P_n, D_n$, as described in the proposition, so that $ D_n P_n \widehat{\Gamma}_n \widehat{D}_0^{1/2} \in E_0$ and $ D_n P_n \widehat{\Gamma}_n $ satisfies \eqref{eq:hat:gamma:k:matrices:in:appendix}. Hence, with the same argument as in the proof of the last part of {Lemma \ref{lem:consistency:k:matrices}}, we have $D_n P_n \widehat{\Gamma}_n \to I_p$ in probability as $n \to \infty$. Furthermore, as in the proof of {Lemma \ref{lem:TCL:k:matrix}}, we can show that $D_n P_n \widehat{\Gamma}_n \Omega^{-1}$ satisfies the conclusion of this {lemma}. The proof is concluded by observing that any matrix $\bar{\Gamma}$ satisfies \eqref{eq:hat:gamma:k:matrices:in:appendix}, with $\widehat{D}_0,\widehat{D}(f_1),\ldots,\widehat{D}(f_k)$ replaced by $\widehat{M}_0,\widehat{M}(f_1),\ldots,\widehat{M}(f_k)$, if and only if the corresponding matrix $\bar{\Gamma} \Omega$ satisfies \eqref{eq:hat:gamma:k:matrices:in:appendix}.}
\end{proof}

\begin{proposition} \label{prop:permuting:lines:asymptotic:normality}
	{Assume the same settings and conditions as in {Lemma \ref{lem:consistency:k:matrices}}.
		Let
		$(\widehat{\Gamma}_n)_{n \in \mathbb{N}}$ be any sequence of  $p \times p$ matrices so that for any $n \in \mathbb{N}$, $\widehat{\Gamma}_n = \widehat{\Gamma}_n\{\widehat{M}_0,\widehat{M}(f_1),\ldots,\widehat{M}(f_k)\}$  satisfies \eqref{eq:hat:gamma:k:matrices:in:appendix}, with $\widehat{D}_0,\widehat{D}(f_1),\ldots,\widehat{D}(f_k)$ replaced by $\widehat{M}_0,\widehat{M}(f_1),\ldots,\widehat{M}(f_k)$. Then, there exists a sequence of permutation matrices $(P_n)$ and a sequence of diagonal matrices $(D_n)$, with diagonal components in $\{-1,1\}$ so that, with $\check{\Gamma}_n = D_n P_n \widehat{\Gamma}_n$, the sequence $(\check{\Gamma}_n)$ satisfies the conclusions of {Lemma \ref{lem:TCL:k:matrix}}.}
\end{proposition}

\begin{proof}
	{The proof is the same as that of Proposition \ref{prop:permuting:lines:consistency}.}
\end{proof}

The results of Propositions \ref{prop:consistency:k:matricesb} and \ref{asymp_bssmulti} derive directly from Propositions \ref{prop:permuting:lines:consistency} and \ref{prop:permuting:lines:asymptotic:normality}.

\begin{lemma} \label{lem:mean}
	Let Conditions \ref{cond:minimal:distance} and \ref{cond:bound:covariance} hold.
	Let $f$ satisfy Condition \ref{cond:bound:weight:function}. Let $\bar{X} = n^{-1} \sum_{i=1}^n X( s_i)$. Let
	\[
	\widehat{M}_{\mathrm{st}}(f)
	=
	n^{-1}
	\sum_{i=1}^n
	\sum_{j=1}^n
	f( s_i- s_j)
	\{X( s_i) - \bar{X}\}
	\{X( s_j) - \bar{X}\}^\T.
	\]
	Then as $n \to \infty$
	\[
	\widehat{M}_{\mathrm{st}}(f)
	-
	\widehat{M}(f)
	=
	O_p
	\left(
	n^{-1}
	\right).
	\]
\end{lemma}

\begin{proof}
	Let ${a},l \in \{1,\ldots,p\}$ and let $f_{i,j} = f( s_i -  s_j)$. We have
	\begin{align} \label{eq:for:proof:barX}
	\widehat{M}_{\mathrm{st}}(f)_{{a},l}
	-
	\widehat{M}(f)_{{a},l}
	&=
	n^{-1}
	\sum_{i=1}^n
	\sum_{j=1}^n
	f_{i,j}
	\{
	X_{a}( s_i) X_l( s_j)
	- X_{a}( s_i) \bar{X}_l
	- \bar{X}_{a} X_l( s_j)
	+  \bar{X}_{a} \bar{X}_l
	\} \nonumber \\
	& ~ ~ - n^{-1}
	\sum_{i=1}^n
	\sum_{j=1}^n
	f_{i,j}
	X_{a}( s_i) X_l( s_j) \nonumber \\
	& =
	-\bar{X}_l \{
	n^{-1}
	\sum_{i=1}^n
	\sum_{j=1}^n
	f_{i,j} X_{a}( s_i)
	\}
	-\bar{X}_{a} \{
	n^{-1}
	\sum_{i=1}^n
	\sum_{j=1}^n
	f_{i,j} X_l( s_j)
	\}\nonumber \\
	& ~ ~
	+ 
	n^{-1}
	\sum_{i=1}^n
	\sum_{j=1}^n
	f_{i,j} \bar{X}_{a} \bar{X}_l.
	\end{align}
	Now, for $q=1,\ldots,p$, ${E}(\bar{X}_q) = 0$ and
	\[
	\mathrm{var} (\bar{X}_q) 
	= 
	\frac{1}{n^2}
	\sum_{i=1}^n
	\sum_{j=1}^n
	\mathrm{cov}\{ X_q( s_i),X_q( s_j) \}.
	\]
	Also, $\max_{i=1,\ldots,n} \sum_{j=1}^n | \mathrm{cov}\{X_q( s_i),X_q( s_j)\} | $ is bounded because of \eqref{eq:bound:cov:fun:X} and Lemma 4 in \cite{furrer16asymptotic}. Hence $\mathrm{var}(\bar{X}_q) =O(1/n)  $ and so $\bar{X}_q = O_p(n^{-1/2})$.
	
	Also, let 
	\[
	\epsilon_q
	=
	n^{-1}
	\sum_{i=1}^n
	\sum_{j=1}^n
	f_{i,j} X_q( s_i). 
	\]
	Then ${E}(\epsilon_q) = 0$ and 
	\[
	\mathrm{var}(\epsilon_q)
	=
	n^{-2}\sum_{i=1}^n
	\sum_{{b}=1}^n
	(\sum_{j=1}^n f_{i,j}
	)
	(
	\sum_{j=1}^n f_{{b},j}
	)
	\mathrm{cov}\{X_q( s_i),X_q( s_{b}) \}.
	\]
	From Condition \ref{cond:bound:weight:function} and Lemma 4 in \cite{furrer16asymptotic}, there exists a finite constant $H$ so that
	\[
	\max_{i=1,\ldots,n}
	\sum_{j=1}^n | f_{i,j} |
	\leq H.
	\]
	Hence
	\begin{align*}
	\mathrm{var}(\epsilon_q)
	& \leq
	H^2 n^{-2}
	\sum_{i=1}^n
	\sum_{{b}=1}^n
	\left|
	\mathrm{cov}\{ X_q( s_i),X_q( s_{b}) \}
	\right| \\
	& = O \left( n^{-1} \right)
	\end{align*}
	as before. Hence $\epsilon_q = O_p(n^{-1/2})$.
	Also, we have seen above that 
	\[
	n^{-1}
	\sum_{i=1}^n
	\sum_{j=1}^n
	|f_{i,j}|
	\]
	is bounded.
	Hence, from \eqref{eq:for:proof:barX}, we conclude the proof of the lemma.
\end{proof}

\section{Simulation complements}
\subsection{Asymptotic approximate distribution of the unmixing matrix estimator}
In Fig.~\ref{fig:simulation_1_1} of Section~\ref{subsec:simu_1}, we used only the expected value of the {asymptotic approximation}. In Fig.~\ref{fig:simulation_1_2} we have plotted the estimated densities of $n (p - 1) \MDI(\widehat{{\Gamma}})^2$, in solid lines, against the {densities of the corresponding asymptotic approximations}, in dashed lines, for all local covariance matrices and a few selected sample sizes. The density functions of the {asymptotic approximations} are estimated from a sample of 100,000 random variables drawn from the corresponding distributions. Overall, the two densities fit each other rather well, especially for the local covariance matrices involving the ring kernel. This shows that the asymptotic distribution of $n (p - 1) \MDI(\widehat{{\Gamma}})^2$ is a good approximation of the true distribution already for small sample sizes.

\begin{figure}[t]
	\centering
	\includegraphics[width=\textwidth]{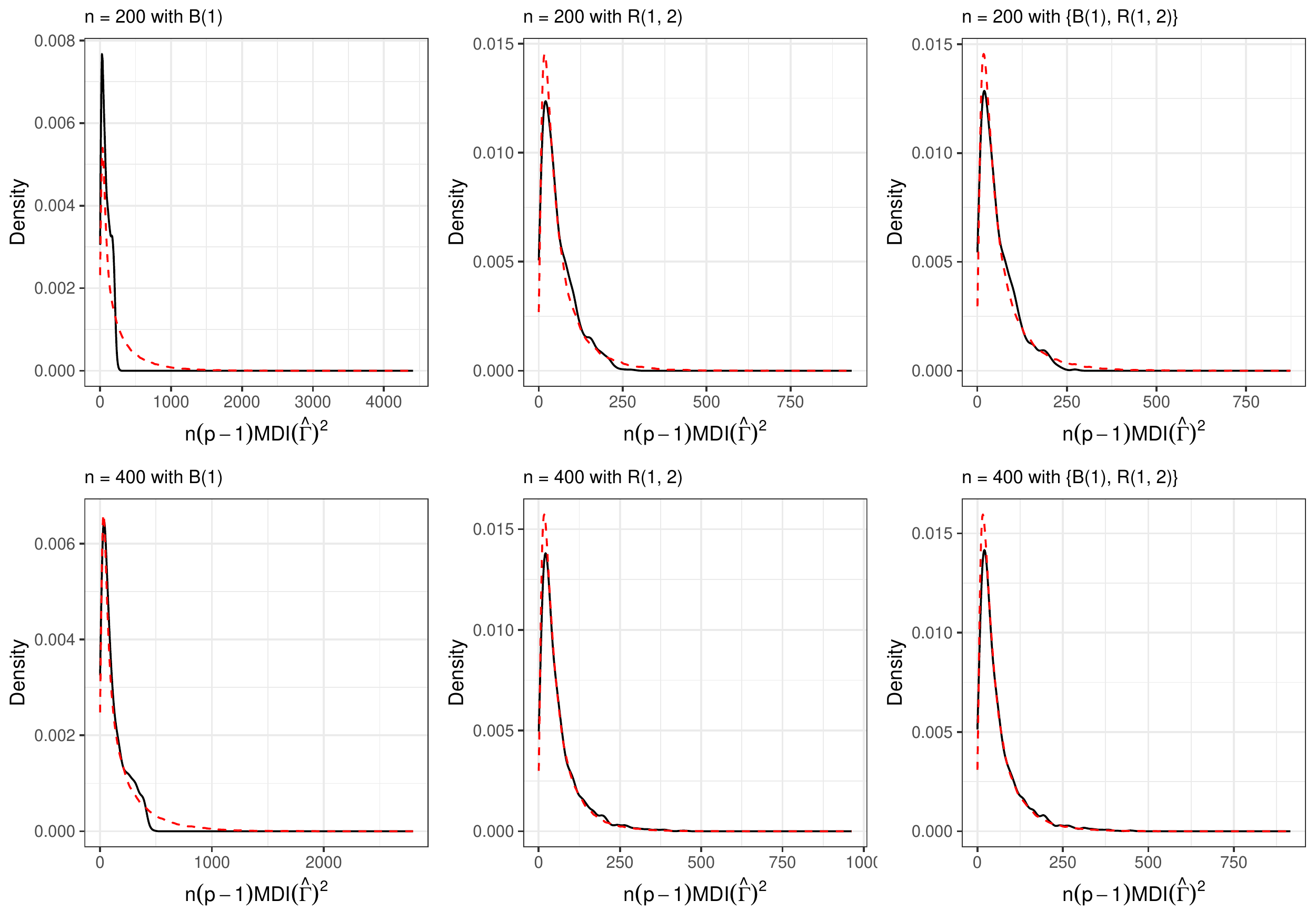}
	\caption{The density of $n (p - 1) \MDI(\widehat{{\Gamma}})^2$ over 2000 replications, in solid lines, against the density of its {asymptotic approximation}, in dashed lines, for different combinations of sample size and local covariance matrices.}
	\label{fig:simulation_1_2}
\end{figure}

\subsection{A simulation study on individual component estimation accuracy}
\label{SuppSimuComp}

In this simulation, we investigate the estimation accuracies of the individual latent fields under the spatial blind source separation model. 

We use the same setting as in the first simulation study in Section \ref{Simu}. That is, let $X({s}) = {\Omega} {Z}({s}) $ where each of the three independent latent fields has a Mat\'ern covariance function with shape and range parameters $(\kappa,\phi) \in \{(6,\text{1$\cdot$2}),(1,\text{1$\cdot$5}),(\text{0$\cdot$25},1)\}$, illustrated in the left panel in Fig.~\ref{fig:cov_functions}. The location pattern is taken to be the same growing pattern of nested squares depicted on the left of Fig.~\ref{fig:grids}.

To quantify the estimation accuracies of the individual components, we use the same strategy as in the real data example of Section \ref{Example}. Let $ \widehat{Z}(s) = \{\widehat{Z}_1(s), \widehat{Z}_2(s), \widehat{Z}_3(s)\}^\T = \widehat{ \Gamma} X({s}) $ contain the estimated components for a single repetition of the simulation. For each of the three true sources, $ j = 1, 2, 3 $, we record the maximum absolute sample correlation between $Z_j(s)$ and $\widehat{Z}_l(s) $ over $ l = 1, 2, 3 $. The larger the maximum absolute correlation, the better the source field $ j $ was estimated.

Due to the affine equivariance of the estimators, the estimated components $ \widehat{ \Gamma} X({s}) $ are invariant to the choice of $ \Omega $, up to their signs and order. 
More precisely, let $\widehat{\Gamma}(I_p)$ be computed from $\{Z(s_i)\}_{i=1,\ldots,n}$ according to \eqref{hat:gamma:k:matrices}, let $\widehat{Z}_{I_p}(s) = \widehat{\Gamma}(I_p) Z(s)$ and recall that $\widehat{\Gamma}$ is computed from $\{X(s_i)\}_{i=1,\ldots,n}$ according to \eqref{hat:gamma:k:matrices}. 
Then we have
\[
\widehat{Z}(s)
=
\widehat{\Gamma} X(s)
=
\widehat{\Gamma} \Omega Z(s)
=
\widehat{\Gamma}(I_p) \Omega^{-1} \Omega Z(s)
=
\widehat{\Gamma}(I_p) Z(s)
=
\widehat{Z}_{I_p}(s),
\]
up to order and signs of the components.
Therefore, it is without loss of generality that we may again assume that $ \Omega  = I_3 $. Any other choice of $ \Omega $ would lead to exactly the same results.

The mean maximum source-wise absolute correlations over 2000 repetitions are shown in 
Fig.~\ref{fig:corSimu} for a range of sample sizes. We have used two choices of kernels, $ R(1, 2) $ and $ R(7, 9) $. The first one was chosen due to its good performance in the main simulation study and the second one to see how the individual components are estimated under a bad choice of kernel. 

\begin{figure}[t]
	\centering
	\includegraphics[width=1.0\textwidth]{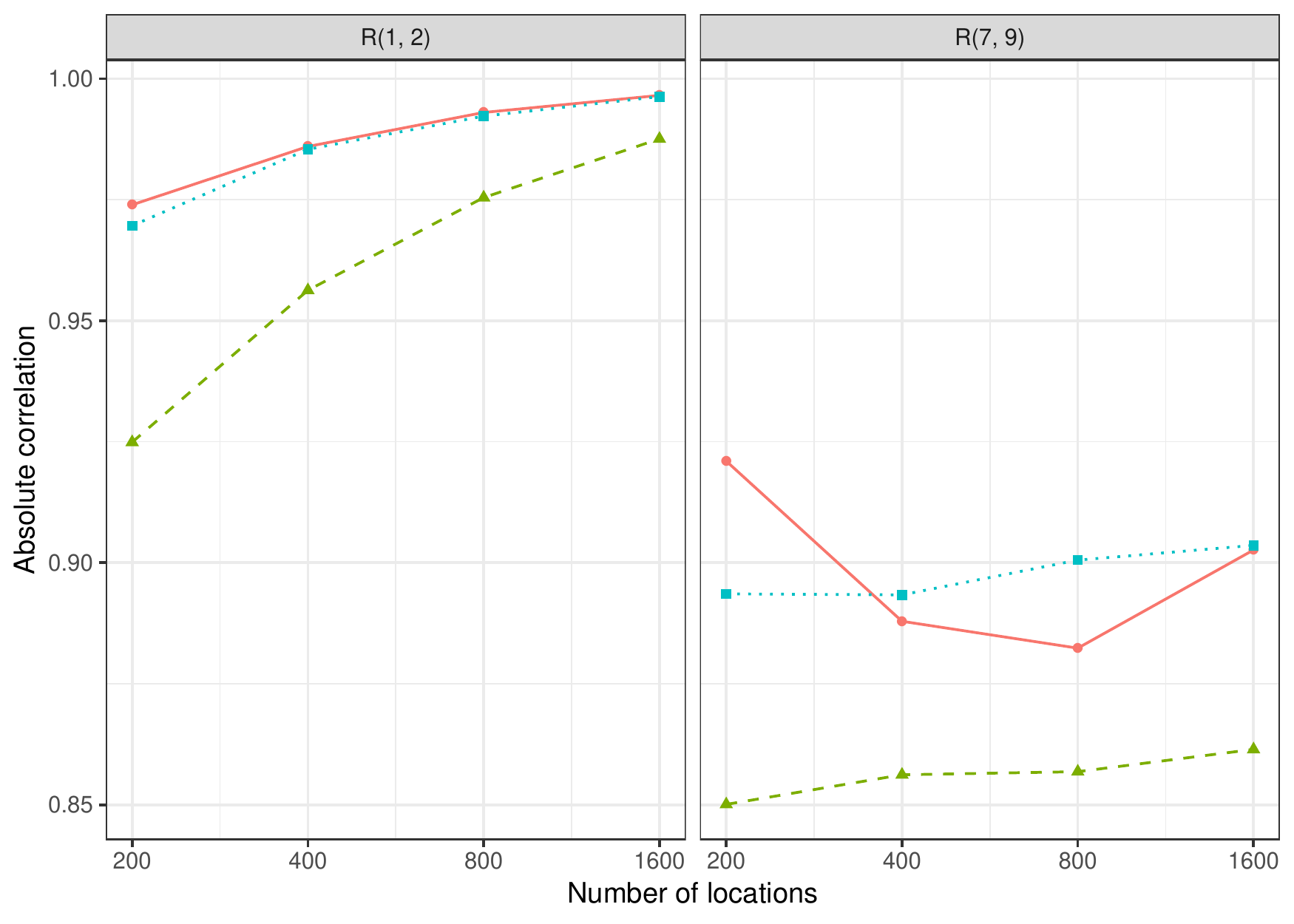}
	\caption{Maximal absolute correlations for {the first (solid red line), second (dashed green line) and third (dotted blue line) source fields} under two kernels and a range of sample sizes.}
	\label{fig:corSimu}
\end{figure}

The results indicate that the first and third sources are estimated almost equally well, but that the second source is somewhat more difficult to estimate. We postulate that this is caused by its corresponding covariance function being the middle one in the left panel in 
Fig.~\ref{fig:cov_functions}. That is, the first and third sources are unlikely to be mixed with each other due to their extremal covariance functions. The second source, on the other hand, is between the other two and can be mistaken for both the first and the third source. We also note in the right panel of {Fig.~\ref{fig:corSimu}} that using an inferior choice of a kernel leads to an overall worse estimation accuracy for all sources.
Finally, on the left-hand side of Fig.~\ref{fig:corSimu}, we observe a convergence to one of the maximum absolute correlation, under the appropriate choice of kernel $ R(1, 2) $.

\section{Further details concerning the real data example}
\label{section:details:real:data}

Figure~\ref{fig:KolaInfo} describes the sampling area from the Kola data and Figures~\ref{fig:B50}-\ref{fig:B100} visualize the components discussed in Section~\ref{Example}.  

\begin{figure}[t]
	\centering
	\includegraphics[width=0.6\textwidth]{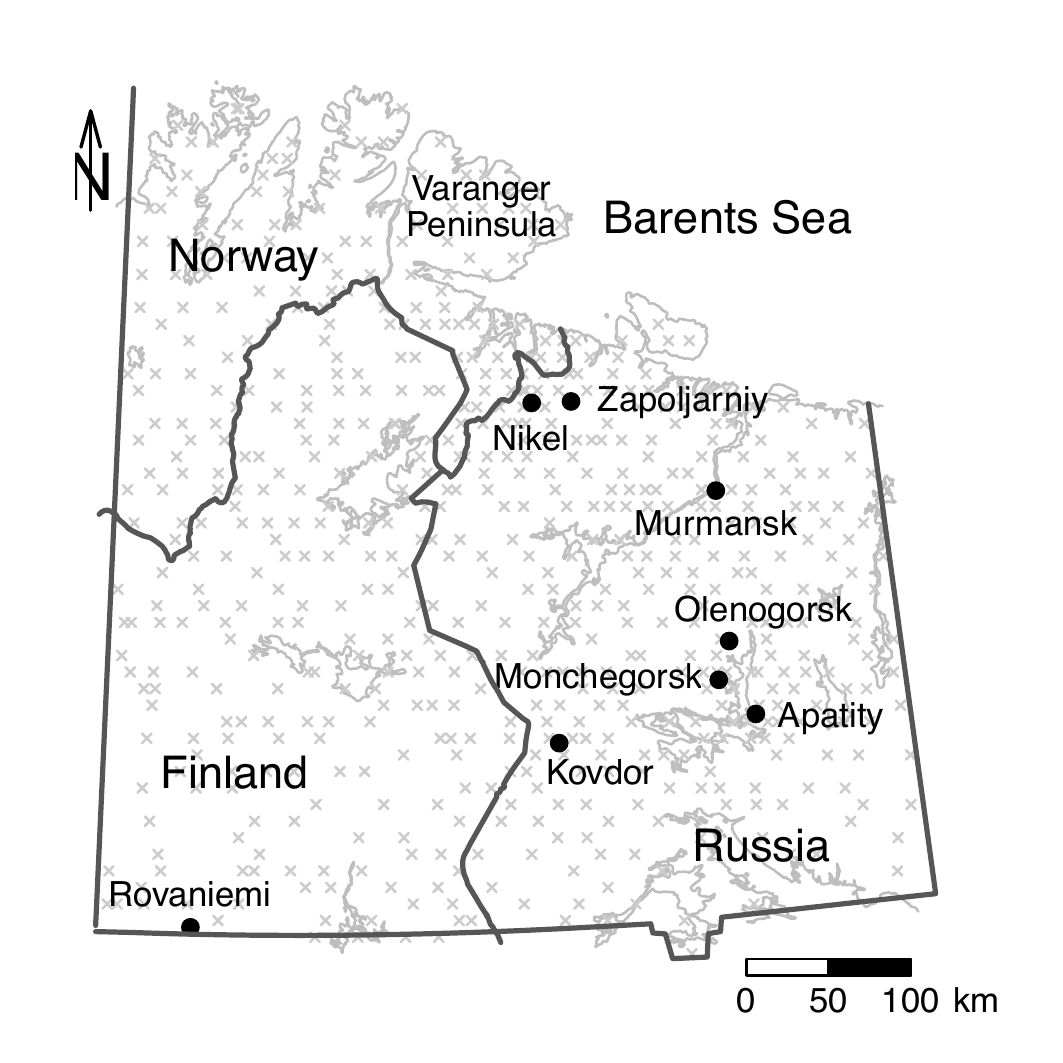}
	\caption{Sampling area of the Kola data. The gray crosses indicate sampling locations. }
	\label{fig:KolaInfo}
\end{figure}

\begin{figure}[t]
	\centering
	\includegraphics[width=0.8\textwidth]{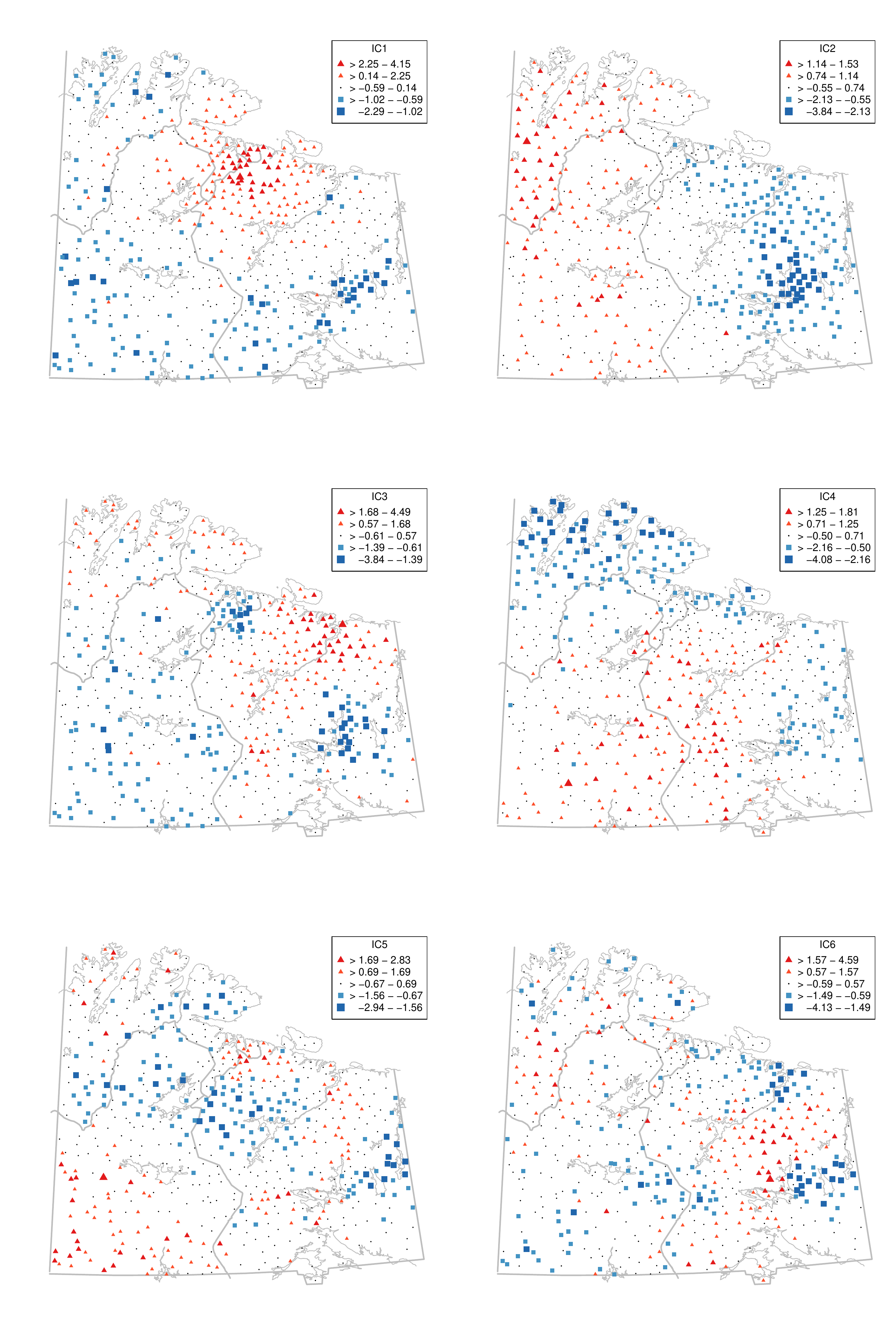}
	\caption{The six most interesting components  using the covariance matrix and $B(50)$. Used as {gold} standard following \citet{NordhausenOjaFilzmoserReimann2015}.}
	\label{fig:B50}
\end{figure}

\begin{figure}[t]
	\centering
	\includegraphics[width=0.8\textwidth]{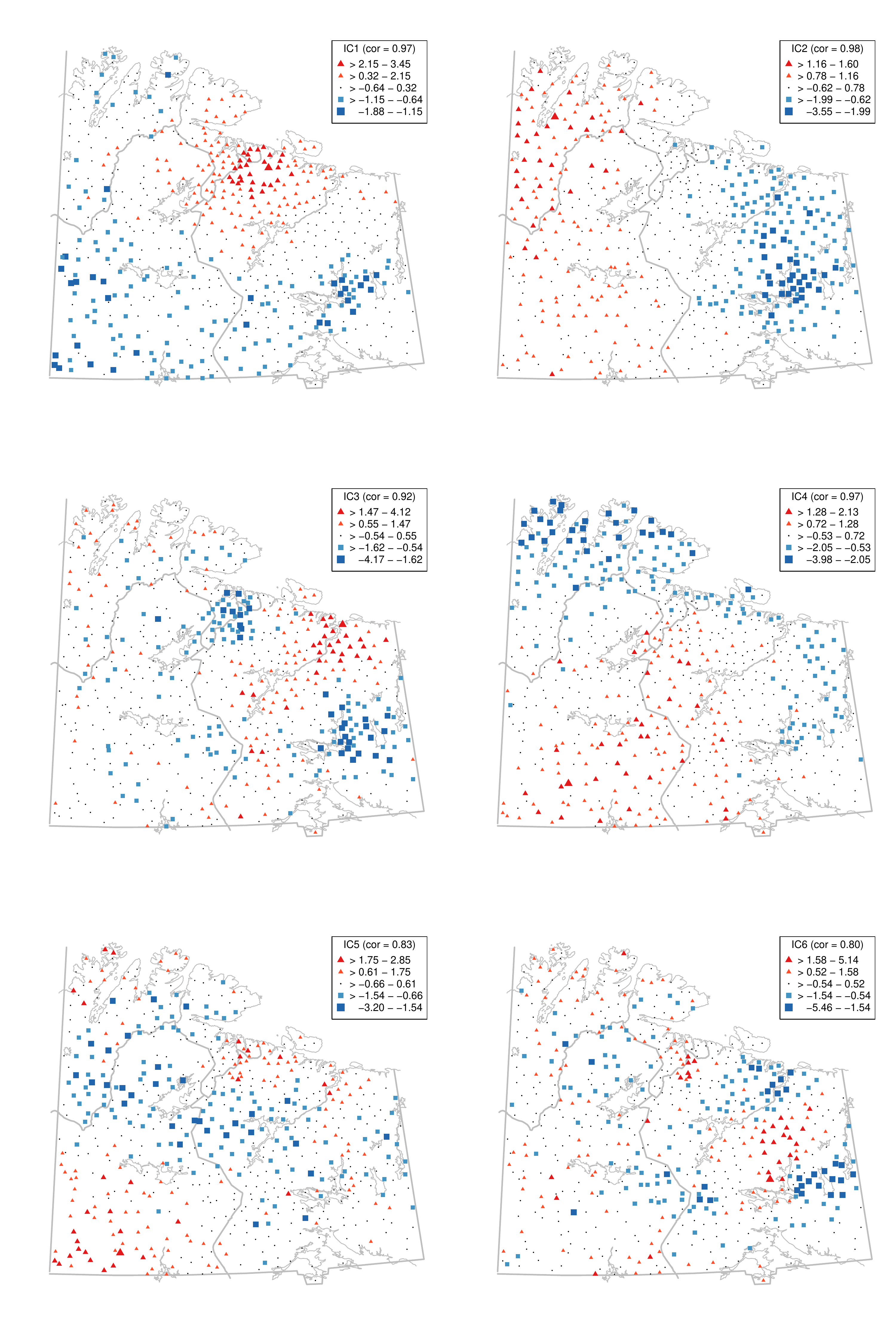}
	\caption{The six components with highest absolute correlations with the components from Figure~\ref{fig:B50} when jointly diagonalising the covariance matrix and the ring kernels $R(0,25)$, $R(25,50)$, $R(50,75)$ , $R(75,100)$. The correlations to the corresponding components of Fig.~\ref{fig:B50} are given in the legends.}
	\label{fig:JD}
\end{figure}

\begin{figure}[t]
	\centering
	\includegraphics[width=0.8\textwidth]{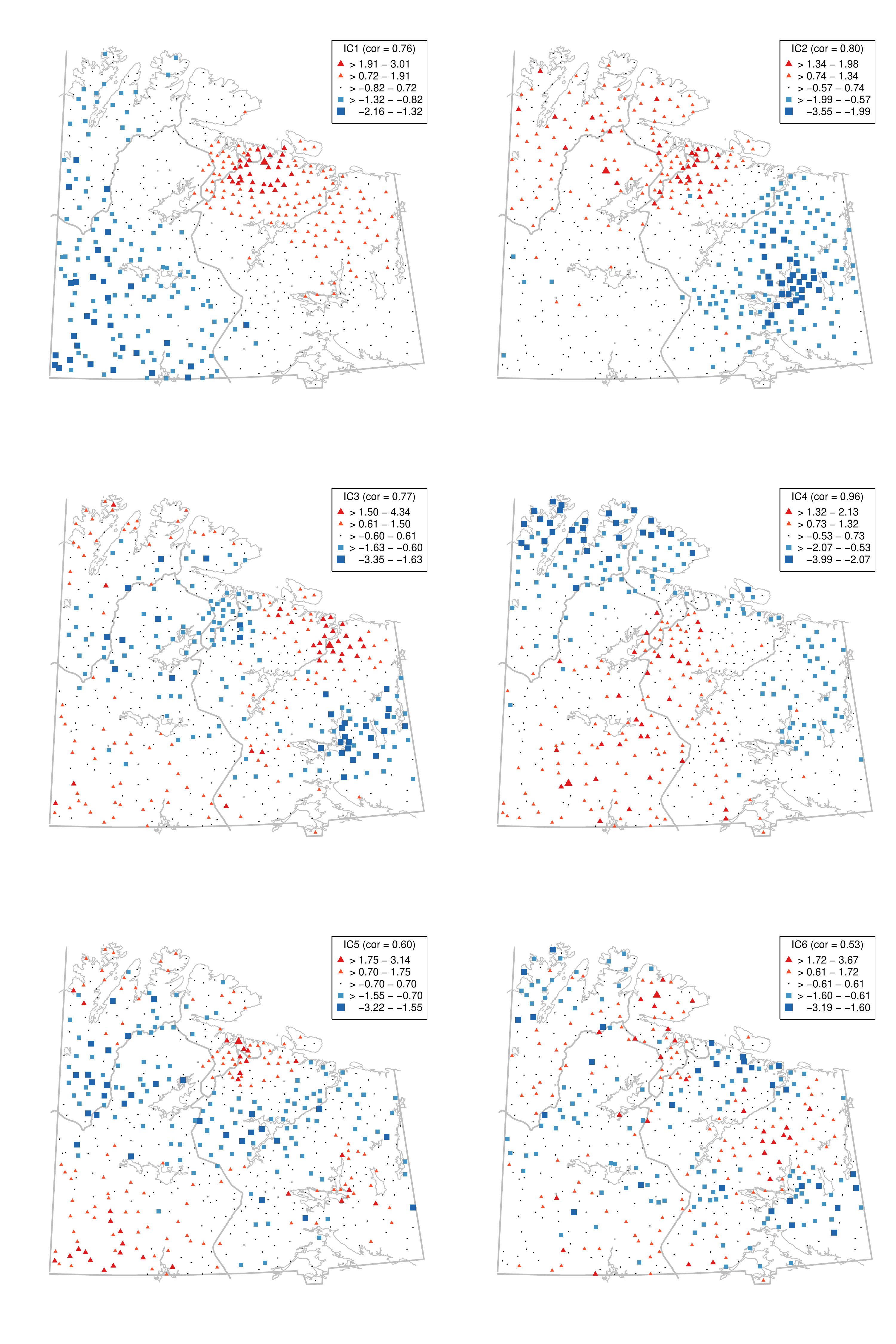}
	\caption{The six components with highest absolute correlations with the components from Fig.~\ref{fig:B50} when jointly diagonalising the covariance matrix and  $B(100)$. The correlations to the corresponding components of Fig.~\ref{fig:B50} are given in the legends.}
	\label{fig:B100}
\end{figure}

\bibliographystyle{biometrika}
\bibliography{SpatialBSS_references}

\end{document}